\documentclass[11pt]{amsart}
\usepackage{amssymb}
\usepackage{latexsym}
\usepackage{amscd} 
\usepackage{amsmath}
\usepackage{epsfig}
\usepackage{pdfpages}
\usepackage{rotating,graphics,psfrag,epsfig}\usepackage{graphicx}
\usepackage{mathrsfs}
\usepackage[margin=3.25cm]{geometry}
\usepackage{ifpdf}
\usepackage{xr}
\usepackage{bm}
\usepackage{amsthm}
\usepackage[all]{xy}%
\usepackage{enumitem}
\newtheorem{theorem}{Theorem}[section]
\newtheorem{introtheorem}{Theorem}
\newtheorem{introcorollary}[introtheorem]{Corollary}
\newtheorem{corollary}[theorem]{Corollary}

\newtheorem{lemma}[theorem]{Lemma}
\newtheorem{definition}[theorem]{Definition}
\newtheorem{proposition}[theorem]{Proposition}
\newtheorem{remark}[theorem]{Remark}

\usepackage{verbatim}
\usepackage{hyperref}

\hypersetup{
backref=true,    
pagebackref=true,
hyperindex=true, 
colorlinks=true, 
breaklinks=true, 
urlcolor= blue,  
linkcolor= blue, 
}

\def\C{\mathbb{C}}

\def\H{{\mathbb{H}}}
\def\R{\mathbb{R}}

\def\Q{\mathbb{Q}}
\def\N{\mathbb{N}}

\def\Z{\mathbb{Z}}
\def\D{\mathbb{D}}

\def\QS{\Q / \Z}
\def\RZ{\R / \Z}
\def\CDC{\C \setminus \overline{\D}}
\def\CP2{{\mathbb{CP}^2}}

\def\poly{{\operatorname{poly}}}

\usepackage{calrsfs}
\DeclareMathAlphabet{\pazocal}{OMS}{zplm}{m}{n}
\def\ccS{{\mathcal{S}}}

\def\ccS{{\mathcal{S}}}

\def\cC{{\pazocal{C}}}

\def\cE{{\pazocal{E}}}

\def\cH{{\pazocal{H}}}
\def\cI{{\pazocal{I}}}

\def\cM{{\pazocal{M}}}

\def\cO{{\pazocal{O}}}

\def\cR{{\pazocal{R}}}
\def\cS{{\pazocal{S}}}

\def\cU{{\pazocal{U}}}

\def\cW{{\pazocal{W}}}

\def\fstar{{f_\star}}

\def\itin{\operatorname{itin}}

\def\mult3{m_3}

\def\sector{Z}

\def\ray{{\operatorname{ray}}}

\newcommand{\numcirc}[1]{\raisebox{.5pt}{\textcircled{\raisebox{-.9pt} {#1}}}}

\numberwithin{equation}{section}

\begin{document}

\title{Irreducibility of periodic curves in cubic polynomial moduli space}

\author{Matthieu Arfeux \and Jan Kiwi}
\address[Matthieu Arfeux]{Pontificia Universidad Cat\'olica de Valpara\'iso.}
\email{matthieu.arfeux@pucv.cl}
\thanks{The first author was partially supported by  ``Proyecto Fondecyt Iniciación \#11170276''.}

\address[Jan Kiwi]{Facultad de Matem\'aticas, Pontificia Universidad Cat\'olica de Chile.}
\email{jkiwi@uc.cl}
\thanks{The second author was partially supported by CONICYT PIA ACT172001" and ``Proyecto Fondecyt \#1160550''}

\date{\today}
\begin{abstract}
  In the moduli space of complex cubic polynomials with a marked critical point,
  given any $p \ge 1$, we prove that the loci formed by polynomials
  with the marked critical point periodic of period $p$ is an irreducible curve. Thus answering a question posed by Milnor in the 90's. 
\end{abstract}

\maketitle

\centerline{\it Dedicated to John Milnor on his 90th birthday.}

\setcounter{tocdepth}{3}
\section{Introduction}
 In their celebrated articles~\cite{BrannerHubbardCubicI,BrannerHubbardCubicII}, Branner and Hubbard 
layed the foundations for the study of complex cubic polynomials as dynamical systems.
These articles shed light on the structural importance of certain dynamically defined 
complex one dimensional slices of
 cubic parameter space~\cite{BrannerMilnorFest60}.
The systematic study of such slices was initiated by Milnor 
in a 1991 preprint later published as~\cite{MilnorPeriodicCubic}.

Following Milnor consider the moduli space $\poly^{cm}_3$ of affine conjugacy classes of cubic polynomials with a marked critical point.
This moduli space is isomorphic to the quotient of $\C^2$
by a linear involution with a unique fixed point.
 For each $p \ge 1$,
 Milnor considered the algebraic subset of $\poly^{cm}_3$
 formed by all polynomials such that the marked critical point  is periodic of exact period $p$, denoted  $\ccS_p$.  
He established that $\ccS_p$ is smooth for all $p \ge 2$.
Therefore, connectedness and irreducibility of $\ccS_p$ are equivalent properties.
Computing explicit parametrizations 
he proved that $\ccS_1, \ccS_2$ and $\ccS_3$ are connected and asked in general: 
\emph{Is $\ccS_p$ connected?}
(\cite[Question~5.3]{MilnorPeriodicCubic}).

Our aim here is to  answer this question in the positive:
\begin{introtheorem}
  \label{thr:main}
  For all $p \ge 1$, the affine algebraic set $\ccS_p$ is irreducible.
\end{introtheorem}

 Our result is similar to one known for  
 {\it quadratic dynatomic curves:}
$$\{ (z,c) \in \C^2 \mid z \mbox{ has exact period } n \mbox{ for } z \mapsto z^2 + c \}.$$ 
Dynatomic curves were shown to be smooth by Douady and Hubbard~\cite{OrsayNotes} and irreducible by Bousch~\cite{BouschThesis}. 
Nowadays several proofs for the irreducibility of these curves are available. The proofs range from those relying more on algebraic
methods like Bousch's original proof (cf. Morton~\cite{MortonDynatomic}) to those more strongly based on dynamical 
techniques by Schleicher-Lau~\cite{SchleicherPhD} (cf. Buff-Tan~\cite{BuffTanDynatomic}). 
Our methods rely on one dimensional parameter space techniques and
bear certain analogy with the work by Schleicher-Lau~\cite{SchleicherPhD}.

In further generality one might consider the algebraic subsets of moduli space determined by critical orbit relations. Their importance is confirmed
by recent results regarding the dynamical analogue of the Andr\'e-Oort conjecture by 
Favre and Gauthier~\cite{FavreGauthierSpecialCubic} and,  Ghioca and Ye~\cite{GhiocaYeSpecialCubic}.
Recently, Buff, Epstein and Koch~\cite{buff2018rational} studied the curves $\ccS_{k,1} \subset \poly^{cm}_3$ formed by polynomials such that the
 marked critical point maps in exactly $k$ iterations onto a fixed point. For any $k \ge 1$ they established that
 $\ccS_{k,1}$ is irreducible. Also they obtained an analogue result in the moduli space of quadratic rational maps.
 Their techniques are related to those employed by Bousch~\cite{BouschThesis} for
 dynatomic curves.

\medskip
There are still plenty of open questions about the global topology of $\ccS_p$. As anticipated by Milnor~\cite[Section 5D]{MilnorPeriodicCubic}
once we know that $\ccS_p$ is connected it becomes meaningful to determine its genus.
A formula for the Euler characteristic of $\ccS_p$ was obtained in~\cite[Theorem~7.2]{AKMCubic}. 
However a direct formula 
for the number of punctures of $\ccS_p$ remains unknown. De Marco and Schiff~\cite{DeMarcoSchiff} gave an algorithm to compute this number based on the work by De Marco and Pilgrim~\cite{DeMarcoPilgrimClassification}. Now it is also natural to ask if the corresponding dynatomic curves over $\ccS_p$ are irreducible. 

There is a rich interplay between the global topology of $\ccS_p$ and
how polynomials are organized within $\ccS_p$ according to dynamics.
Since the curves $\ccS_p$ are complex one dimensional parameter
spaces, they have been natural grounds to employ and further develop
the wealth of ideas available for quadratic polynomials
(cf.~\cite{FaughtThesis,RoeschSone,MilnorPeriodicCubic}).  However, a
relevant novelty here is that $\ccS_p$ has non-trivial global topology
for $p \ge 4$.  The usual dichotomy between connected and disconnected
Julia sets takes place on $\ccS_p$ as well.  The relative
connectedness locus $\cC(\ccS_p)$ is the set formed by all polynomials
in $\ccS_p$ with connected Julia set and its complement, the escape locus
$\cE(\cS_p)$, is a disjoint union of 
disks punctured at infinity
(e.g. ~\cite{AKMCubic}). Hence, as an immediate consequence of our main result
we have the following:

\begin{introcorollary}
\label{cor:main}
   For all $p \ge 1$, the connectedness locus $\cC(\ccS_p)$ is connected. 
 \end{introcorollary}

 The dynamics of quadratic rational maps, in analogy with complex cubic polynomials,  also strongly depends on the behavior of two  critical points and quadratic rational moduli spaces is also a complex surface.
 The  curves analogous to $\ccS_p$  have been under intensive study during the last 25 years as well (e.g. see~\cite{ReesI,ReesII,ReesAsterisque,AspenbergYampolskyS2,TimorinRegluing,FirsovaKahnSelinger,HironakaKoch,ramadas2020quadratic}). Nevertheless, the irreducibility of these curves is still to be determined. 

\section{Overview}
\label{s:overview}


The moduli space of complex cubic polynomials with a marked critical point,
denoted $\poly^{cm}_3 $, is formed by affine conjugacy classes of pairs $(f,a)$ where $f: \C \to \C$ is a cubic polynomial and $f'(a)=0$.
Following Milnor, by passing to a double cover one may avoid the nuances of working with conjugacy classes.
More precisely, for $(a,v) \in \C^2$ let $$f_{a,v} (z) = (z-a)^2 (z+2a) +v.$$
The critical points of $f_{a,v}$ are $a$ and $-a$. We say that $a$ is the \emph{marked critical point} and $-a$ is the
\emph{free critical point}.
Via conjugacy by $z \mapsto -z$, the polynomials $f_{a,v}$ and $f_{-a,-v}$ represent the same element of
$\poly^{cm}_3$. The quotient of $\C^2$ by the involution $(a,v)\mapsto (-a,-v)$ is $\poly^{cm}_3$.

Denote by $\cS_p \subset \C^2$ the affine algebraic curve formed by all  $f_{a,v}$ such that $a$ has period exactly $p$
under $f_{a,v}$. It follows that $\cS_1$ is a double cover of $\ccS_1$ ramified at the monomial map $f_{0,0}(z) =z^3$ and
$\cS_p \to \ccS_p$ is a regular double cover for all $p \ge 2$. Thus, to establish our main result it suffices to
prove that $\cS_p$ is connected for all $p\ge 2$.

For the rest of the paper let us fix $p \ge 2$ and consider the  natural decomposition of $\cS_p$ into the \emph{connectedness locus} $\cC(\cS_p)$ and the \emph{escape locus $\cE(\cS_p)$}. The former  consists of
all $f_{a,v} \in \cS_p$ with connected Julia set. The latter is  formed by maps with
disconnected Julia set or equivalently maps such that free critical point escapes to infinity. 

It will also be convenient to drop the subscripts from $f_{a,v} \in \cS_p$. We will simply say that $f \in \cS_p$
has marked critical point $a(f)$, free critical point $-a(f)$ and, for $j \ge 0$, 
let  $a_j(f) = f^j (a(f))$ be the corresponding periodic point in the  orbit of $a(f)$.

Given a map $f \in \cS_p$
let $G_f$ be the Green function of the basin of infinity. For all
$f \in \cE(\cS_p)$,
$$\{ z : G_f(z) < G_f(-a(f)) \}$$
is the union of two disjoint topological disks $D_0$ and $D_1$.
We label
these disks so that $a(f) \in D_0$. Although $D_0$ and $D_1$ depend on $f$
it will cause no confusion to employ a notation that ignores this fact. 
The filled Julia set $K(f)$ is contained in
$D_0 \cup D_1$, in particular, the period $p$ critical orbit points $a_0(f), a_1(f), \dots, a_{p-1}(f)$ lie either in $D_0$ or $D_1$. 

The \emph{kneading word} $\kappa(f)$ of $f$ is the binary sequence
$$\kappa(f)=\iota_1 \dots \iota_{p-1} \iota_p$$ where $\iota_j = 0$ or $1$ according to whether $a_j(f) \in D_0$ or $D_1$. Note that $\iota_p=0$.

Connected components of the escape locus are called \emph{escape regions}. They will play a central role in our work. From~\cite[\S 5B]{MilnorPeriodicCubic} we may extract the following fundamental facts. Escape regions are in one-to-one correspondence with the ends of $\cS_p$. Maps within an escape region $\cU$
share the same {kneading word} which we denote by $\kappa(\cU)$.
Different escape regions can have the same kneading word. However there is exactly one region $\cU$, called
the \emph{distinguished escape region}, such that $\kappa(\cU) =1^{p-1}0$. 
We say that $1^{p-1}0$ is the \emph{distinguished kneading word}. 

Every component of the algebraic set $\cS_p$ contains at least one escape region. Given an undistinguished escape region $\cU$ with kneading word $\kappa$ our strategy consists on constructing a path in $\cS_p$ joining $\cU$ with another escape region $\cU'$ such that, in a certain sense,
the kneading word $\kappa' = \kappa(\cU')$
is closer to the distinguished kneading word (cf.~\cite{ArfeuxKiwiDCDS}).
The kneading word $\kappa'$ will be obtained from $\kappa$ by type $A$ or $B$ move described below.

To describe type $A$ and $B$ moves, for a kneading word
$\kappa \neq 0^p$, let $\mu$ be such that $1^{\mu-1}$ is the largest string of $1$'s that one encounters in $\kappa$; for  $\kappa = 0^p$, let $\mu=1$. If $f \in \cE(\cS_p)$ has kneading word $\kappa$, then $\mu$ is the maximal time that takes an element of the periodic critical orbit in $D_0$
to return to $D_0$.
We say that $\mu$ is the \emph{maximal return time of $\kappa$}.

Let $\kappa=\iota_1 \dots \iota_{p-1} 0$ be an undistinguished word
with maximal return time~$\mu$:  
 \begin{itemize}
\item[(A)]   We say that $\kappa'$ is obtained from $\kappa$ by a \emph{type $A$ move} if:
  $$\kappa = 1^{\mu-1}0 \iota_{\mu+1} \cdots \iota_{p-1} 0$$
  and
  $$\kappa = 0^{\mu-1}1 \iota'_{\mu+1} \cdots \iota'_{p-1} 0$$
  for some $\iota'_{\mu+1}, \dots, \iota'_{p-1} \in \{0,1\}$.
\item[(B)] We say that $\kappa'$ is obtained from $\kappa$ by a \emph{type $B$ move} if, for some $1 \le k \le p-1$:
$$\kappa = \iota_1 \cdots \iota_{k-1} 0 \cdot 1^{\mu-1} \iota_{k+\mu} \cdots \iota_{p-1} 0$$
and
$$\kappa' = \iota_1 \cdots \iota_{k-1} 1 \cdot 1^{\mu-1} \iota_{k+\mu} \cdots \iota_{p-1} 0.$$
  \end{itemize}
  That is, the type $A$ move changes an initial string of $\kappa$ in a specific manner but there is no restriction on how the tail changes.
  The type $B$ just changes the
  $k$-th symbol from a ``$0$'' to a ``$1$'' and leaves the rest of the symbols untouched. The former can only occur if the maximal return time
  is ``attained'' at the position $k=0$ and the latter if it is ``attained'' at a position $k \neq 0$ in the kneading word $\kappa$.
  
  It is not difficult to show (see~Lemma~\ref{lem:typeAB}) that no matter how we successively apply type $A$ and $B$ moves in at most $p^2+p$ steps we arrive to the distinguished word.
  Type $A$ and $B$ moves will be associated to
  a special way of crossing certain type of hyperbolic components in $\cC(\cS_p)$ known as type $A$ and $B$ components.

\smallskip
Given an escape region $\cU$ our path within $\cS_p$ starts at a map with a \emph{ray connection}. For a general background on external rays we refer to \S~\ref{s:external-rays}.
Denote by $V(a_k(f))$ the Fatou component of $f \in \cU$ containing $a_k(f)$. 
  We say that a non-smooth ray $R^\sigma_f(\vartheta)$ where $\sigma=+$ or $-$
  is  a \emph{ray connection between}
$-a(f)$ and $a_k(f)$ if the following hold:
  \begin{itemize}
  \item $3 \vartheta$ is periodic of period $p$ under multiplication
    by $3$ in $\RZ$.
    \item $-a(f) \in R_f^\sigma(\vartheta)$.
    \item $R_f^\sigma(\vartheta)$ is a relatively (left or right) supporting ray of $V(a_k(f))$.
    \end{itemize}
    A detailed discussion about ray connections is given in the introduction
    to \S~\ref{s:ray-connection} where the notion of relatively supporting rays is introduced as a generalization of the usual notion of supporting rays.
    Then we prove that for any undistinguished escape region $\cU$ there exists a map $f \in \cU$ which has a ray connection between $-a(f)$ and $a_k(f)$ for some
  $a_k(f) \in D_0$ such that $a_k(f)$ has maximal return time to $D_0$ (Theorem~\ref{thr:ray-connection}).

\medskip
Once we have a map $f$ with an appropriate ray connection we head towards the connectedness locus along the parameter ray
$\cR_\cU(\theta)$ containing $f$.
The landing behavior of  rational parameter rays $\cR_\cU(\theta)$
is studied by Bonifant, Milnor and Sutherland
in~\cite{cm3,ParabolicGreen}.
These rays land at parabolic or postcritically finite (pcf) maps.
For our purpose we need to describe in greater detail the
supporting properties of external rays and to keep track of kneading words
at the landing point of these parameter rays.
To properly state our description, given a parabolic or pcf map $f_0$, we introduce the notions of a \emph{take-off argument $\theta$} of $f_0$ and its \emph{associated    kneading word $\kappa(f,\theta)$}, see~definitions~\ref{def:take-off} and \ref{def:take-off-word}. 
Rays landing at parabolic maps are the subject of
\S~\ref{s:parabolic-landing} while rays landing at pcf maps are considered in
\S~\ref{s:pcf-landing}. Special care is taken to describe the maps which are the landing point of parameter rays with ray connections. 
 ``Take-off'' results are also relevant to our discussion.
Namely, given any  parabolic or pcf map $f_0$
with take-off argument $\theta$ we establish that
there exists an escape region $\cU$ such that $\kappa(\cU)= \kappa(f,\theta)$
and 
a parameter ray $\cR_\cU(\theta)$ of $\cU$ that lands at $f_0$.
This follows directly from  Bonifant and Milnor's work~\cite{cm3} 
for $f_0$ pcf and from
Theorem~\ref{thr:take-off-parabolic} for $f_0$ parabolic.

\smallskip
According to Milnor~\cite{MilnorPeriodicCubic}, bounded hyperbolic components of $\cC(\cS_p)$ fall into four types: adjacent (A), bitransitive (B), capture (C), and disjoint (D).
Our main interest will be on types $A$ and $B$. A hyperbolic map $f \in \cC(\cS_p)$ is of type $A$ if the free critical point $-a(f)$ lies in
the Fatou component $V(a(f))$.  The map $f$ is of type $B$ if
$-a(f) \in V(a_k(f))$ for some $0 < k <p$. For completeness let us just mention that $f$ is of type $C$ if $-a(f)$ lies in a strictly preperiodic Fatou component and of type $D$ if $-a(f)$ and $a(f)$ lies
in disjoint cycles of Fatou components.

In \S~\ref{s:trekking} we start by discussing parameter internal rays and sectors
of type $A$ and $B$ hyperbolic components (see~\S~\ref{s:bounded-hyperbolic} and~\S~\ref{s:sectorAB}). The key for the rest of the results in \S~\ref{s:trekking}
is contained in~\ref{s:landing-internal-parameter} where
we study the landing points of $0$ and $1/2$ parameter internal rays.
For us, a ``root'' of a type $A$ or $B$ component will be the landing point
of a $0$-parameter internal ray and a ``co-root'' will be the landing point of a $1/2$-parameter internal ray of a type $A$ component.
In~\S~\ref{s:trekkingB}
and ~\S~\ref{s:trekkingA}
given a ``root'' or a ``co-root'' of 
a type $A$ and $B$ hyperbolic component we compute the effect
that crossing (trekking) to another  ``root'' or ``co-root'' has
in their take-off kneadings.
 In~\S~\ref{s:parabolic-immigration} and \S~\ref{s:pcf-immigration},
 we certify that the landing points of parameter rays with ray connections are in the boundary of
a type $A$ or $B$ hyperbolic component. Loosely speaking, we are granted entry to such a hyperbolic component and say that these are ``immigration'' results.
More precisely, we will show that in the presence of a periodic ray connection the landing point is a ``root'' of a type $A$ or $B$ component, see~\S~\ref{s:parabolic-immigration}.  In the presence of a preperiodic ray connection the landing point will be a ``co-root'' of a type $A$ component, see~\S~\ref{s:pcf-immigration}.

In \S~\ref{s:final} we assemble the proof of our main result.
Given an escape region $\cU$ with kneading word $\kappa$
we consider  the landing point $f_0$ of a parameter ray
$\cR_\cU(\theta)$ with a ray connection between $-a(f)$ and $a_k(f)$ where $a_k(f)$ has maximal return time to $D_0$. The immigration results
guarantee that $f_0$ is in the boundary of a type $A$ or $B$ hyperbolic component $\cH$. The trekking theorems  apply
 to find in $\partial \cH$ a parabolic or pcf map $f_1$ with a take-off argument $\vartheta$ so that the corresponding kneading $\kappa(f_1, \vartheta)$
 is obtained from $\kappa(f_0, \theta)=\kappa$
 by a type $A$ or $B$ move, according to
 whether $k=0$ or $k \neq0$. Finally, the
 take-off theorems yield
 an escape region $\cU'$ with kneading word $\kappa'=\kappa(f_1, \vartheta)$ and a parameter ray $\cR_{\cU'}(\vartheta)$ landing at $f_1$.

 In Figure~\ref{HypTypeB} we label parameters  \numcirc{1} through \numcirc{6} along a path crossing a type $B$ component. The path
 starts at \numcirc{1} which corresponds to a map in $\cR_\cU(\theta)$.
 The parabolic maps $f_0$ and $f_1$ above correspond to
 \numcirc{2} and \numcirc{5}, respectively. The dynamical planes along this path are illustrated in Figure~\ref{trekingB}.

\begin{figure}[h]
\begin{center}
\fbox{\includegraphics[width=10cm]{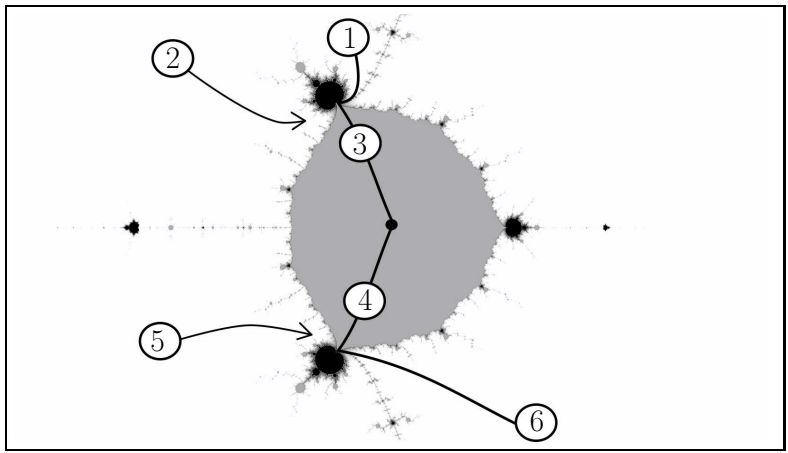} }
\end{center}
\caption{Illustration of a path across a type $B$
  hyperbolic component used to prove Theorem \ref{thr:main}. For the dynamical planes of parameters \raisebox{.5pt}{\textcircled{\raisebox{-.9pt} {1}}} to \raisebox{.5pt}{\textcircled{\raisebox{-.9pt} {6}}} see Figure \ref{trekingB}.   
\label{HypTypeB}}
\end{figure}

\begin{figure}[h]
\begin{minipage}[c]{.30\linewidth}
\begin{center}
\fbox{\includegraphics[width=4.1cm]{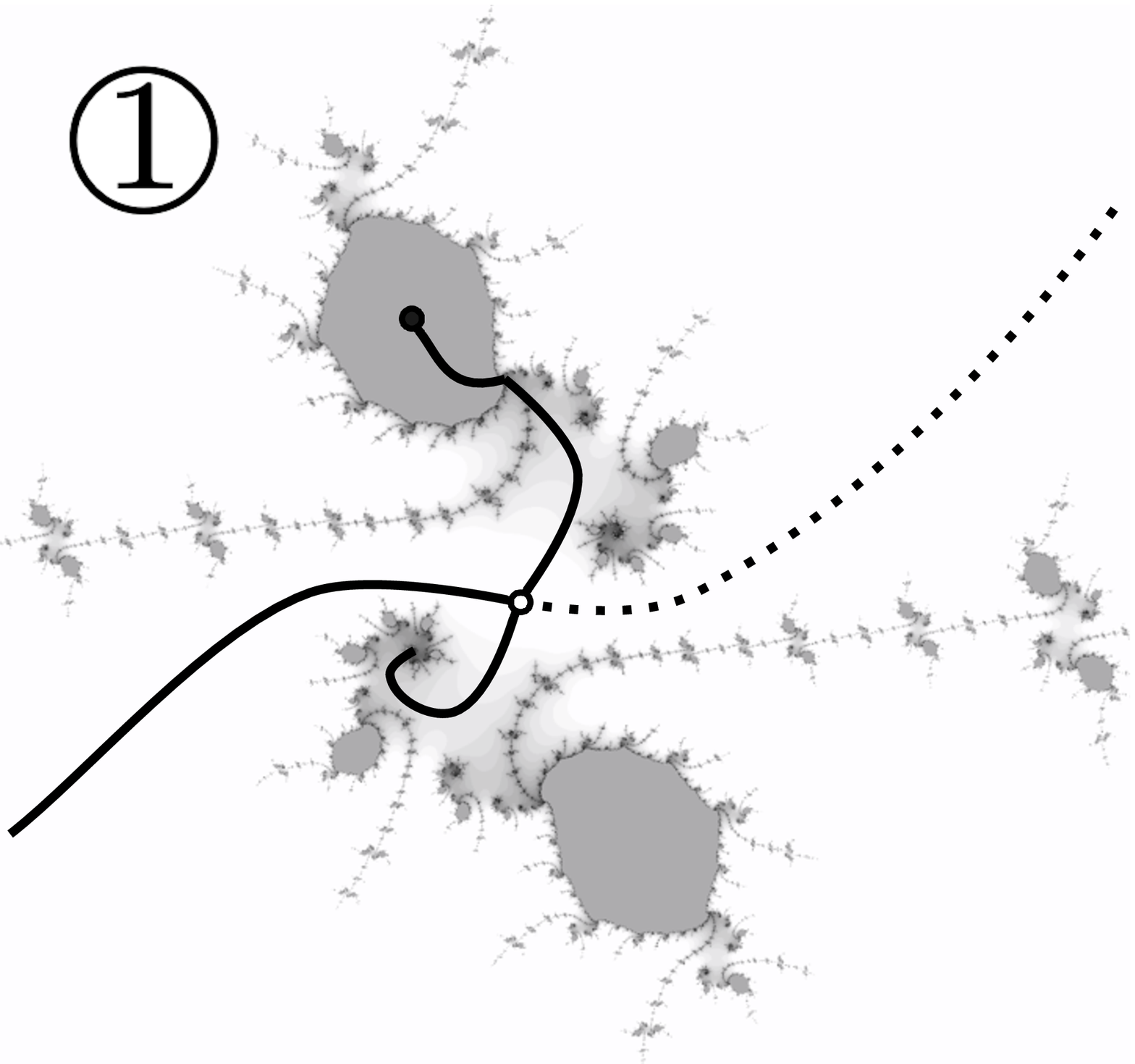}} 
\end{center}
\end{minipage}
\hfill
\begin{minipage}[c]{.30\linewidth}
\begin{center}
\fbox{\includegraphics[width=4cm]{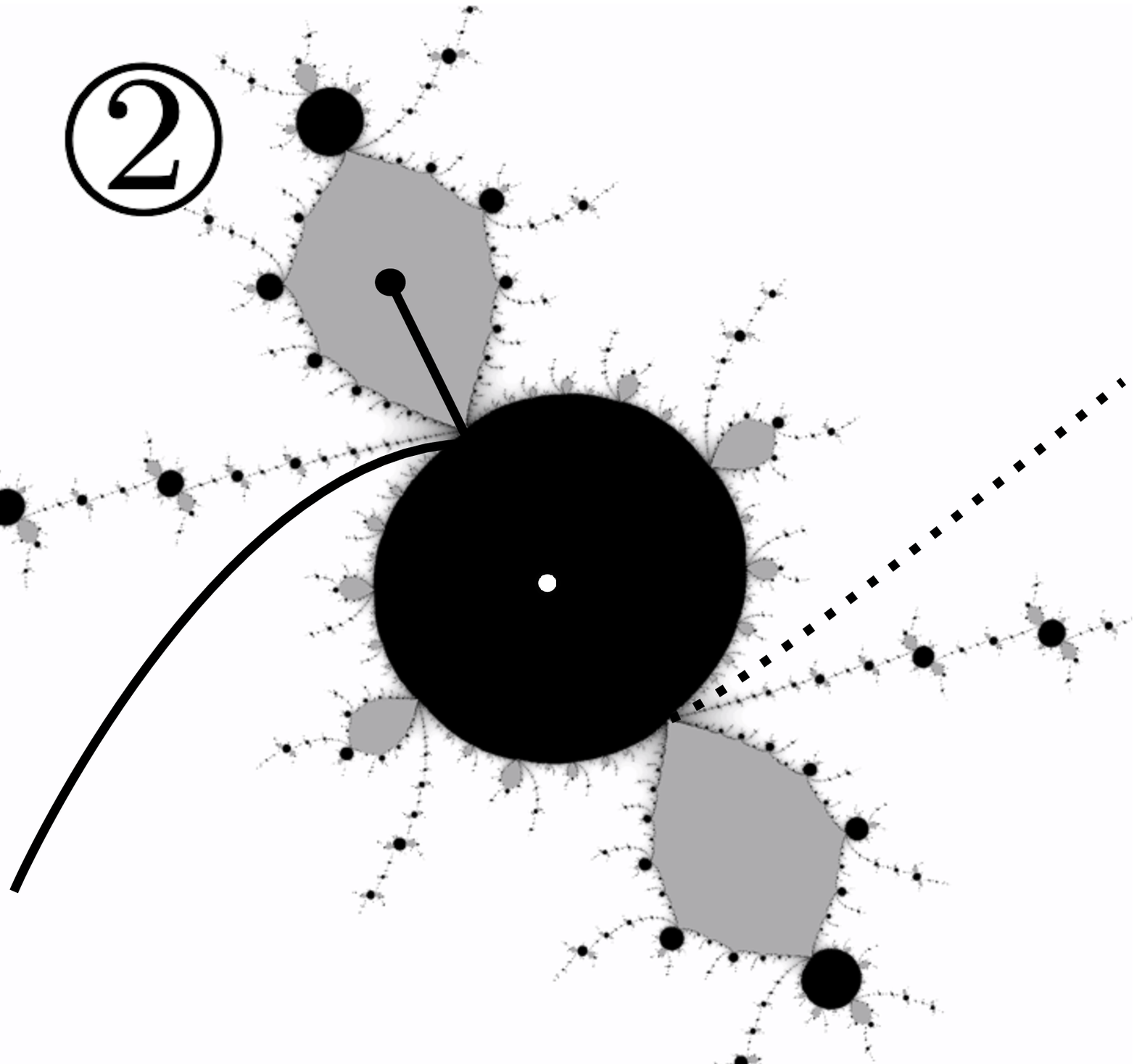} }
\end{center}
\end{minipage}
\hfill
\begin{minipage}[c]{.3\linewidth}
\begin{center}
\fbox{\includegraphics[width=4cm]{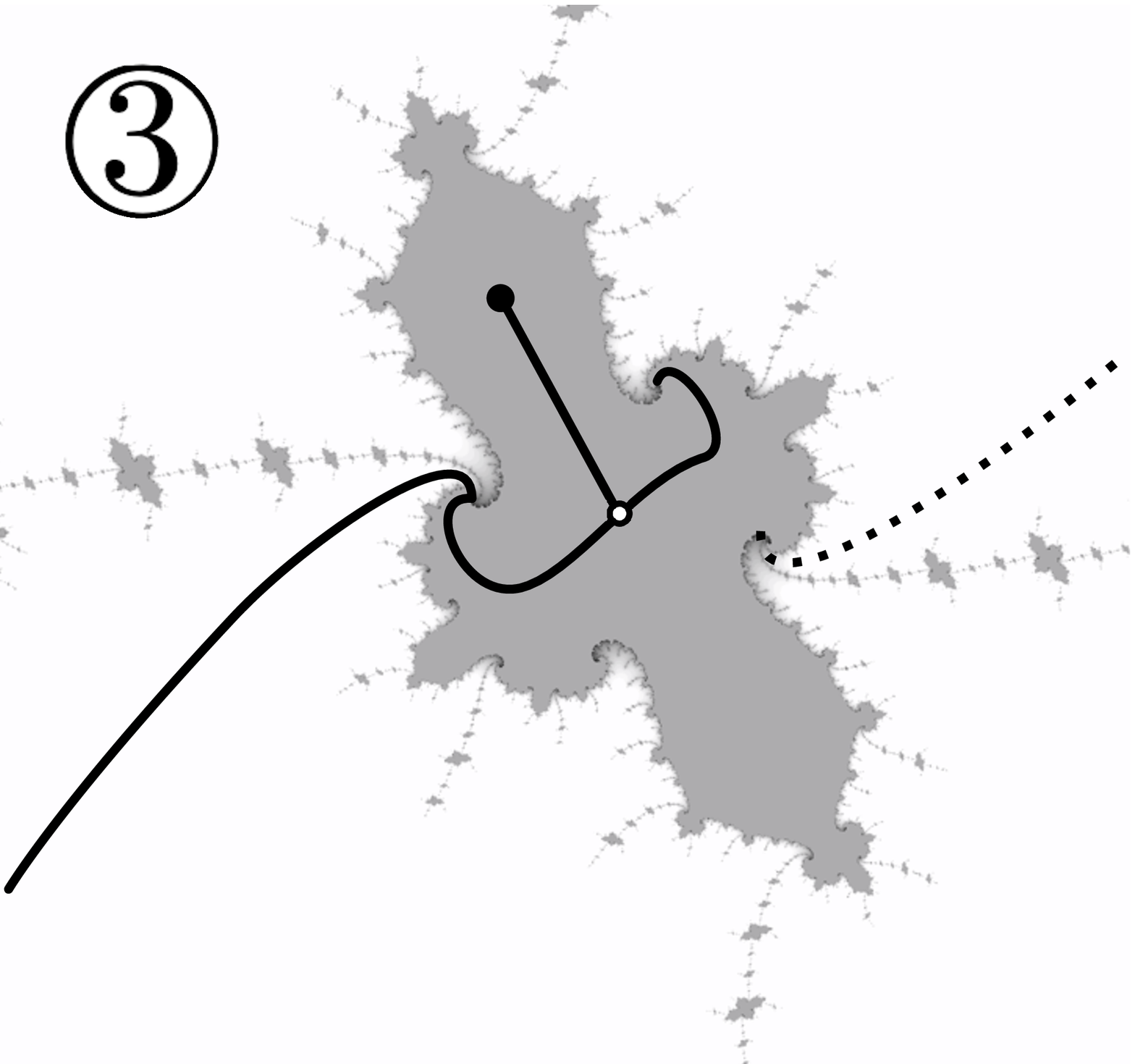} }
\end{center}
\end{minipage}
\begin{minipage}[c]{.3\linewidth}
\begin{center}
\fbox{\includegraphics[width=4cm]{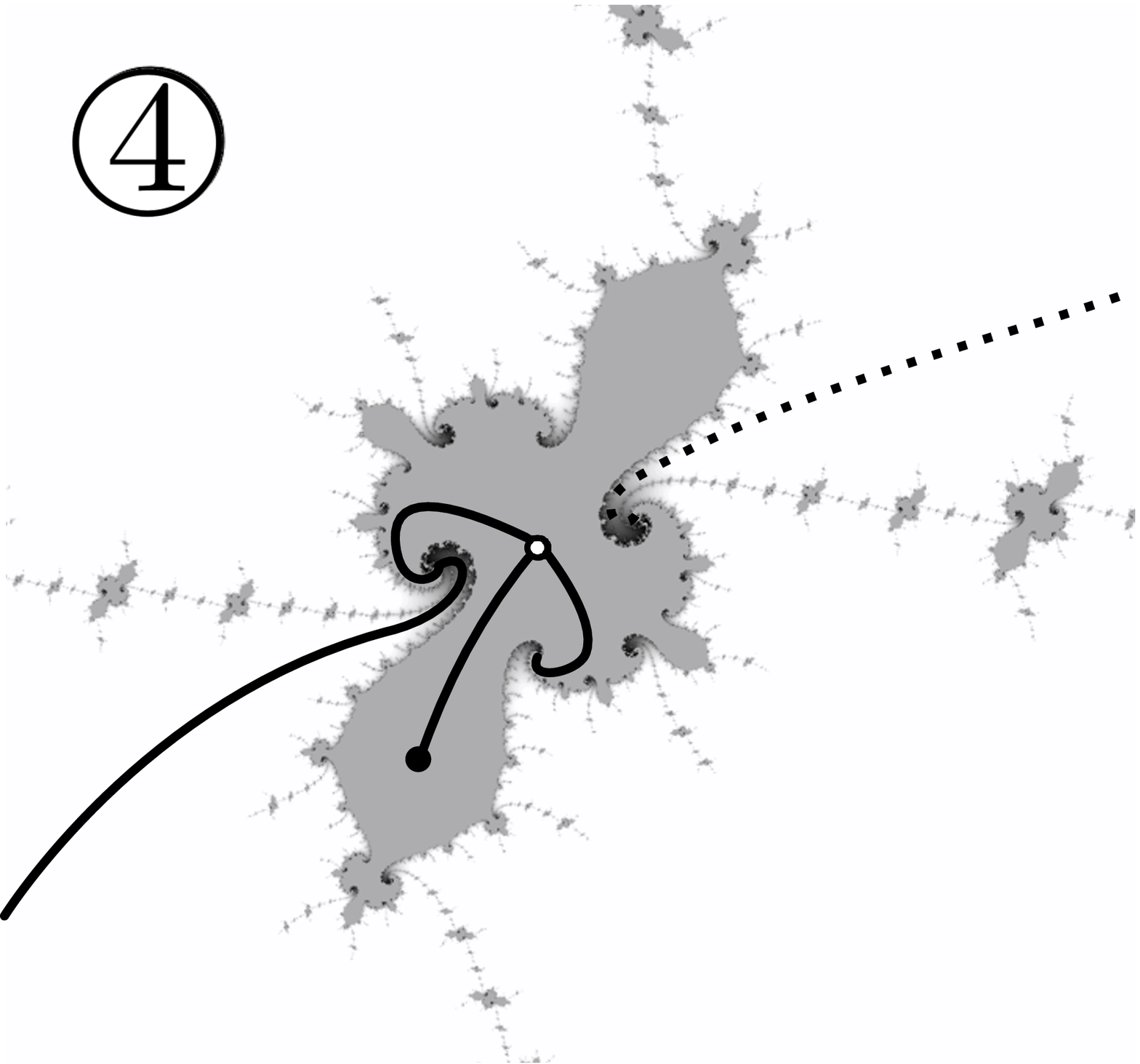} }
\end{center}
\end{minipage}
\hfill
\begin{minipage}[c]{.3\linewidth}
\begin{center}
\fbox{\includegraphics[width=4cm]{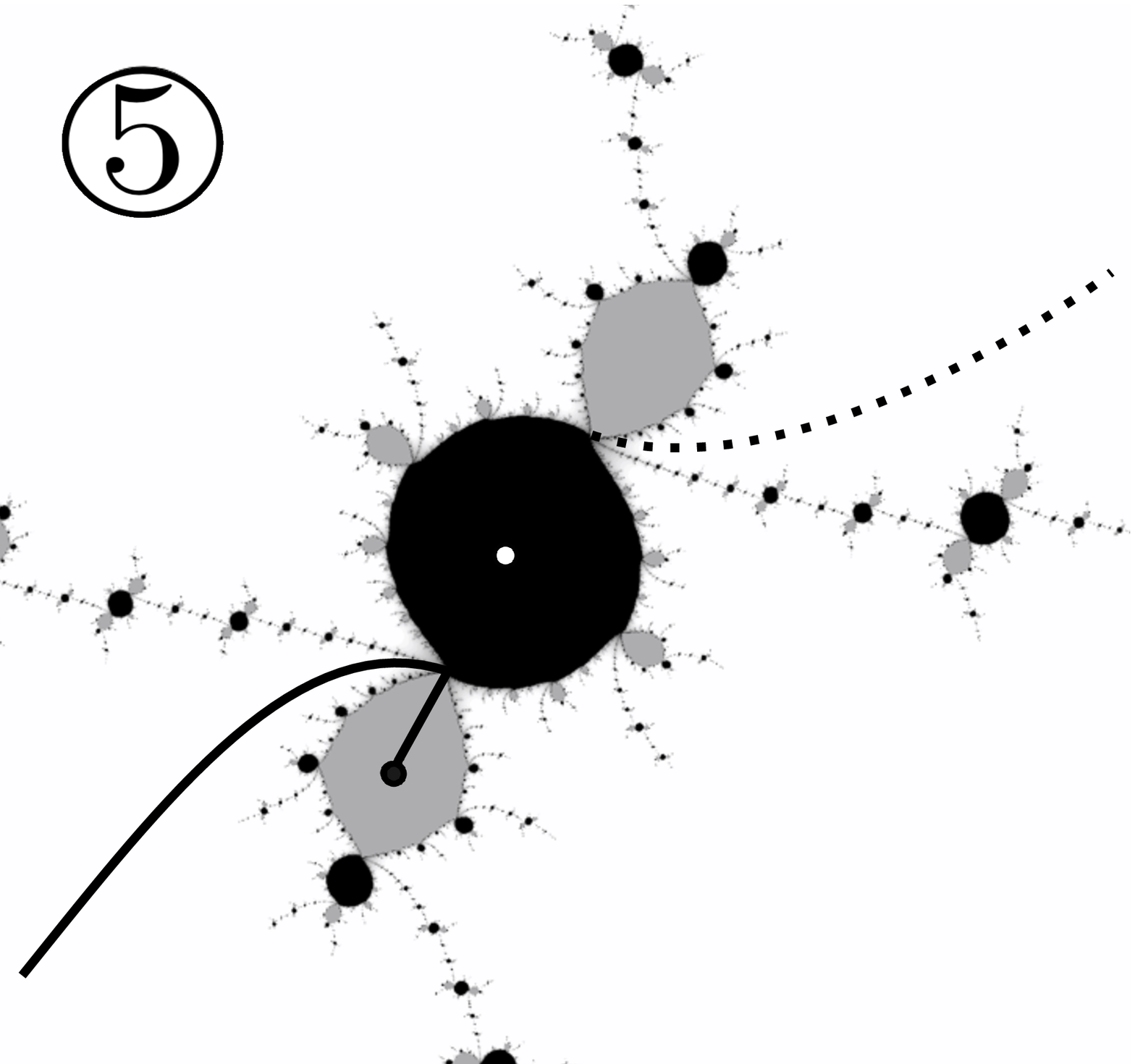} }
\end{center}
\end{minipage}
\hfill
\begin{minipage}[c]{.3\linewidth}
\begin{center}
\fbox{\includegraphics[width=4cm]{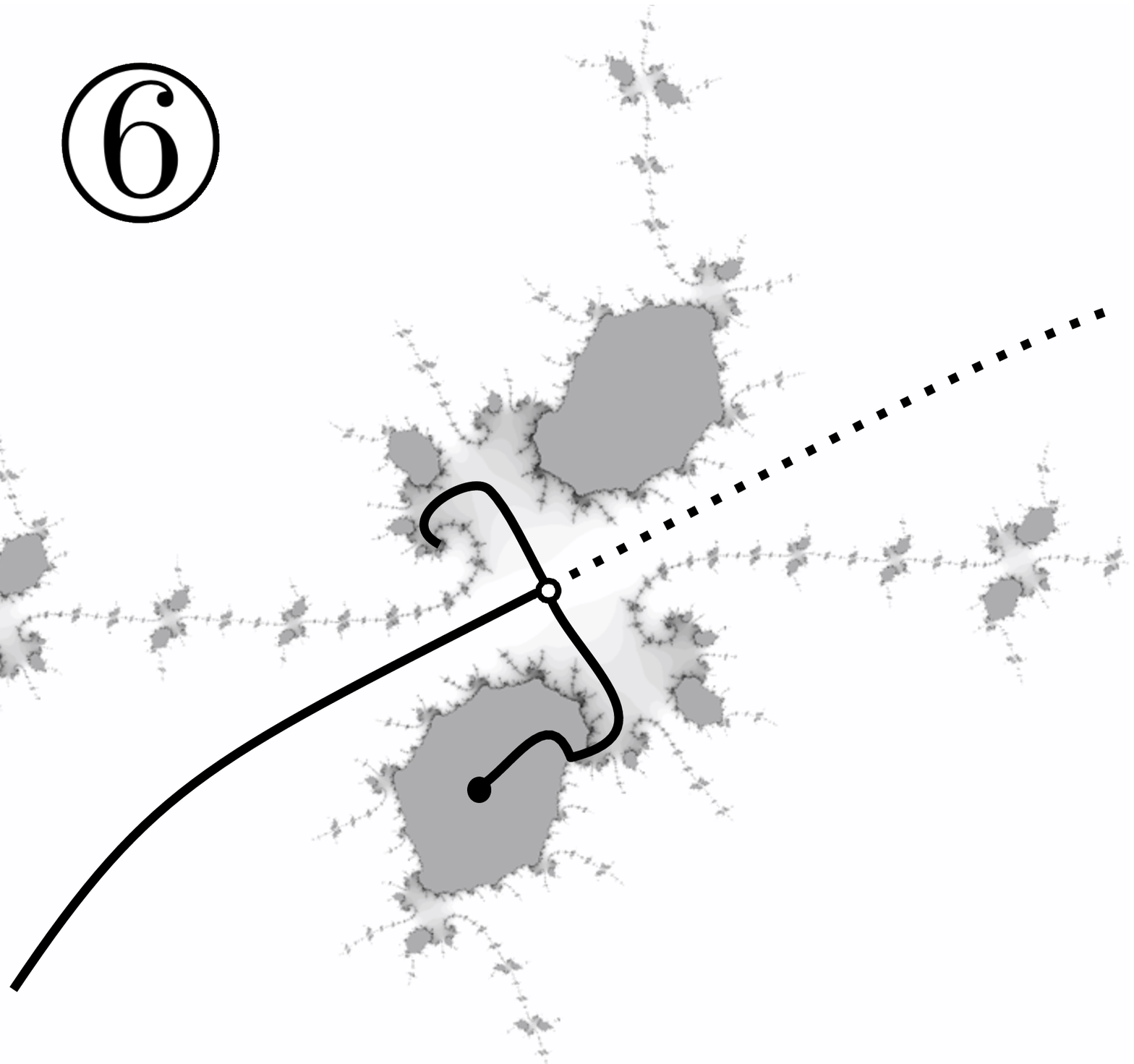} }
\end{center}
\end{minipage}
\caption{Details of the dynamical planes of parameters labeled on the path in  Figure \ref{HypTypeB}, starting at  \numcirc{1} with a ray connection  between $-a$ and $a_k$ 
  and ending at \numcirc{6} where the kneading symbol corresponding to $a_k$ changes.
 Labels for $-a$ and $a_k$ are unfilled and filled dots, respectively.
   \label{trekingB} }
  \end{figure}

  \begin{remark} \emph{The figures contained in the paper are details of $\cS_3$ or Julia sets of this parameter space generated by FractalStream using the parametrization provided in \cite{MilnorPeriodicCubic}. The parameter and dynamical rays are handmade crude approximations. In parameter space illustrations, escape regions are colored white, type $A$, $B$, $C$ hyperbolic components 
are in grey, and type $D$ in black. In dynamical plane illustrations, the basin of infinity is colored white, the basin of the marked critical periodic point $a(f)$ is in grey and, when existing, other Fatou components are in black.}
\end{remark}

\section{Rays and Sectors}
For a general background on
complex polynomial dynamics we refer to~\cite{DynInOne}. In~\S~\ref{s:external-rays} we concentrate on external rays of disconnected Julia sets.
External rays are often employed
to subdivide the complex plane into regions for which
it is convenient to introduce a notation explained in~\S~\ref{s:sector-notation}.
Parameter external rays are the subject of
~\S~\ref{s:parameter-rays}. Internal rays are discussed in~\S~\ref{s:internal-rays}. 

\subsection{External rays}
\label{s:external-rays}
We will intensively exploit external rays for maps with disconnected Julia sets. This brief section aims at recalling notations, definitions and
summarizing results that apply to the context of disconnected Julia sets in $\cS_p$ (see~\cite{GoldbergMilnor,LevinPrzytyckiExternal,LevinSodin}). 

 Given $f \in \cS_p$, the B\"ottcher coordinate
 $\phi_f$, initially defined near $\infty$~(e.g. see \cite[Section~9]{DynInOne}), is the unique conformal isomorphism asymptotic to the identity at infinity ($\phi_f(z)/z \to 1$) such that $(\phi_f(z))^3 = \phi_f (f (z)).$
There exists a unique continuous extension of $|\phi_f(z)|$ 
 to the basin of infinity $V_f(\infty)$ respecting the functional relation
$|\phi_f(z)|^3 = |\phi_f (f (z))|$.

 Denote by $V_f^*(\infty)$ the basin of infinity under the flow of the vector field  $\nabla|\phi_f|$.
 The B\"ottcher coordinate extends along flow lines of  $\nabla |\phi_f|$
 to $V_f^*(\infty)$ (cf.~\cite{LevinSodin}).
 The singularities $z$ of the vector field (i.e. $\nabla |\phi_f(z)| =0$)
 coincide with the points $z$ which eventually map onto a critical point (in our case onto $-a(f)$).

 If $f \in \cC(\cS_p)$, then $V_f^*(\infty)=V_f(\infty)$,
$\phi_f : V_f(\infty) \to \CDC$ is a conformal isomorphism and
$R_f(t) = \phi_f^{-1} (]1,\infty[ \exp(2 \pi i t)$ is called the \emph{external ray of $f$ at argument $t \in \RZ$}. 

When $f \in \cE(\cS_p)$,  for each $t \in \RZ$ there exists a unique
maximal $\nabla |\phi_f|$-flow line, denoted  $R^*_f(t)$,
and a number $\rho(t) \ge 1$ such that
$$\phi_f : R^*_f(t) \to  ]\rho(t), \infty[ \cdot \exp(2 \pi i t) $$
is a homeomorphism. If $\rho(t) > 1$, then the limit in the negative flow
direction along $R^*_f(t)$ is a singularity $w$ of the flow (i.e.
an iterated preimage of the critical point $-a(f)$).
We say that $R^*_f(t)$ \emph{terminates} at $w$.
There is exactly one line $R^*_f(\theta)$ that contains the cocritical point $2a(f)$ (i.e. $f(2a(f)) = f(-a(f))$).
Then  $R^*_f(\theta + 1/3)$ and $R^*_f(\theta - 1/3)$ are the
flow lines terminating at $-a(f)$. These flow lines together with $-a(f)$ cut the complex plane into two regions.
The one containing the periodic critical point $a(f)$ also contains the disk $D_0$ as well as $R^*_f(t)$ for all  $t \in ]\theta-1/3, \theta+1/3[$.
The critical point free region contains the disk $D_1$ as well as $R^*_f(t)$ for all  $t \in ]\theta+1/3, \theta-1/3[$.
Observe that $\rho(t) = 1$ if and only if
$3^nt \neq \theta \pm 1/3$ for all $n \ge 0$ and, in this case, we say that
$R_f(t) = \phi_f^{-1} (]1,\infty[ \exp(2 \pi i t))$ is the \emph{external ray of $f$ at argument $t \in \RZ$}.
Although  the external ray $R_f(t)$ does not exist for arguments $t$ such that $\rho(t) >1$, we may consider the \emph{left} and \emph{right} limit rays at $t$, denoted $R^-_f(t)$ and $R^+_f (t)$, respectively. They are defined as the limit of $R_f(s)$ when $s \to t$
from the left or from the right, as arcs parametrized by $|\phi_f|$
(see~Goldberg and Milnor~\cite[Appendix A]{GoldbergMilnor}). Left and right limit rays with arguments $t$ such that $\rho(t)>1$ are usually called \emph{singular, bouncing, broken or non-smooth}.
We adopt the convention that $R^\pm_f(t)$  simply denotes the smooth ray $R_f(t)$ when $\rho(t)=1$. With this convention $f$ maps
$R^\pm_f(t)$ onto $R^\pm_f(3t)$. It follows that every non-smooth ray
eventually maps onto $R^\pm_f(\theta \pm 1/3)$.

In this paper, multiplication by $3$ acting in $\RZ$ will be denoted by $m_3: \RZ \to \RZ$. Similarly, $m_d$ will denote multiplication by $d$ for  $d \ge 2$.

In $\cS_p$ a map $f$ with disconnected Julia set  $J(f)$
is hyperbolic.
Thus, every ray lands (i.e. has a well defined limit as it approaches the Julia set). Keep in mind that a repelling periodic point $z$ can be the landing point of infinitely many rays if  $\{z \}$ is a connected component of $J(f)$; otherwise, $z$ is the landing point of finitely many periodic rays
(some of which might be left or right limit rays)~\cite{LevinPrzytyckiExternal}.

For any map $f \in \cS_p$, according to Roesch and Yin~\cite{RoeschYin} we have that $\partial V(a(f))$ is a Jordan curve. It follows that $f^p : \partial V(a(f)) \to \partial V(a(f))$ is conjugate to multiplication by $m_d$ for $d=2,3$ or $4$. Thus, $\partial V(a(f))$ contains at least one periodic point of period
dividing $p$. The external rays landing at any one of these points have period $p$ (e.g. see~Proposition~\ref{dynamics-typeAB}).

\subsection{Regions and  sectors}
\label{s:sector-notation}
Given  $f \in \cS_p$, a connected and full subset $X$ of $K(f)$ and
two rays $R_f(s), R_f(t)$ landing at $X$ we denote by
$\sector_f(X,s,t)$ the connected component of
$$\C \setminus (R_f(s) \cup R_f(t) \cup X)$$
containing the external rays with arguments in $]s,t[ \subset \RZ$.
When $X$ is a singleton $\{ z \}$ we simply write $\sector_f(z,s,t)$.
In this case, we say that $\sector_f(z,s,t)$ is a \emph{sector at $z$} if no ray
with argument in $]s,t[$ lands at $z$. 

We freely employ this notation in the presence of non-smooth rays, for example
if $R_f^-(s)$ and $R_f(t)$ land at $z$, then $\sector_f (z,s,t)$ denotes the region containing all rays with arguments in $]s,t[$.

Consider a point $z$ in the boundary of a Fatou component
$V$ of $f \in \cS_p$. Following Poirier~\cite{PoirierCP}, if
$\sector_f(z,s,t)$ is the sector based at $z$ containing $V$, we say that
$s$ (resp. $t$) is the \emph{left (resp. right) supporting argument of $V$ at $z$}.

\subsection{Parameter external rays}
\label{s:parameter-rays}
For every $f \in \cE (\cS_p)$ the cocritical point $2 a(f)$ 
belongs to the domain $V^*_f(\infty)$ of the B\"ottcher coordinate $\phi_f$.
Moreover, $f \to \phi_f (2 a (f))$ depends holomorphically on $f \in \cE (\cS_p)$. Furthermore, the map
$$
\begin{array}{rccc}
  \Phi_\cU: &\cU & \to & \CDC\\
   & f  & \mapsto & \phi_f (2 a (f)), 
\end{array}
$$
is a regular covering map (without branched points) of some degree called the \emph{multiplicity of $\cU$}~\cite[Lemma~5.9]{MilnorPeriodicCubic}.

A \emph{parameter ray $\cR_\cU (\theta)$ with argument $\theta \in \RZ$} is defined as an arc in $\cU$ which maps bijectively onto
$]1,\infty[ \exp (2 \pi i \theta)$ by $\Phi_\cU$. Be aware that there might be
several parameter rays $\cR_\cU (\theta)$ associated to the same argument $\theta$. Note that $f$ lies  in a ray $\cR_\cU (\theta)$ if and only if
$R^*_f(\theta \pm 1/3)$ terminate at $-a(f)$.

\subsection{Internal rays} 

\label{s:internal-rays}

We work under the assumption that
$-a(f)$ does not belong to the periodic orbit of $a(f)$ unless
otherwise stated.
That is, the local degree
of $f^p$ at $a_k(f)$ is $2$, for all $k$ subscripts modulo $ p$.
There exists a neighborhood $W$ of $a_k(f)$ and a unique 
univalent map $\varphi_{f,a_k} : W \to \D$
 conjugating the action of $f^p$ with
$z \mapsto z^2$. We say that $\varphi_{f,a_k} : W \to \D$ is the \emph{B\"ottcher coordinate at $a_k$.}
It follows that $$ \varphi_{f,a_{k+1}} \circ f (z) = \left( \varphi_{f,a_{k}} (z) \right)^{\deg_{a_k} f}$$
where $\deg_{a_k} f =1$ if $k \neq 0$ and $\deg_{a_0} f =2$.
According to~\cite[Corollary~9.2]{DynInOne},
$|\varphi_{f,a_k}|$ has a unique continuous extension
 to $V_f(a_k(f))$ which
satisfies the functional relation $|\varphi_{f,a_k} (f^p(z))| =
|\varphi_{f,a_k}(z)|^2$. 
In analogy with what occurs in the basin of infinity we let $V^*_f(a_k(f))$ be the basin 
of $a_k(f)$ under the flow of $-\nabla |\varphi_{f,a_k}|$. The B\"ottcher coordinate extends along flow lines to 
a conformal isomorphism
$$\varphi_{f,a_k} : V^*_f(a_k(f)) \to U^*_{f,a_k} \subset \D$$
where $U^*_{f,a_k}$ is a star-like domain around $w=0$.  The zeros of
$\nabla  |\varphi_{f,a_k}|$ are exactly the points in $V_f(a_k(f))$ that
eventually map onto $-a(f)$.

If $-a(f) \notin V_f(a_j(f))$ for all $j$, or equivalently
$f$ is not in a type $A$ or $B$ component, then the B\"ottcher coordinate extends to a conformal isomorphism
$\varphi_{f,a_k}: V_f(a_k(f)) \to \D$. Moreover, $\partial  V_f(a_k(f))$ is a Jordan curve and $\varphi_{f,a_k}$ extends continuously to a conjugacy
$\varphi_{f,a_k}: \overline{V_f(a_k(f))} \to \overline{\D}$
between $f^p$ and $z \mapsto z^2$. In particular, there is exactly one periodic point of period dividing $p$ in $\partial V_f(a_k(f))$. (This holds in further generality e.g. see~Proposition~\ref{dynamics-typeAB}.)

Our main interest here is on maps $f$
that lie in a type $A$ or $B$ hyperbolic component.
For such a map $f$ be aware that $V^*_f(a_k(f))$ 
is not dense in $V_f(a_k(f))$, in contrast with the discussion in~\S~\ref{s:external-rays}. 
Given $t \in \RZ$ there exists a maximal
$0< \rho(t) \le 1$ such that $[0,\rho(t)[ \cdot \exp(2 \pi i t) \subset U^*_{f,a_k}$.
Let
$$I^*_{f,a_k} (t) = \varphi_{f,a_k}^{-1} ([0,\rho(t)[ \cdot \exp(2 \pi i t)).$$
This arc starts at $a_k(f)$ and, at the other end, it approaches $\partial V_f (a_k(f))$ or
an iterated preimage $w$ of $-a(f)$. In the former case $\rho(t)=1$ and we say 
that $I_{f,a_k}(t)$ is an \emph{internal ray at $a_k$ with argument $t$}. In the latter $\rho(t)  <1$ and we say that
$I^*_{f,a_k} (t)$ \emph{terminates} at $w$. Right and left limit internal rays $I^\pm_{f,a_k}(t)$ are defined 
as in~\S~\ref{s:external-rays}.  

\section{Ray connections}
\label{s:ray-connection}
The definition of ray connections relies on the notion 
 of relatively supporting rays.
This notion generalizes the definition of left and right supporting rays for Fatou components (see~\S~\ref{s:sector-notation}).

\begin{definition}
  \label{d:relatively-supporting}
  Given a bounded Fatou component $V$ of a map $f \in \cS_p$ consider  an eventually periodic point $z \in \partial V$.
  There exist unique $t_1,t_2 \in \QS$ (possibly equal)
  in the same grand orbit under $m_3$ such that 
  external rays (maybe bouncing) with arguments $t_1, t_2$ land at $z$,
  $V \subset \sector (z,t_1,t_2)$ and, if a ray with argument $ t \in ]t_1,t_2[$  lands
  at $z$, then $t$ is not in the grand orbit of $t_1,t_2$.
We say that $t_1$ (resp. $t_2$) is a \emph{relatively left (resp. right) supporting argument of $V$ at $z$}.
\end{definition}

Recall that for any $f \in \cS_p$ and $0 \le k < p$, there exists a period $p$ argument $t$ such that a ray  with argument $t$ lands at $\partial V(a_k(f))$.  Among
the arguments in the orbit of $t$ we can always choose one that is relatively left (resp. right)  supporting for $V(a_k(f))$.
We are particularly interested on the situation when $f$ lies in an escape region and such a ray is one bouncing off the critical point $-a(f)$. We will also be interested on the situation when $t$ itself is not periodic but $3t$ has period $p$ and a ray with argument $t$ is relatively supporting for $V(a(f))$.

\medskip
Consider an escape region $\cU$ and $f \in \cU$.
  Recall that an external ray $R_f^\sigma(\vartheta)$, where $\sigma \in \{ +,-\}$, is  a \emph{ray connection between $-a(f)$ and $a_k(f)$} if the following hold:
  \begin{itemize}
    \item $3 \vartheta$ is periodic of period $p$ under $m_3$. 
    \item $-a(f) \in R_f^\sigma(\vartheta)$.
    \item $R_f^\sigma(\vartheta)$ is a relatively (left or right) supporting ray of $V(a_k(f))$.
  \end{itemize}

   If a map $f$ in a parameter ray $\cR_\cU (\theta)$ has a ray connection $R_f^\sigma (\vartheta)$  between $-a(f)$ and $a_k(f)$, then $\vartheta = \theta + 1/3$ or $\theta-1/3$. Moreover,   all the maps in $\cR_\cU (\theta)$ have such a ray connection. That is, ray connections are preserved along parameter rays.
  
  A periodic ray connection $R_f^\sigma (\vartheta)$ lands at the
  unique point $z_0$ of period dividing $p$ in $\partial V(a_k(f))$.
  A preperiodic ray connection lands at the unique
  non-periodic point $w_0$ in $\partial V(a_0(f))$ such that $f(w_0)$
  has period dividing $p$. 
  That is, preperiodic ray connections only occur between
  $-a(f)$ and $a_0(f)=a(f)$. 

\medskip
\medskip
The key to start our journey towards the distinguished escape region is the following result:

\begin{theorem}
  \label{thr:ray-connection}
  Given an undistinguished escape region $\cU$,
  there exists $f  \in \cU$ with a ray connection between 
$-a(f)$ and $a_k(f)$  such that the following statements hold:
\begin{enumerate}
\item $a_k(f) \in D_0$,
\item the return time of $a_k(f)$ to $D_0$ is maximal.
\end{enumerate}
\end{theorem}

The rest of the section is devoted to the proof of this theorem.

\smallskip
Given an escape region $\cU$ it will be convenient to consider a one-parameter family
$$
\begin{array}[h]{ccc}
  \RZ \setminus \{0\} & \to & \cU \\
  \theta & \mapsto & f_\theta
\end{array}
$$
 such that:
\begin{enumerate}
\item $f_\theta = f_{a(\theta), v(\theta)}$ where $a(\theta)$ and $v(\theta)$ are real analytic functions.
\item $f_\theta$ lies in a parameter ray with argument $\theta$.
\item The escape rate of $-a (\theta)$ is independent of $\theta$.
\end{enumerate}
The existence of such a family is guaranteed by~\S~\ref{s:parameter-rays}.
In the sequel $f_\theta$ will always denote a family as above in an escape region $\cU$ which will be clear from context. The external rays of $f_\theta$ will be simply denoted by $R_\theta (t)$, the Fatou component containing $a_k(f_\theta)$
simply by $V(a_k(\theta))$, etc.

It is not difficult to show that, for any $t \in \RZ$,  as $\theta \nearrow \theta_0$, the rays $R_\theta(t)$ converge to $R_{\theta_0}^+ (t)$, as arcs parametrized by $|\phi_\theta|$.
A similar statement holds as $\theta \searrow \theta_0$. 

Let us first show that ray connections are present in
any escape region. However, it is essential (and harder) to obtain 
ray connections with the extra properties required in the theorem.

\begin{lemma}
  \label{l:easy-ray-connection}
  If $\cU$ is an escape region, then there exists $f \in \cU$ with a ray connection
  between $-a(f)$ and $a_k(f)$ for some $k$.
\end{lemma}

\begin{proof}[Proof of Lemma~\ref{l:easy-ray-connection}]
  Pick any $0 < \theta <1/3$ and let $t$ be the argument of a relatively supporting periodic ray $R^\pm_\theta(t)$
  landing at the periodic  point $w(\theta)$ in $\partial V(a(\theta))$ of period dividing $p$. If the ray is not smooth, then
  an iterate, say  $R^\pm_\theta(3^k t)$, contains $-a(\theta)$. Hence, it is a ray connection between $-a(\theta)$ and
  $a_k(\theta)$. Thus we may assume that $R_\theta(t)$ is smooth and relatively supporting for $V(a(\theta))$.
  Then $\theta -1/3 < t < \theta +1/3$. For some minimal $\theta' \in ]\theta, t+1/3]$ we have that rays at argument $t$ are not smooth. Such a $\theta'$ exists since for $\theta=t+1/3$ the rays at arguments $t$ and $t+2/3$ are not smooth.
  It follows that $R^+_{\theta'} (t)$ is relatively supporting for $V(a(\theta'))$ since it lands at the analytic continuation
  of $w(\theta)$. As before, we conclude that an iterate of $R^+_{\theta'} (t)$ is a ray connection between $-a(\theta')$ and
  $a_k(\theta')$ for some $k$.
\end{proof}

In the special case that $\kappa(\cU) = 0^p$, the existence of a ray connection
as claimed in Theorem~\ref{thr:ray-connection} follows from the above lemma.
For a general escape component $\cU$, our search for ray connections with maximal return time starts
by determining certain parameter intervals of $\theta$ such that the rays
bouncing off $-a(\theta)$ land at points in $D_0$ with a prescribed return time to $D_0$. 

\medskip
Given $\theta \in \RZ$ we are interested on the dynamics of $m_3: t \mapsto 3t$
according to the partition of 
$\RZ$  by the intervals:
$$\{ I^\sigma_0 (\theta), I^\sigma_1 (\theta) \}$$ where
$\sigma = + \mbox{ or } -$, and
\begin{eqnarray*}
I^+_0 (\theta) &=& [\theta -1/3, \theta+1/3[,\\
   I^+_1 (\theta) &=&  [\theta+1/3, \theta-1/3[,
\end{eqnarray*}
\begin{eqnarray*}
I^-_0 (\theta) &=& ]\theta -1/3, \theta+1/3],\\
   I^-_1 (\theta) &=&  ]\theta+1/3, \theta-1/3].
\end{eqnarray*}
These intervals are chosen so that $R^+_\theta(t)$ lands in $D_0$ or $D_1$ according to whether $t \in I^+_0(\theta)$ or $I^+_1(\theta)$.
A similar statement holds for left limit rays. 
Now define
$$
\begin{array}[h]{rccl}
  \itin^\sigma_\theta:&\RZ & \to & \{0,1\}^{\N_0}\\
  &t& \mapsto & (i_n)
\end{array}
$$
where $i_n$ is such that $m_3^n (t) \in I_{i_n}^\sigma(\theta)$,
for all $n \ge 0$.
Note that $\itin^+_\theta (t) = \itin^-_\theta (t)$ for all $t$ which are not eventually mapped
into $\{ \theta-1/3, \theta+1/3 \}$. That is, for all $t$ such that $R_\theta(t)$ exists and is smooth.
 
\begin{lemma}
  \label{itineraries}
  For $\ell \ge 2$, consider the arguments:
 \begin{eqnarray*}
   \alpha_\ell &=& \dfrac{1}{3} -  \dfrac{1}{3^\ell-1},   \\
   \beta_\ell &=& \dfrac{1}{3} -  \dfrac{1}{3(3^\ell-1)}. 
 \end{eqnarray*}
 \begin{enumerate}
 \item If $\theta \in ]\beta_\ell,1/3[$ or $\theta \in - ]\beta_\ell,1/3[$, then for some $i_n, j_n$:
   \begin{eqnarray*}
          \itin^-_\theta (\theta+1/3) & =& 0 1^{\ell} i_{\ell+1} i_{\ell+2} \cdots,\\
     \itin^+_\theta (\theta-1/3) & = &0 1^{\ell} j_{\ell+1} j_{\ell+2} \cdots.
   \end{eqnarray*}
\item If $\theta \in ]\alpha_\ell, \beta_\ell[$ or $\theta \in -]\alpha_\ell,\beta_\ell[$, then for some $i_n, j_n$:
   \begin{eqnarray*}  
     \itin^-_\theta (\theta+1/3) &= & 0 1^{\ell-1} 0 \, i_{\ell+1} i_{\ell+2} \cdots,\\
     \itin^+_\theta (\theta-1/3) &= & 0 1^{\ell-1} 0 \, j_{\ell+1} j_{\ell+2} \cdots.
   \end{eqnarray*}
  \end{enumerate}
\end{lemma}


Let us denote the unique fixed point  of $f_\theta$ in $D_1$ by $z(\theta)$.
The arguments $\alpha_\ell$ and $\beta_\ell$ are closely related to the rays landing at
 $z(\theta)$.
In fact, $\alpha_\ell-1/3= m_3(\beta_\ell+1/3)$ has a period $\ell$ orbit $\cO$ under $m_3$ contained in $[\beta_\ell+1/3,\alpha_\ell-1/3]$ with combinatorial rotation number $-1/\ell$. For all $\theta \in ]\alpha_\ell,\beta_\ell[$, it follows that $\cO \subset 
]\theta+1/3,\theta-1/3[$ and  the external rays
of $f_\theta$  with arguments in $\cO$ land at $z(\theta)$. Similar statements hold for $\theta \in -]\alpha_\ell,\beta_\ell[$
with the difference that the orbit $\cO$ contained in $[-\alpha_\ell+1/3,-\beta_\ell-1/3]$ now has combinatorial rotation number $1/\ell$. The orbits
$\cO$ are examples of ``rotation sets'' thoroughly discussed in~\cite{ZakeriRotationSets}.

 \begin{proof}[Proof of Lemma~\ref{itineraries}]
   It is sufficient to  prove assertion (1) for $\theta \in ]\beta_\ell, 1/3[$ and
   assertion (2) for $\theta \in ]\alpha_\ell, \beta_\ell[$ since
   $m_3^n (\theta) \in I_1^\sigma (\theta)$ if and only if 
$m_3^n (-\theta) \in I_1^{-\sigma} (-\theta)$.

If $\theta \in ]\beta_\ell,1/3[$, then
$$\dfrac{-1}{3(3^\ell-1)} < \theta- \dfrac13 < 0,$$
where we employ hereafter ``$<$''
to denote cyclic order in  $\RZ$.
Therefore, $$\theta -\dfrac{2}{3} < 3^\ell \theta < 3^{\ell-1} \theta < \cdots < 3 \theta < \theta- \dfrac13 < 0$$
and (1) follows.

If $\theta \in ]\alpha_\ell, \beta_\ell[$, then 
$$ \theta - \dfrac13 <3^\ell \left(\theta - \dfrac13\right) < \left(\theta - \dfrac13\right) - \dfrac13 = \theta - \dfrac23.$$
Thus $m^\ell_3 (\theta) \in I_0^\sigma(\theta)$.  For $\ell =2$ it is easy to check that $3 \theta \in I_1^\sigma$ and (2) follows.
For $\ell \ge 3$
we may apply (1) to obtain (2) since $]\alpha_\ell, \beta_\ell[ \subset ]\beta_{\ell-1}, 1/3[$.

\end{proof}

\medskip
Let us now locate in the dynamical plane of $f_\theta$ the points
with a given return time to $D_0$. 
Following Branner and Hubbard~\cite{BrannerHubbardCubicI,BrannerHubbardCubicII}, for $\ell \ge 1$,  we say that 
$$L^{(\ell)}_\theta = \{ z \in \C: f^j_\theta (z) \in D_{0} \cup D_1, 0 \le j < \ell\}$$
is the \emph{level $\ell$ set of $f_\theta$}.
Each connected component of a level $\ell$ set is a Jordan domain called \emph{a level $\ell$ disk}. For $z \in L^{(\ell)}_\theta$ we denote by $$D^{(\ell)}_\theta (z)$$
the unique level $\ell$ disk containing $z$. 

\medskip
Given a word $w = (i_0, \dots, i_{\ell-1}) \in \{ 0, 1 \}^\ell$ let 
$$L^{(\ell)}_\theta (w)= \{ z \in \C: f^j_\theta (z) \in D_{i_j}, 0 \le j < \ell\}$$
We say that points in $L^{(\ell)}_\theta (w)$ have itinerary starting with $w$. 
Note that our numbering of the levels is so that words of length $\ell$ determine subsets
of the set of level $\ell$. In particular, 
$$L^{(1)}_\theta (0) = D_0, \quad L^{(1)}_\theta (1) = D_1.$$
Both level $1$ disks have a common image $D$ under $f_\theta$.
While $f_\theta:D_0 \to D$ has degree $2$,
the map $f_\theta:D_1 \to D$ has degree $1$. Recall that  $D_1$ contains a unique fixed point
denoted $z(\theta)$. For all $\ell \ge 1$,
$$L^{(\ell)}_\theta (1^{\ell}) =D^{(\ell)}_\theta (z(\theta)).$$
Moreover, each disk $D^{(\ell)}_\theta (z(\theta))$ contains exactly two disks of the next level.
One of this disks is $D^{(\ell+1)}_\theta (z(\theta))$ and the other one 
agrees with $L^{(\ell+1)}_\theta (1^{\ell}0)$.

Our main interest is to locate points in the orbit of $a(\theta)$ with maximal return time to $D_0$.
These points have itineraries starting with  $01^{\ell-1}0$ for some $\ell \ge \mu_0$
where $\mu_0$ is the return time of $a(\theta)$ to $D_0$. Their location is closely related to the two preimages of the fixed point  $z(\theta)$ in $D_0$.

\medskip
\noindent
{\bf Notation:}
\emph{Denote by $z'(\theta)$ and $z''(\theta)$ the two $f_\theta$-preimages in $D_0$ of the unique $f_\theta$-fixed point $z(\theta) \in D_1$.}

\medskip
For $\ell \ge 2$,
\begin{equation}
  \label{eq:01ell}
  L^{(\ell)}_\theta (01^{\ell-1}) = D^{(\ell)}_\theta (z'(\theta)) \cup D^{(\ell)}_\theta (z''(\theta)), 
\end{equation}
where the union is disjoint if and only if $\ell > \mu_0$; otherwise the disks involved coincide
with $D^{(\ell)}_\theta (a(\theta))$. 
The level $\mu_0+1$ disk $D^{(\mu_0+1)}_\theta (a(\theta))$
is the preimage in $D_0$ of $L^{(\mu_0)}_\theta (1^{\mu_0-1}0)$, therefore
\begin{equation}
  \label{eq:mu0}
  L^{(\mu_0+1)}_\theta (01^{\mu_0-1}0) = D^{(\mu_0+1)}_\theta (a(\theta)). 
\end{equation}
If $\ell > \mu_0$ and $z =  z'(\theta)$ or $z''(\theta)$, then
$D^{(\ell)}_\theta (z)$ 
contains exactly two disks of level $\ell+1$. One of these disks is $D^{(\ell+1)}_\theta (z)$.
We denote the other one by $C^{(\ell+1)}_\theta (z)$ and call it the \emph{level $\ell+1$ companion disk of $z$}.
Thus, for $\ell > \mu_0$,
\begin{equation}
  \label{eq:ell}
L^{(\ell+1)}_\theta({01^{\ell-1}0}) = C^{(\ell+1)}_\theta (z'(\theta)) \sqcup C^{(\ell+1)}_\theta (z''(\theta)).
\end{equation}

Now our aim is to show that for $\ell \ge \mu_0$ we can adjust the parameter $\theta$ in order
to 
connect $-a(\theta)$ with a level $\ell$ disk around a given
prefixed point $z'(\theta)$ or $z''(\theta)$.
We label $z'(\theta)$ and $z''(\theta)$ such that:

\smallskip
\centerline{\emph{$R_{1/3}^-(2/3)$ lands at $z'(1/3).$}}
\smallskip
For any $\theta \in ]1/3,2/3[$, the external ray $R_\theta(0)$ lands at the fixed point $z(\theta)$. When we cross from $\theta < 1/3$ to $\theta >2/3$
 rays bouncing off the critical point $-a(\theta)$ switch from landing ``close'' to $z'(\theta)$ to landing ``close'' to $z''(\theta)$:

\begin{lemma}
  \label{lem:4}
  Assume that $\mu_0 >1$ where $\mu_0$
  is the return time of $a(\theta)$ to $D_0$. 
If $\ell \ge \mu_0 $, then the following statements hold:
\begin{enumerate}
\item for all $\theta \in ] \beta_\ell,1/3[$ the landing points of $R^-_\theta(\theta+1/3)$ and $R^+_\theta(\theta-1/3)$ belong to $D^{(\ell+1)}_\theta (z'(\theta))$,
\item for all $\theta \in ]2/3, -\beta_\ell[$ the landing points of $R^-_\theta(\theta+1/3)$ and $R^+_\theta(\theta-1/3)$ belong to $D^{(\ell+1)}_\theta (z''(\theta))$.
\end{enumerate}
\end{lemma}

\begin{proof}
  
  For all $\theta \in ]\beta_\ell,1/3[$ and $1 \le k \le \ell$, a direct computation shows that $m_3^k(\theta \pm 1/3) \notin \{ \theta +1/3,\theta-1/3\}$. Hence,  the rays $R^-_\theta(\theta+1/3)$ and
  $R^+_\theta(\theta-1/3)$ share a common sub-arc between $-a(\theta)$ and
  $L^{(\ell+1)}_\theta$ since the sub-arc in each ray is free of iterated preimages of
  $-a(\theta)$. Therefore these rays land in the same level $\ell+1$ disk, say $D_\theta$.
  By Lemma~\ref{itineraries}  (1) the landing point has  itinerary starting with $01^\ell$.
  From~(\ref{eq:01ell}), we have that $D_\theta
  =D_\theta^{(\ell+1)}(z'(\theta))$ or 
  $D_\theta^{(\ell+1)}(z''(\theta))$. By continuity of the arcs and disks involved, we have that either $D_\theta
  =D_\theta^{(\ell+1)}(z'(\theta))$ for all $\theta \in ]\beta_\ell,1/3[$ or 
   $D_\theta
   =D_\theta^{(\ell+1)}(z''(\theta))$ for all $\theta \in ]\beta_\ell,1/3[$.
   Our choice of labeling for $z'(\theta)$ and $z''(\theta)$ implies that, for all $t \in ]2/3 - 1/3^{\ell+1}, 2/3[$ such that the external ray $R_{1/3}(t)$ is smooth, this ray $R_{1/3}(t)$  lands  in $D^{(\ell+1)}_{1/3} (z'(1/3))$.  If
   $\theta \in ]1/3-1/3^{\ell+1}, 1/3[$ then
   all the smooth rays  $R_\theta (t)$ with arguments in $]2/3-1/3^{\ell+1}, \theta+1/3[$ land in the same level $\ell+1$ disk $D_\theta$  which depends continuously on $\theta$.
   Therefore, $D_\theta = D^{(\ell+1)}_{\theta} (z'(\theta))$ and the first statement of the lemma holds.

   For $\theta \in ]1/3,2/3[$, the (smooth) fixed  external  ray $R_\theta(0)$ lands at $z(\theta)$ and the landing point
   of the smooth ray $R_\theta(2/3)$ is 
 either  $z'(\theta)$ or  
 $z''(\theta)$. 
 From our choice of labeling  and  continuity
we conclude that  $R_\theta(2/3)$ lands at $z'(\theta)$  for all $\theta  \in ]1/3,2/3[$. Therefore $R^+_{2/3}(2/3)$ lands at $z'(2/3)$.
Note also that $R_{2/3}^+(0)$ lands at $z(2/3)$.
It follows that $R^+_{2/3} (1/3)$ lands at $z''(2/3)$ since this ray is the other preimage of $R_{2/3}^+(0)$.

The second statement of the lemma now follows along  similar  lines. Indeed,
for parameters $\theta \in ]2/3, -\beta_\ell[$
again  $R^+_\theta (\theta-1/3)$ and $R^-_\theta (\theta+1/3)$
must land in a continuously varying level $\ell+1$ disk $D'_\theta$.
We may argue as above to conclude that   $D'_\theta =D^{(\ell+1)}_{\theta} (z''(x))$
since from the previous paragraph we know that for $t \in ]1/3,1/3+1/3^{\ell+1}[$ the rays $R_{2/3}(t)$ land
inside $D^{(\ell+1)}_{2/3} (z''(x))$.
\end{proof}

\begin{corollary}
  \label{lem:2}
  Suppose
  that $\ell > \mu_0 >1$ where $\mu_0$ is the return time of $a_0(\theta)$ to $D_0$.
Let
$$\begin{array}{ccr}
  \theta_0 &\in& ]\alpha_\ell,\beta_\ell[,\\
  \theta_1&\in &-]\alpha_\ell,\beta_\ell[.\\
\end{array}
$$
Then  the following statements hold:
\begin{enumerate}
\item $R^-_{\theta_0} (\theta_0+1/3)$ and $R^+_{\theta_0}(\theta_0 -1/3)$ land at points in the companion
  disk $C^{(\ell+1)}_\theta (z'(\theta_0))$.
\item $R^-_{\theta_1} (\theta_1+1/3)$ and $R^+_{\theta_1}(\theta_1 -1/3)$ land at points in the companion
  disk $C^{(\ell+1)}_\theta (z''(\theta_1))$.
\end{enumerate}
\end{corollary}

\begin{proof}
  Let us just prove the first assertion
  as the proof of the second one is similar. 
  Note that $]\alpha_\ell,\beta_\ell[ \subset ]\beta_{\ell-1},1/3[$.
  By the previous lemma, $R^\mp_{\theta_0} (\theta_0\pm 1/3)$ lands in $ D^{(\ell)}_{\theta_0}(z'(\theta_0))$.
  In view of Lemma~\ref{itineraries}, the itinerary $\itin^\mp_{\theta_0} (\theta_0\pm 1/3)$ starts with $01^{\ell-1}0$.
From~\ref{eq:ell}, the corresponding rays land in the companion disk $C^{(\ell+1)}_\theta (z'(\theta_0))$.
\end{proof}

\begin{lemma}
  \label{lem:5}
Let $\mu$ be the maximal return time to $D_0$ and assume that $1<\mu <p$.
Then at least one of the following statements hold:

\begin{enumerate}
\item  
$\cO(a(\theta)) \cap D^{(\mu)}_{\theta}(z'(\theta)) \neq \emptyset$, for some $\theta$, and
there exists $\theta_0 \in ]\alpha_\mu, \beta_\mu[$ such that   $f_{\theta_0}$ has a ray connection between $-a(\theta_0)$ and $a_k(\theta_0)$ for some $k$ such that $a_k(\theta_0) \in D^{(\mu)}_{\theta_0}(z'(\theta_0))$.

\item
$\cO(a(\theta)) \cap D^{(\mu)}_{\theta}(z''(\theta)) \neq \emptyset$, for some $\theta$, and  
there exists $\theta_1 \in ]\alpha_\mu, \beta_\mu[$ such that   $f_{\theta_1}$ has a ray connection between $-a(\theta_1)$ and $a_k(\theta_1)$ for some $k$ such that $a_k(\theta_1) \in D^{(\mu)}_{\theta_1}(z''(\theta_1))$.
\end{enumerate}
\end{lemma}

\begin{proof}
  Points having itinerary
  starting with $01^{\mu-1}$ belong to $ D^{(\mu)}_{\theta}(z'(\theta)) $ or  to $D^{(\mu)}_{\theta}(z''(\theta))$, in view of~\ref{eq:01ell}. Thus if $a_j(\theta) \in D_0$ has return time $\mu$ to $D_0$, then
  $a_j(\theta) \in  D^{(\mu)}_{\theta}(z'(\theta)) $ for all $\theta$
  or,  $a_j(\theta) \in D^{(\mu)}_{\theta}(z''(\theta))$   for all $\theta$.
  We assume the former and  prove that the first assertion holds
  for some $k$ not necessarily equal to $j$. The other case
  follows along similar lines.

  When $\mu_0 = \mu$ we let  $j=0$ and observe that $a_0(\theta) = a(\theta)$
  belongs to  $D^{(\mu)}_{\theta}(z'(\theta)) =D^{(\mu)}_{\theta}(z''(\theta))$. Otherwise we choose any $j$ such that $a_j(\theta) \in  D^{(\mu)}_{\theta}(z'(\theta)) $.
  Note that $a_j(\theta)$ has itinerary starting with $01^{\mu-1}0$. If $\mu >\mu_0$, then 
  $D^{(\mu+1)}_{\theta}(a_j(\theta))$ is the level $\mu+1$ companion disk of $z'(\theta)$.
  Let $\theta \in ]\alpha_\mu, \beta_\mu[$.
  By~\ref{eq:mu0} if $\mu=\mu_0$ and by Corollary~\ref{lem:2} if $\mu > \mu_0$, the rays $R^\mp_\theta(\theta\pm1/3)$ land
  in $D^{(\mu+1)}_{\theta}(a_j(\theta))$. That is, our previous work says that rays bouncing off $-a(\theta)$ into $D_0$
  land in the level $\mu+1$ disk
  of $a_j(\theta)$ whenever $\theta \in ]\alpha_\mu, \beta_\mu[$.
  We will find 
  a suitable value of $\theta \in ]\alpha_\mu, \beta_\mu[$ such that
  $f_\theta$ has a ray connection with return time $\mu$.

Note that $\alpha_\mu-1/3$  which has period $\mu < p$ under $m_3$. In particular, no period $p$ ray is bouncing. 
Thus we may consider a period $p$ argument $t$  such that the external ray $R_{\alpha_\mu} (t)$ is smooth and relatively supporting for $V(a_j(\alpha_\mu))$.
Moreover, $R_{\theta} (t)$
is a relatively supporting ray of $V(a_j(\alpha_\mu))$ for all  $\theta$ in  a sufficiently small neighborhood of
$\alpha_\mu$. 

\emph{Claim.} If for all  $\theta \in  [\alpha_\mu,\beta_\mu]$ the ray $R_\theta(t)$
exists (it is smooth), then there exists $\theta_0 \in  ]\alpha_\mu,\beta_\mu[$ such that
$3^\mu t = 3^\mu (\theta_0-1/3)$.

\emph{Proof of the Claim.}
Assume that the ray $R_{\theta}(t)$ is smooth for all $\theta \in [\alpha_\mu,\beta_\mu]$.
Consider the continuous functions $\delta(\theta) \ge 0$ and $\hat{\delta}(\theta) \ge 0$ with domain $[\alpha_\mu,\beta_\mu]$,
where $\delta(\theta)$ is defined as the arc length of $[ \theta-1/3,
m_3^\mu (\theta-1/3) ] \subset \RZ$
and $\hat{\delta}(\theta)$ is the arc length of the interval $[\theta-1/3,
m^\mu_3 (t)]\subset \RZ$. 
The arguments $\alpha_\mu-1/3$ and $\beta_\mu+1/3$ have period $\mu$ and therefore
$\delta(\alpha_\mu)=0$ and $\delta(\beta_\mu) =2/3$.
Observe that $R_\theta(m_3^\mu(t))$ lands in $\partial V(a_{j+\mu}(\theta))$.
 Since $a_{j+\mu}(\theta) \in D_0$ we have that
the external ray with argument $m^\mu_3 (t)$ lands in $D_0$.
Therefore $\hat{\delta}(\theta) \le 2/3$.
Taking into account that $t$ has  period $p > \mu$ it  follows that $0<\hat{\delta}(\theta)<2/3$. By the Intermediate Value Theorem, there exists
$\theta_0 \in  ]\alpha_\mu,\beta_\mu[$ such that $\delta(\theta_0) = \hat{\delta}(\theta_0)$. That is, $3^\mu t = 3^\mu(\theta_0-1/3)$ and the claim follows.

\smallskip
The proof continues by considering two situations. The first situation is when there exist $\theta$ in
$[\alpha_\mu,\beta_\mu]$ such that rays with argument $t$ of $f_\theta$ are not smooth.
The second is when $R_\theta(t)$ is smooth for all $\theta \in [\alpha_\mu,\beta_\mu]$. In the latter situation the claim allows us to work with a parameter $\theta_0$ such that also $3^\mu t = 3^\mu(\theta_0-1/3)$. Below we show that the first situation yields a periodic ray connection while the second leads to a preperiodic one. 

\smallskip
In the first situation let
 $\theta_0 \in  ]\alpha_\mu,\beta_\mu[$  be minimal  such that the 
rays with argument $t$ are not smooth. As $\theta \in ]\alpha_\mu,\theta_0[$ converges to
$\theta_0$ the rays $R_\theta(t)$ converge to $R^+_{\theta_0} (t)$. Hence, $R^+_{\theta_0} (t)$ is a non-smooth ray relatively supporting $V(a_j(\theta_0))$. Since every non-smooth ray eventually maps onto one containing $-a(\theta_0)$, there exists $j'$ such that
$R^+_{\theta_0} (3^{j'}t)$ contains $-a(\theta_0)$
and relatively supports $V(a_{k}(\theta_0))$ where $k=j+j'$.
It follows that $3^{j'} t = \theta_0 -1/3$ or $\theta_0+1/3$.
The itinerary $\itin^+_{\theta_0}(3^{j'}t)$ coincides with
the itinerary of $a_k(\theta_0)$ according to the disks $D_0$ and $D_1$.
Thus $3^{j'} t \neq \theta_0+1/3$, since $\itin^+_{\theta_0}(\theta_0+1/3)$ starts with $1^\mu 0$ and the maximal length for a
string of $1$'s in $\kappa(\cU)$ is $\mu-1$. Therefore, $3^{j'} t = \theta_0-1/3$ and $R^+_{\theta_0} (\theta_0-1/3)$ is
a periodic
ray connection between $-a(\theta_0)$ and $a_{k}(\theta_0)$.
Lemma~\ref{lem:4} (2) yields that 
 $a_{k}(\theta_0)$ lies in $D^{(\mu)}_{\theta_0}(z'(\theta_0))$.

\smallskip
Now we consider the second situation
in which $R_{\theta} (t)$ is smooth for all $\theta \in [\alpha_\mu,\beta_\mu]$ and  $\theta_0 \in ]\alpha_\mu,\beta_\mu[$ is such that
$3^\mu t = 3^\mu(\theta_0-1/3)$.
To lighten notation we omit the dependence on $\theta_0$  of rays and disks.
Recall that  $R^- (\theta_0+1/3)$ lands in $D^{(\mu+1)}(a_j(\theta_0))$.
Let $\zeta$ be the intersection point of  $R^- (\theta_0+1/3)$ with 
$\partial D^{(\mu+1)}(a_j(\theta_0))$. Note that $\zeta$ also lies in $R^+ (\theta_0-1/3)$
since the portion of $R^-(\theta_0+1/3)$ between the boundaries of levels $1$ and $\mu+1$ is free 
of iterated preimages of $-a(\theta_0)$.
Also, $f^\mu_{\theta_0} (\zeta) \in R(3^\mu t)$
since $3^\mu t = 3^\mu(\theta_0-1/3)$. Moreover, the smooth ray $R(t)$
intersects $\partial D^{(\mu+1)} (a_j(\theta_0))$ at one point, say $\xi$.
If $ \mu > \mu_0$, then $f^\mu_{\theta_0} : \partial D^{(\mu+1)} (a_j(\theta_0)) \to D_0$ is injective, and therefore $\zeta = \xi  \in R (t)$ which contradicts our assumption that $R  (t)$ is smooth.
Thus, $\mu = \mu_0$, $a_j(\theta_0) = a (\theta_0)$, and  $f^\mu_{\theta_0} : \partial D^{(\mu+1)}(a(\theta_0)) \to \partial D_0$ is two-to-one. In this case, $f_{\theta_0} : \partial D^{(\mu+1)}(a(\theta_0)) \to  \partial D^{(\mu)}(a_1(\theta_0))$ is two-to-one.
The gradient flow lines between $\zeta$ and the Julia set as well as the flow line between $\xi$ and the Julia set
map under $f_{\theta_0}$ onto the same sub-arc of $R(3t)$.
It follows that $t=\theta_0$. Thus $R(\theta_0)$ lands at the unique
point of period dividing $p$  in $\partial V(a(\theta_0))$ and
the rays $R^\pm (\theta_0 \mp 1/3)$ also land at a point in
 $\partial V(a(\theta_0))$. Moreover,
one of these rays is relatively supporting for $V(a(\theta_0))$ because $R(\theta_0)$ is relatively supporting. That is, $R^+ (\theta_0 - 1/3)$ or  $R^- (\theta_0 + 1/3)$ is a  preperiodic ray connection between $-a(\theta_0)$ and $a(\theta_0)$.
\end{proof}

For $\theta=\theta_0$ or $\theta_1$ as in the Lemma, the return time of
$a_k(\theta)$ to $D_0$ is $\mu$ and hence maximal. Therefore, the lemma finishes the proof of Theorem~\ref{thr:ray-connection} by establishing the existence of the corresponding ray connection.

\section{Landing and take-off}
\label{s:landing}

The aim of this section is to discuss where rational parameter rays land and
which rays land at a given map. 
According to Bonifant, Milnor and Sutherland~\cite{cm3,ParabolicGreen}
a parameter ray $\cR_\cU(\theta)$ with $\theta \in \QS$ lands at a parabolic or pcf map $f_0$ whose basic features are summarized in~\cite[Theorem~2.8]{cm3} which in our notation reads as follows:

\begin{theorem}[Milnor, Bonifant and Sutherland]
  \label{thr:bms}
  A parameter ray $\cR_\cU(\theta)$ with $\theta \in \QS$ lands at a map $f_0$.
  If $\theta'= \theta+1/3$ or $\theta-1/3$ is periodic, then $f_0$ is a parabolic map, the dynamical rays $ R_{f_0}(\theta+1/3), R_{f_0}(\theta-1/3)$ land
  in $\partial V(-a(f_0))$, and $R_{f_0}(\theta')$ lands at a parabolic periodic point.
  If $\theta+1/3$ and $\theta-1/3$ are strictly preperiodic, then $f_0$ is a
  pcf map and $ R_{f_0}(\theta+1/3), R_{f_0}(\theta-1/3)$ land at $-a(f_0)$.  
\end{theorem}

For our purpose we need to study in greater detail the supporting properties
of the rays $ R_{f_0}(\theta+1/3), R_{f_0}(\theta-1/3)$, the kneading words induced by these rays and
to determine the existence of certain parameter rays landing at $f_0$.
Our study relies on the notions of take-off arguments and their associated kneading word defined below.

\begin{definition}[Take-off argument]
  \label{def:take-off}
  Let $f_0$ be a map in $\cS_p$ with a parabolic periodic point $z_0$. Assume that $-a(f_0)$ is in the immediate basin of $z_0$ (i.e. $z_0 \in \partial V(-a(f_0))$).
We say that $\theta$ is a \emph{take-off  argument} for $f_0$ if
$R_{f_0}(\theta + 1/3)$ and $R_{f_0}(\theta - 1/3)$
are relatively supporting arguments for
$\partial V(-a(f_0))$ and one of these rays lands at $z_0$.

For a pcf map $f_0 \in \partial \cC(\cS_p)$
we say that \emph{$\theta$ is a take-off argument for $f_0$} if
the external rays $R_{f_0} (\theta \pm 1/3)$ land at $-a(f_0)$. Equivalently, if $R_{f_0}(\theta)$ lands at the cocritical point $2a(f_0)$.
\end{definition}

\begin{definition}[Take-off kneading]
  \label{def:take-off-word}
Given a parabolic (resp. pcf) map $f_0 \in \cS_p$  with take-off argument $\theta$,
let $X= \overline{V(-a(f_0))}$ (resp. $X=\{-a(f_0)\}$) and
$\Gamma = R_{f_0}(\theta-1/3)\cup X \cup  R_{f_0}(\theta+1/3)$.
Then $\C \setminus \Gamma$ consists of two connected components. Denote by $U_0$ the one containing $a(f_0)$ and by $U_1$ the other one. We say that
$$\kappa (f_0, \theta) = i_1 i_2 \dots i_{p-1} 0$$
is the \emph{take-off kneading word of $f_0$ associated to $\theta$} if $a_j(f_0) \in U_{i_j}$ for $j=1, \dots, p-1$.
\end{definition}

Parameter rays landing at pcf maps are completely described by~\cite[Theorem~5.2]{cm3}:

\begin{theorem}[Bonifant and Milnor]
  \label{thr:bm}
  If $\theta$ is a take-off argument of a pcf map $f_0 \in \cS_p$, then there exists  unique parameter ray $\cR_\cU(\theta)$ of some escape region $\cU$
  with argument $\theta$
landing at $f_0$.
\end{theorem}

In~\S~\ref{s:parabolic-landing}, we  analyze
in greater detail the case in which $f_0$ is parabolic where we reprove part of Theorem~\ref{thr:bms}.
The techniques employed are a combination of the original techniques of the Orsay Notes~\cite{OrsayNotes} (cf.~\cite{TanLeiTheme}) with elementary
properties of the limit of rational dynamical rays.
 The aforementioned elementary properties are deduced in \S~\ref{s:limit-rays}
 following ideas from~\cite{kiwi-1997}. Ideas that are 
 related to those employed by Petersen and Ryd~\cite{PetersenRyd} and by Bonifant, Milnor and Sutherland~\cite{ParabolicGreen}.
 In
 \S~\ref{s:pcf-landing} we discuss how to deduce from the above theorems
the statements that we need about parameter rays landing at pcf maps.

\subsection{Limit of rational rays}
\label{s:limit-rays}
Below we consider the behavior of $R^*_f (t)$ as $f \in \cR_\cU(\theta)$
approaches a map $f_0$ in the connectedness locus.
In this context,
$\limsup R^*_f (t)$ is defined as the set formed by all $z \in \C$ such that
  for every neighborhood $W$ of $z$ and any neighborhood $\cW$ of $f_0$
  we have that 
  $R^*_f(t) \cap W \neq \emptyset$ for some $f \in \cR_\cU(\theta) \cap \cW$.
  The set $\limsup I_{f,a_k}(t)$ is defined similarly.
  
\begin{lemma}
  \label{l:limsup}
  Let $\cR_\cU(\theta)$ be a parameter ray of an escape region $\cU$ such that $\theta \in \QS$. Then $\cR_\cU(\theta)$ lands at a parabolic or pcf map $f_0$.
  Moreover, the following statements hold:
  \begin{enumerate}
  \item If $t \in \QS$ is periodic under $m_3$ then $\{ u_0 \} =
  J(f_0) \cap \limsup R^*_f (t)$ where $u_0$ is the landing point of
  $R_{f_0}(t)$.
\item For all $k$, we have that $\{ v_0 \} = J(f_0) \cap \limsup I_{f,a_k}(0)$
  where $v_0$ is the landing point of $I_{f_0,a_k}(0)$.
\item If  $t \in \QS$ is strictly preperiodic and $w_0 \in J(f_0) \cap \limsup R^*_f (t)$, then $w_0$ is strictly preperiodic.
  \item  If $\theta' = \theta +1/3$ or $\theta-1/3$ is periodic, then  $f_0$ is a parabolic map and $R_{f_0}(\theta')$ lands at a parabolic periodic point.
  \end{enumerate}
\end{lemma}

Under the assumptions of the previous lemma, Bonifant, Milnor and Sutherland~\cite{ParabolicGreen}  show that if
$ \theta +1/3$ and $\theta-1/3$ are strictly preperiodic, then
the landing point of $\cR_\cU(\theta)$ is a pcf map. Moreover, when
$\theta' = \theta +1/3$ or $\theta-1/3$ is periodic they give
a complete description of $ \limsup R^*_f (\theta')$.
Although we provide a self-contained prove here,
it is not difficult to deduce the lemma directly from these facts.

\begin{proof}
To show that $\cR_\cU(\theta)$ lands, consider an accumulation point $g$ 
as we approach the connectedness locus along $\cR_\cU(\theta)$.
Denote by $w$ the landing point of $R_{g} (\theta+ 1/3)$.
Then, for some $k \ge 0$,
$g^k(w)$ is a parabolic periodic point or a critical point; for otherwise
$w$ would be an eventually repelling periodic point (maybe periodic) with a critical point free orbit
and for every $f$ close to $g$
the corresponding ray $R_{f} (\theta+ 1/3)$ would be smooth which is not the case for $f \in \cR_\cU(\theta)$ (e.g. see~\cite{GoldbergMilnor}).
Thus, $g$ is a parabolic or  pcf map.
Since the accumulation set of $\cR_\cU(\theta)$ in $\cC(\cS_p)$ is connected and there are only countably many parabolic and pcf maps in $\cS_p$,
we have that $\cR_\cU(\theta)$ lands, say at $f_0$.
It follows as well that if $\theta+ 1/3$ or $\theta- 1/3$ is periodic, the landing point of the corresponding $f_0$-ray is a parabolic periodic point. That is, (4) holds.

Now given $t \in \QS$ we study $X:=J(f_0) \cap \limsup R^*_f (t)$.
Provided that $t$ is periodic of period $q$,
we start by showing that all the elements of $X$ are periodic of period dividing $q$. Afterwards, taking into account that $f_0$ has only one free critical point we establish (1).

For $f \in R_\cU(\theta)$  the B\"ottcher coordinate $\phi_f : V_f^*(\infty) \to \CDC$ has as image a star-like domain $U^*_f$ around infinity obtained from $\CDC$  by removing the ``needles'' $$]1,|\phi_f(-a(f))|] \exp(2 \pi i(\theta\pm 1/3))$$
as well as all their iterated preimages under $z \mapsto z^3$. Recall that we let $\rho(s)$ be such that
the maximal radial line  with argument $s$ contained in
$U^*_f$ is $]\rho(s),\infty[ \exp(2 \pi i s)$.
Given $R >1$, it is not difficult to show that there exists $\delta >0$ such that
for all $f \in \cR_\cU(\theta)$ with $|\phi_f(2 a(f))|<R$,
the sector
$$S_t= \{ z \in \CDC : |\arg(z- \rho(t) \exp(2 \pi i t)) - 2 \pi t| < \delta \}$$
is contained in $U^*_f$. Hence, for $ r > \rho(t)$, the hyperbolic distance in $S_t$ from $z = r \exp (2 \pi i t)$ to $z^{3^q}$ is uniformly bounded above by a constant $C$ independent of $r >  \rho(t)$.
Denote by $d_f$ the hyperbolic metric  of $V_f(\infty)$.
The standard comparison between hyperbolic metrics yields that the hyperbolic distance in $\phi_f^{-1}(S_t)$ bounds from above $d_f$.
Then for all $f \in \cR_\cU(\theta)$
 and all $z \in R^*_f(t)$, we have
  $d_f(z, f^q(z)) < C$ for some constant $C$ independent of $f$. 
  Now if $ z_0 \in X $
  then there exist $f_n \to f$ and $w_n \to z_0$ such that $w_n \in R^*_{f_n}(t)$. Since $\limsup J(f_n) \supset J(f_0)$, the Euclidean distance $\varepsilon_n$ from $w_n$ to $J(f_n)$ converges to
  $0$, which together
  with $d_{f_n}(w_n,f^q_n(w_n)) <C$, implies that $|w_n - f^q_n(w_n)| \to 0$, using the comparison between Euclidean and hyperbolic metrics. Thus, $f_0^q(z_0) =z_0$. That is, every element of $X$ is periodic of period dividing $q$.

  Now let $u_0$ be the landing point of $R_{f_0}(t)$.
  When $u_0$  is a repelling periodic point,  $R_{f}(t)$ is a smooth ray
for all $f$ in a neighborhood of $f_0$ and $R_{f}(t)$ converges uniformly
to $R_{f_0} (t)$, so (1) holds in this case (cf.~\cite{GoldbergMilnor}).
When $u_0$ is a parabolic periodic point, 
 consider the connected set
 $Y:=K(f_0) \cap \limsup R^*_{f}(t)$ which contains
 $u_0$ and $X$. 
  Let $P$ be the union of the periodic Fatou components having $u_0$ in its boundary.
  Taking into account that the return map to any parabolic periodic Fatou component has degree $2$, we conclude that $u_0$ is the unique periodic point in $\partial P$ of period dividing $q$
  and  $X \subset Y \cap \partial{P} = \{u_0\}$. Hence we have proven (1).

  The second assertion of the Lemma follows along similar lines but
  considering $\varphi_{f,a_k}$ instead of $\phi_f$.



  For assertion (3), it is sufficient to consider the case in which
  $t \in \QS$ is not periodic and $3t$ is periodic.
  Let $w \in J(f_0) \cap \limsup R^*_f (t) $. We claim that
  $w$ is strictly preperiodic. Proceeding by contraction,
  let us suppose that $w$ is periodic.
  Let $f_n \in \cR_\cU(\theta)$
  and $w_n \in R^*_{f_n} (t)$ be such that $f_n \to f_0$ and $w_n \to w$.
  Let $t'$ be the periodic preimage of $3t$ and $w'_n \in R^*_{f_n} (t')$
  be such that $f_n(w_n') = f_n(w_n)$. Then both $w_n'$ and $w_n$ converge to the unique periodic preimage $w$ of $f_0(w)$. Thus $f_0$ is not locally injective around $w$ for otherwise $f_n$ would  also be, but $w_n' \neq w_n$. Hence $w$ is a periodic critical point of $f_0$ in $J(f_0)$ which is impossible. Therefore $w$ is strictly preperiodic and the lemma follows.  
\end{proof}



\subsection{Parabolic landing and take-off theorems}
\label{s:parabolic-landing}

\begin{theorem}[Periodic Landing Theorem] 
\label{thr:arrival-periodic-connection}
Consider a parameter ray $R_\cU(\theta)$ of an escape region $\cU$
such that $\theta'=\theta+1/3$ or $\theta-1/3$ is periodic.
Then  $R_\cU(\theta)$ lands at a parabolic map with take-off argument $\theta$ and $\kappa(f_0,\theta) = \kappa (\cU)$.
Moreover, if $R^\sigma_f(\theta')$ is a periodic ray  connection  between
  $-a(f)$ and $a_k(f)$ for $f\in R_\cU(\theta)$, then $\theta'$ is a relatively supporting  argument for $V(a_k(f_0))$.
\end{theorem}

A similar result is proven by Bonifant and Milnor~\cite[Theorem 2.8]{cm3}.
Under the assumptions of the theorem they prove that 
$R_\cU(\theta)$ lands at a parabolic map $f_0$ such that the rays
$R_{f_0}(\theta+1/3)$ and $R_{f_0}(\theta-1/3)$
land  in $\partial V(-a(f_0))$. For our purpose we will also need
to study the supporting properties of these rays. 

\smallskip
Our proof relies on the following:
\begin{figure}[h]
\begin{minipage}[c]{.48\linewidth}
\begin{center}
\fbox{\includegraphics[width=6.3cm]{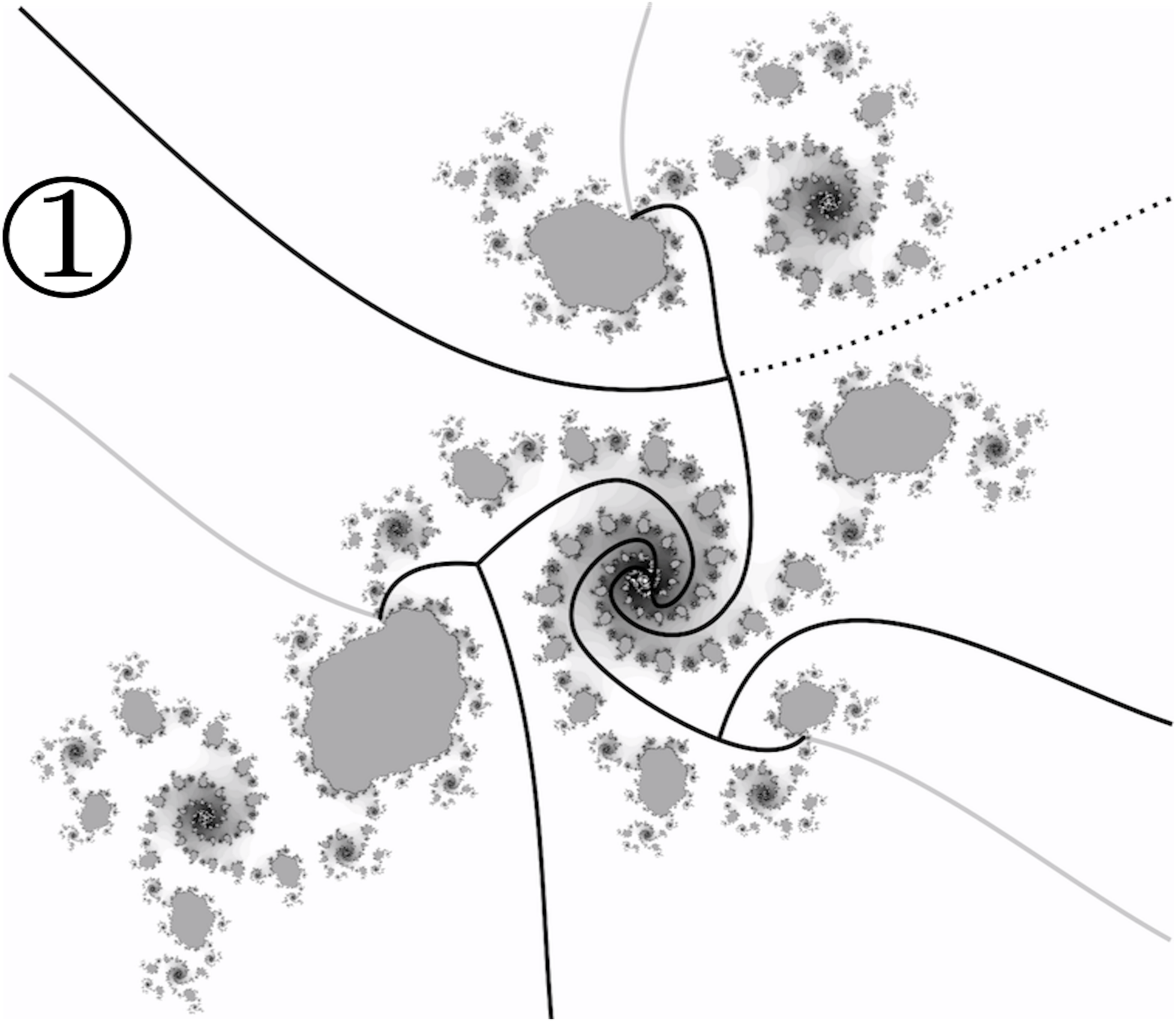}  }
\end{center}
\end{minipage}
\begin{minipage}[c]{.48\linewidth}
\begin{center}
\fbox{\includegraphics[width=6.5cm]{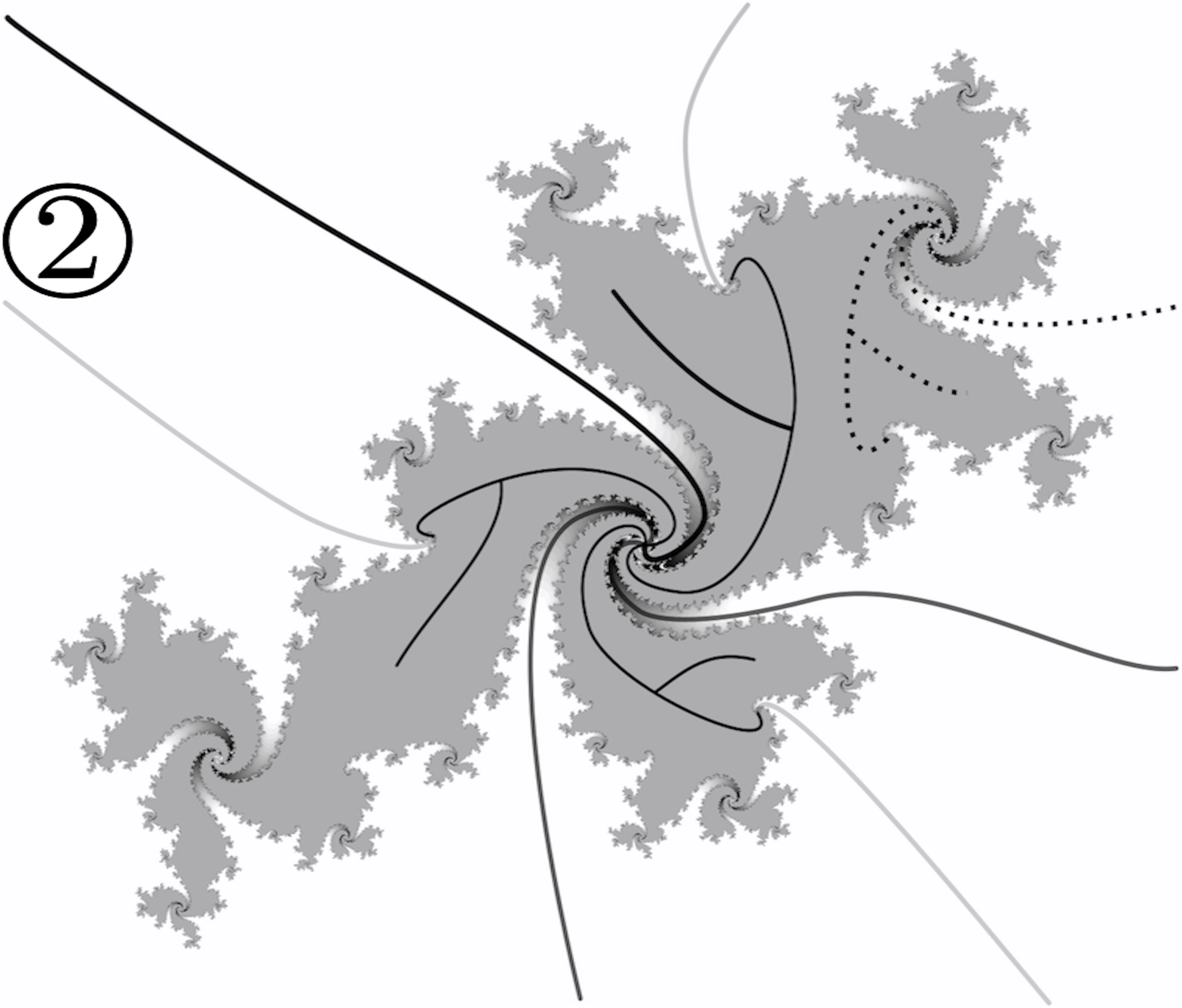} }
\end{center}
\end{minipage}
\begin{minipage}[c]{.48\linewidth}
\begin{center}
\fbox{\includegraphics[width=6.3cm]{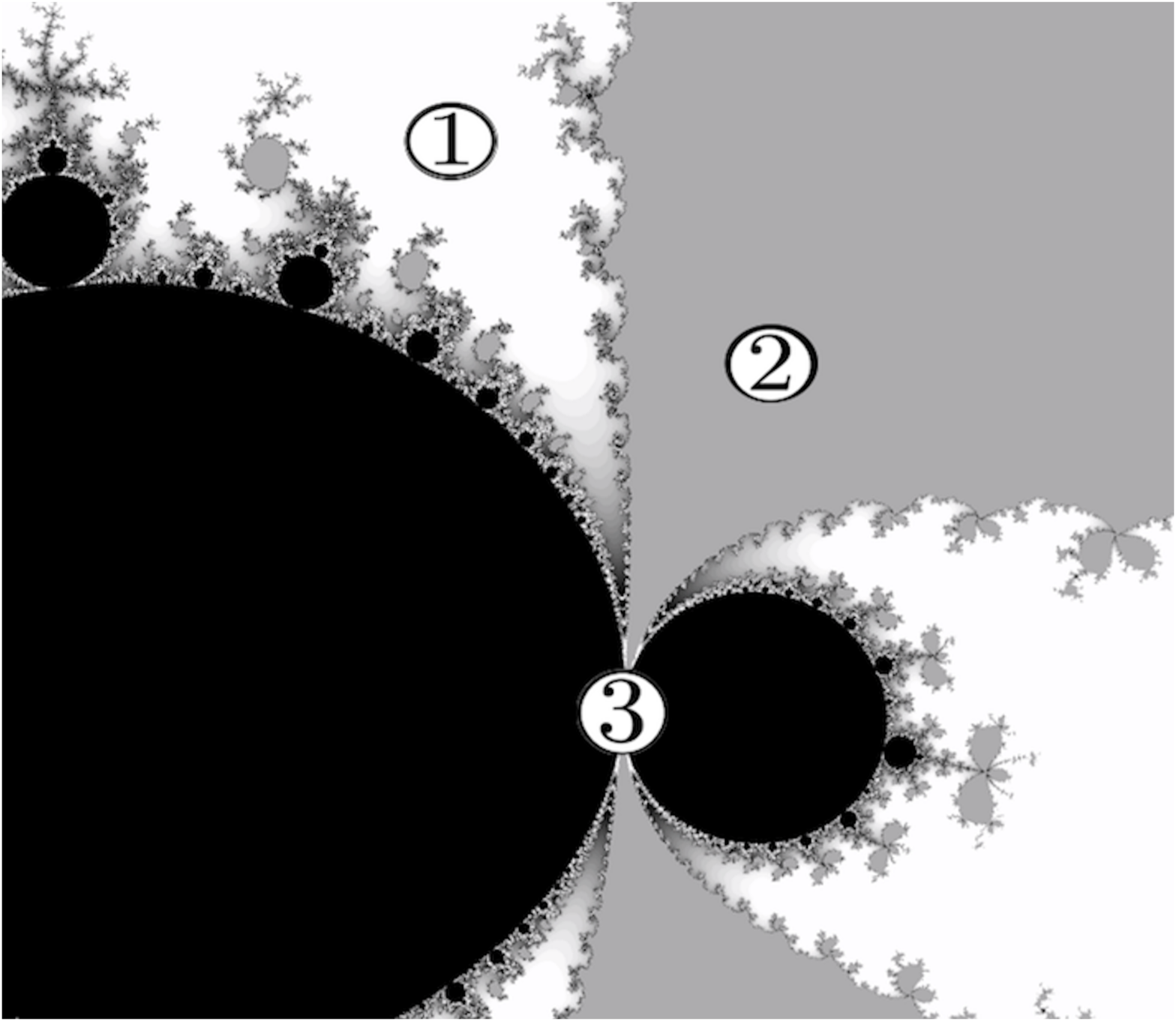}  }
\end{center}
\end{minipage}
\begin{minipage}[c]{.48\linewidth}
\begin{center}
\fbox{\includegraphics[width=6.5cm]{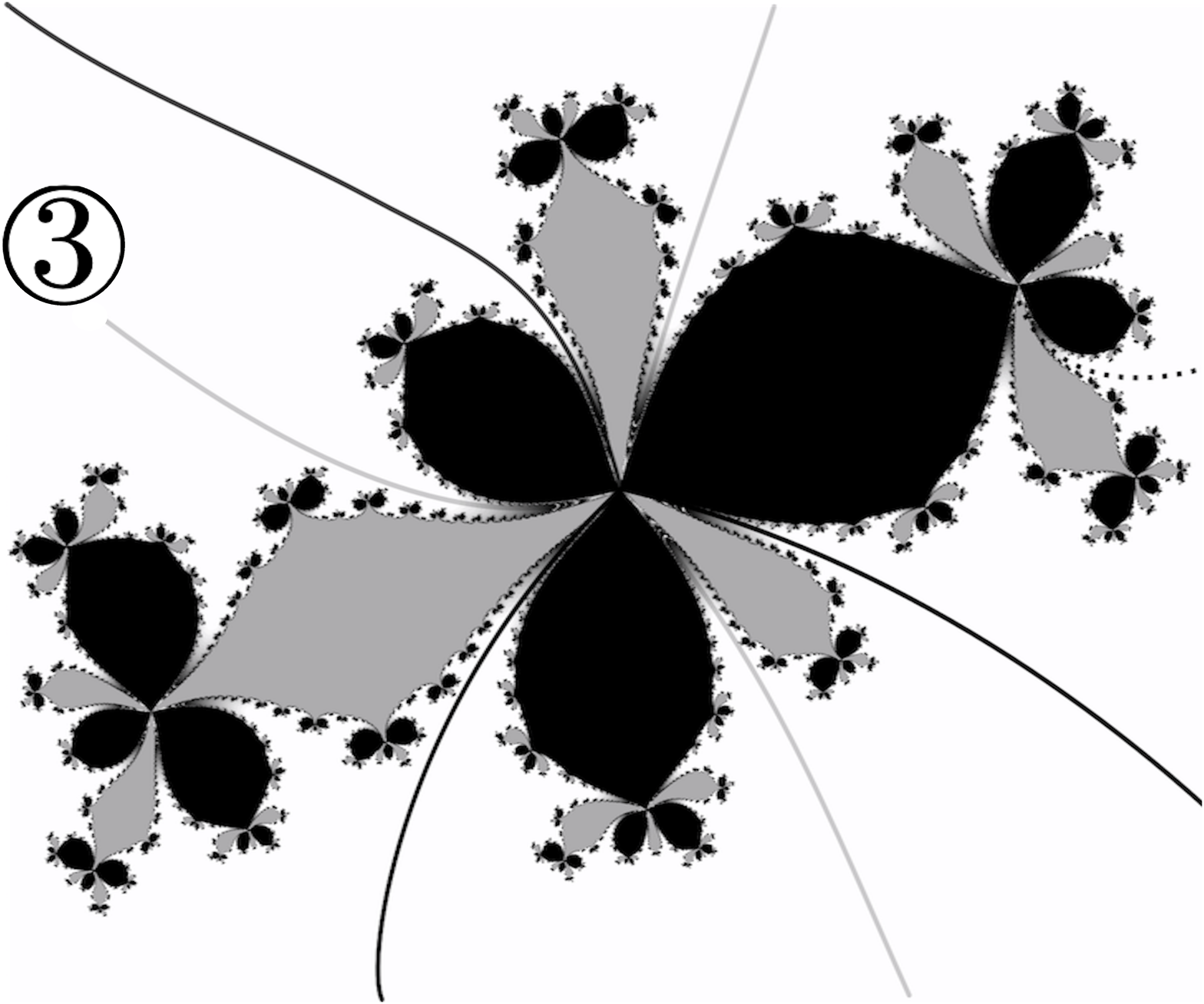}  }
\end{center}
\end{minipage}
\caption{
  The lower left figure is a detail of $\cS_3$ with an external parameter ray 
  and a $0$-internal ray of a type $B$ component, both landing at a parabolic map
  \numcirc{3}. 
  When converging to \numcirc{3}, along each one of the two parameter rays, a fixed point with combinatorial rotation number $1/3$ and
  a period $3$ periodic orbit coalesce.
   Relevant external rays are illustrated. For \numcirc{2},
  the $0$-internal rays are also drawn.
}\label{NoTriv}
\end{figure}

\begin{lemma}
  \label{l:limsupperiodic}
   Consider $\theta \in \QS$ such that $\theta'= \theta+1/3$ or $\theta-1/3$ is periodic of exact period $q$.
   Let $\cR_\cU(\theta)$ be a parameter ray of an escape region $\cU$
   landing at $f_0$. 
 For $f\in\cR_\cU(\theta)$ denote by $z^\pm(f)$ the landing point of $R_{f}^\pm(\theta')$. Then all of the following statements hold:
  \begin{enumerate}
  \item $R_{f_0}(\theta')$ lands at a parabolic periodic point $z_0\in \partial V_{f_0} (-a(f_0))$ of  some period $\ell$ that divides $q$.
  \item $z^\pm(f) \to z_0$ as $f \in \cR_\cU(\theta)$ approaches $f_0$.
\item The periods of $z^+(f)$ and $z^-(f)$ are $\ell$ and $q$, maybe not respectively.
\item A ray $R_{f_0} (t)$ lands at $z_0$ if and only if a ray with argument $t$
  lands at the iterates of $z^+(f)$ or $z^-(f)$ under $f^{\ell}$.
\item $R_{f_0}(\theta+1/3)$ and $R_{f_0}(\theta-1/3)$ land at $\partial V_{f_0} (-a(f_0))$.
\item $f_0$ has take-off argument $\theta$.
\item  $\kappa(f_0,\theta) = \kappa (\cU)$.
  \end{enumerate}
\end{lemma}

\begin{proof}
  As a direct consequence of Lemma~\ref{l:limsup}, the parameter ray $\cR_\cU(\theta)$ lands at a map
  $f_0$ such that the landing point $z_0$ of  $R_{f_0}(\theta')$ is a parabolic periodic point, say of period $\ell$.
  Moreover, since $-a(f_0)$ belongs to the connected
  set $ K(f_0) \cap \limsup R^*_f(\theta') $ whose unique Julia set point
  is  $ z_0$, we have that $z_0  \in \partial V_{f_0} (-a(f_0))$. 
  That is, (1) holds.


There is only one cycle of repelling petals at $z_0$ since $-a(f_0)$ is the unique free critical point of $f_0$. Hence, for $f$ converging to $f_0$, there exist
one period $q/\ell$ orbit of $f^\ell$ and one fixed point of $f^\ell$  that converge to $z_0$.  
Note that possibly $q=\ell$, in this case two distinct fixed points of $f^q$ converge to $z_0$. See Figure \ref{NoTriv} for a case in which $\ell =1$ and $q=3$.
No other periodic point can converge to $z_0$ (that is, given a period $N$,
all other periodic points of period at most $N$ are bounded away from $z_0$ for $f$ close to $f_0$).

Let us prove $z^\pm(f) \to z_0$. Denote by $d_f$ the hyperbolic metric
in the basin of infinity of $f$ and by $\H_- \subset \C$  the lower half plane.
It is convenient to consider the universal covering $h:\H_-\to \C\setminus \overline{\D}$ given by $z\mapsto \exp(i {z})$. Let $\widetilde{U}_f := \pi^{-1} (U^*_f)$.
Given $R>1$, there exists $\delta >0$ such that for all $f \in \cR_\cU(\theta)$
with $|\phi_f(-a(f))| <R$ 
we have that $\widetilde{U}_f$ contains  sectors
$S^+= \{ z \in \H_- : \pi/2 <\arg (z-2 \pi \theta')<(\pi/2) + \delta \}$ and
 $S^-= \{ z \in \H_- :
 (\pi/2) -\delta <\arg (z-2 \pi \theta')<\pi/2 \}$. Thus, for $\alpha \in 
 ]0,\delta[$ we may define
 $$\gamma^\pm_f(\alpha):=\phi_f^{-1}(h(2 \pi \theta'- ]0,\infty[ \cdot i \cdot \exp(\pm i \alpha))).$$
For each $\alpha$, the arc $\gamma^\pm_f(\alpha)$ is $f^q$ invariant and
$d_f(z,f^q(z))$ is uniformly bounded for all $z \in \gamma^\pm_f(\alpha)$
and all $f \in \cR_\cU(\theta)$
with $|\phi_f(-a(f))| <R$.
In particular, $\gamma^\pm_f(\alpha)$ limit towards a periodic point $w^\pm(\alpha)$ of period dividing $q$ as it approaches the Julia set.
It follows that $w^+(\alpha)$ (resp. $w^-(\alpha)$) is independent of $\alpha \in ]0,\delta[$ (e.g. see~\cite{DynInOne}[Corollary 17.10]). Therefore,
$w^+(\alpha) = z^+(f)$ and $w^-(\alpha) = z^-(f)$.
Similarly, we may introduce arcs $\gamma^\pm_{f_0}(\alpha)$ for $f_0$ and
conclude that their limit point in the Julia set is $z_0$. 
We can argue as in the proof of Lemma \ref{l:limsup} but  considering $\gamma^\pm_f(\alpha)$ instead of  $R^*_f (t)$  and prove that 
$\{ z_0 \}=J(f_0) \cap \limsup \gamma^\pm_{f}(\alpha)$.
Since 
any limit point of $z^\pm(f)$  as $f \in \cR_\cU(\theta)$ approaches $f_0$ lies in this set, we deduce that $z^\pm(f) \to z_0$, which concludes the proof of (2).

 Taking into account that only two periodic orbits under $f^\ell$ can converge to $z_0$ and that their $f^\ell$-periods are $1$ and $q/\ell$, assertion (3) follows. 

 The $f$-period of $z_0$ is $\ell$, thus a ray $R_{f_0}(3^n \theta')$ lands at $z_0$ if and only $\ell$ divides $n$. That is, (4) holds for arguments in the orbit of $\theta'$. Suppose that $t$ is a period $q$ argument not in the orbit of $\theta'$.
  As $f \in \cR_\cU(\theta)$ converges to $f_0$, from Lemma~\ref{l:limsup},
 $R_{f_0}(t)$ lands at $z_0$ if and only if the landing point
 $z(f)$ of $R_f(t)$ converges to $z_0$. This is equivalent to
 $z(f)$ being in the $f^\ell$-orbit of $z^+(f)$ or $z^-(f)$, since these are the only  periodic points that converge to $z_0$. That is we have proven (4).

 Let $\theta'' = \theta+1/3$ or $\theta-1/3$ be such that $\theta''$ is strictly preperiodic. We claim that $R_{f_0}(\theta'')$ lands at $\partial V(-a(f_0))$.
Denote by $w_0$ the unique non-periodic preimage of $f_0(z_0)$ in $\partial V(-a(f_0))$.
Let $P'$ be the union of the Fatou components that have $w_0$ in their boundaries. No other preperiodic boundary point of these Fatou components maps onto $f_0(z_0)$. 
It follows that $K(f_0) \cap \limsup R^*_{f} (\theta'')$
is contained in $\overline{P'}$. 
Therefore, $J(f_0) \cap \limsup R^*_{f} (\theta'') = \{ w_0 \}$ and
$R_{f_0} (\theta'')$ lands at $w_0 \in \partial  V(-a(f_0))$ which yields (5).
 
To prove (6) we must show that $\theta'$ relatively supports $V(-a(f_0))$.
From (4) we have a good description of the rays landing at $z_0$.
By (3) either $z^+(f)$ or $z^-(f)$ has period $q$.
For simplicity let us assume that $z^+ (f)$ has period $q$. 
The period of $\theta'$ is also $q$.
Therefore, the ray $R^+_f(\theta')$ is the unique right limit ray with argument in the orbit of $\theta'$ landing at $z^+(f)$.
 Moreover, since $z^-(f)$ and $z^+(f)$ belong to distinct orbits, no left
ray with argument of the form $3^n \theta'$ can land at $z^+(f)$.
In the dynamical plane of $f \in \cR_\cU(\theta)$ consider the graph $\Gamma$
formed by the flow lines $R^*_f(\theta')$, $R^*_f(\theta'')$ and $-a(f)$.
The point  $z^+(f)$ belongs to the connected component of $\C \setminus \Gamma$ containing rays with arguments
in $]\theta',\theta''[$ while 
$z^-(f)$ and hence the rest of the $f^\ell$-orbit of $z^+ (f)$ lies in the other component. It follows from (4) that every ray of $f_0$ with argument $t$ in the orbit of $\theta'$ landing at $z_0$ is such that $t \in ]\theta'',\theta']$ which together with (5) shows that $\theta'$ is relatively left supporting.
Similarly, if we assume that $z^-(f)$ has period $q$, then
$\theta'$ is relatively right supporting. That is, we have proven that
$\theta$ is a take-off argument for $f_0$.
 
To prove (6), let $\kappa(f_0,\theta) = \iota_1 \dots \iota_{p-1} 0$.
If $\iota_j =0$ (resp. $1$),  consider $t_j \in ]\theta-1/3,\theta+1/3[$ (resp. $]\theta+1/3,\theta-1/3[$) such that  the ray
$R_{f_0}(t_j)$ lands at a repelling periodic point $u_0$ in $\partial V(a_j(f_0))$. Since $u_0$ is at the same time the landing point of some periodic internal ray $I_{f_0,a_j}(s_j)$, we have that the analytic continuation $u (f)$ of $u_0$ lies in
$\partial V(a_j(f))$ for $f$ sufficiently close to $f_0$. Also the external ray
$R_f(t_j)$ must be smooth and land at $u(f)$ (e.g.~\cite{GoldbergMilnor}).
We conclude that the kneading words $\kappa(f_0,\theta)$ and $\kappa(\cU)$ coincide.

\end{proof}

\begin{proof}[Proof of Theorem~\ref{thr:arrival-periodic-connection}]
  Suppose that
  $R^\sigma_f(\theta')$ is a periodic ray connection between $-a(f)$ and $a_k(f)$ for all
  $f \in \cR_\cU(\theta)$. To fix ideas we assume that $\sigma=-$.
  In the notation of the previous lemma,
  $z_0 \in \partial V(a_k(f_0))$
  since $I_{f,a_k}(0)$ lands at $z^-(f)$ and, from Lemma \ref{l:limsup}, we 
  have that $I_{f_0,a_k}(0)$ lands at $z_0 = \lim z^-(f)$.
  Hence, the landing point $z_0$ of $R_{f_0}(\theta')$ is in $\partial V(a_k(f_0))$. The theorem follows once we have proven that $R_{f_0}(\theta')$ is a relatively supporting ray for $V(a_k(f_0))$. 

  In view of (4) of the previous lemma, $R_{f_0}(3^n \theta')$ lands at $z_0$
  if and only $\ell$ divides $n$.  Consider a reference
  smooth periodic ray $R_f(t)$ landing at some point
  $u (f) \in \partial V(a_k(f))$ where
  $u (f) \neq z^-(f)$. Note that $R_{f_0}(t)$ lands at a repelling periodic point $u(f_0) \in \partial V(a_k(f_0))$ such that $u(f_0) \neq z_0$.
  If the period of $z^-(f)$ is also $\ell$, then 
  $]t,\theta'[$ or $]\theta',t[$ is disjoint from $\{ 3^{n\ell} \theta' : n \ge 1\}$ according to whether $\theta'$ is a right or left relatively
  supporting argument of $V(a_k(f))$.
  If the period of $z^-(f)$ is  $p$ and, therefore, $z^{+}(f)$ has period
  $\ell$, then $]t,\theta'[$  and 
  $\{ 3^{n\ell} \theta' : n \ge 1\}$ are disjoint
  since 
   $z^{+}(f)$ lies in $Z_f(\overline{V(a_k(f))},\theta', t)$ and all the rays $R^{+}_f(3^{n\ell} \theta')$ land at $z^{+}(f)$.
In all cases,  $R_{f_0}(\theta')$ relatively supports $V(a_k(f_0))$.



 \end{proof}


\begin{theorem}[Parabolic Take-off Theorem]
  \label{thr:take-off-parabolic}
  Let $f_0 \in \cS_p$ be a parabolic map with take-off argument $\theta$.
  Then there exists an escape region $\cU$
  with $\kappa(\cU) = \kappa(f_0,\theta)$
  and a parameter ray $R_\cU (\theta)$ that lands at $f_0$.
\end{theorem}

Our proof of this theorem follows the original ideas introduced in the Orsay Notes (\cite{OrsayNotes}) to characterize which external rays of the Mandelbrot set land at a given parabolic parameter.
Our exposition employs perturbed Fatou coordinates in the spirit of Tan Lei's article~\cite{TanLeiTheme}. The existence of appropriate perturbed Fatou coordinates
is guaranteed by~\cite{ShishikuraHausdorff,ShishikuraTheme} (cf.~\cite{Oudkerk}).  We employ the Walz Theorem as stated and proved in~\cite{OrsayNotes}.

\begin{proof}
  Let $f_0 \in \cS_p$ be a  map such that $z_0 \in \partial V(-a(f_0))$
is a parabolic periodic point of period $q_0$ and multiplier a primitive $q_1$ root of unity.
Let $R_{f_0} (\theta')$ be a relatively supporting ray for $V(-a(f_0))$ landing at $z_0$
where $\theta'= \theta+1/3$ or $\theta-1/3$.
It follows that $\theta'$ has period $q=q_0q_1$.

From the fact that $f_0$ has one cycle of attracting petals,
maybe after passing to a branched covering $\lambda \mapsto f_\lambda$ from a disk $\Lambda \subset \C$ centered at $\lambda =0$
onto a neighborhood of $f_0$ in $\cS_p$,
we may assume that for all $\lambda \in \Lambda \setminus \{0\}$
the map $f_\lambda$ has a periodic 
point $z(\lambda)$ of  period $q_0$ and distinct
periodic points  $\zeta_1(\lambda),
\dots, \zeta_{q_1}(\lambda)$ of period $q$ depending analytically on $\lambda$ such that at $\lambda=0$ all of them converge to $z_0$. For convenience we change coordinates in the dynamical plane, via a translation $z \mapsto z + (z(\lambda)-z_0)$,
and  abuse of notation by also calling $f_\lambda$
the resulting family of monic critically marked cubic polynomials.
With this normalizations $z_0$ is a fixed point of $f_\lambda$ for all $\lambda \in \Lambda$.
The free critical point of $f_\lambda$ will be denoted by $\omega_\lambda$.
We continue denoting by  $\zeta_1(\lambda),
\dots, \zeta_{q_1}(\lambda)$ the period $q$ points of $f_\lambda$ close to $z_0$.
It is useful to also consider
 the $q$-th  iterate of $f_\lambda$ by introducing
$g_\lambda = f_\lambda^q$ and setting $d= 3^q$.

Let $L_-$ be the attracting direction 
at $z_0$ under iterations of $g_0$ corresponding to the Fatou component $V(\omega_0)$ and
let $L_+$ be the repelling direction  corresponding to the external ray $R_{f_0} (\theta')$. Since $R_{f_0} (\theta')$ is relatively supporting for $V(\omega_0)$ and there is only one cycle of attracting/repelling directions under
$f_0^{q_0}$, the directions $L_+$ and $L_-$ are consecutive. That is, their angle at $z_0$ is
$\pi/q_1$. Denote by $L$ the bisector of $L_+$ and $L_-$. That is,
$L$ is  a half-line emanating from $z_0$ making an angle of $\pi/2q_1$ with
both directions $L_+$ and $L_-$.

Given a small sector $\hat{L}$ based at $z_0$ around $L$, maybe after shrinking $\Lambda$ and relabeling $\zeta_j(\lambda)$, there exists a parameter space 
sector
$\Lambda' \subset \Lambda$ based at $\lambda = 0$ such that $\zeta_1(\lambda) \in \hat{L}$  and $\pi/4 < |\arg((g_\lambda)'(z_0) - 1)| < 3\pi/4$, for all $\lambda \in \Lambda'$ (e.g. see~\cite[Expos\'e XI]{OrsayNotes}).

According to Shishikura (see~\cite[Proposition 3.2.2]{ShishikuraTheme}
 and \cite[Appendix A.5]{ShishikuraHausdorff}),
 given a neighborhood $U$ of $z_0$, after shrinking $\Lambda'$ if necessary,
for all $\lambda \in \Lambda'$,
 we have that  $U$ contains a Jordan domain $S_{\pm,\lambda}$
 bounded 
by a closed arc $\ell_\pm$ and its image $g_\lambda (\ell_\pm)$ 
such that the endpoints of
 $\ell_\pm$ are $z_0$ and $\zeta_1(\lambda)$, and $S_{-,\lambda} \cap S_{+,\lambda}
= \ell_\pm \cap g_\lambda (\ell_\pm) = \{z_0, \zeta_1(\lambda) \}$.
Moreover, for every $z \in S_{-,\lambda}$ there exists  $n$ (depending on $\lambda$)
such that $g^n_\lambda (z) \in S_{+,\lambda}$ and, for the smallest such $n$, we have that  $g^j_\lambda(z)$ belongs to the Jordan domain $\Omega_\lambda$ enclosed by $\ell_- \cup g_\lambda(\ell_+)$ for all $0 \le j \le n$. 
Also, $\ell_{\pm,\lambda}$ converges to an arc $\ell_{\pm,0}$ as $\lambda \in \Lambda'$ converges to $\lambda=0$. The arc $\ell_{-,0}$ bounds an attracting petal
$\Omega_{0,-}$ in the direction
$L_-$ and  $\ell_{+,0}$ a repelling petal $\Omega_{0,+}$ in the direction $L_+$. We denote by $S_{\pm,0}$ the Jordan domain bounded by $\ell_{\pm,0}$ and $g_0 (\ell_{\pm,0})$.

Let $m \ge 1$ be the smallest positive integer such that
$g_0^{m}(\omega_0) \in S_{-,0}$. 
This integer exists,
since $g_0^{n}(\omega_0)$ converges to $z_0$ tangentially to $L_-$.
Shrinking $U$ we may assume that $g_0^{m}(\omega_0)$ lies in a small sector based at $z_0$ around $L_-$. 

Recall from \S\ref{s:external-rays} that we denote the maximal
$\nabla |\phi_{f_\lambda}|$ flow line asymptotic to the direction
$\theta'$ at infinity by $R^*(\theta')$. This flow line is naturally
parametrized via $|\phi_{f_\lambda}|$ by $]\rho_\lambda,\infty[$, for some $\rho_\lambda >1$. Here it
will be convenient to consider the continuous parametrization
$$\ray_\lambda:[\log|\rho_\lambda|, \infty[ \to
\overline{R^*_{f_\lambda}(\theta')}$$ such that
$r = \log |\phi_{f_\lambda}(\ray_\lambda (r))|$ for all
$r > \log|\rho_\lambda|$. Sometimes we abuse of notation and simply write $\ray_\lambda$ instead of $\overline{R^*_{f_\lambda}(\theta')}$.

The ray $R_{f_0}(\theta')$ lands at $z_0$
tangentially to $L_+$. Thus we may also assume
that $R_{f_0}(\theta') \cap \Omega_{0,+}$ is contained in a small sector
based at $z_0$ around
$L_+$. Let $r_0 > 0$ be such that $\ray_0 (r_0) \in S_{+,0}$.
 Then according to the Walz Theorem~\cite[Expos\'e XI]{OrsayNotes}
there exists $N_0$ with the property that for all $n \ge N_0$, we can find
$\lambda_n \in \Lambda'$ 
such that
$$g^{n}_{\lambda_n} (g^{m}_{\lambda_n} (\omega_{\lambda_n})) =
\ray_{\lambda_n} (r_0).$$
Moreover, $\lambda_n \to 0$ as $n \to \infty$.

Along the lines of the Orsay Notes (cf.~\cite{TanLeiTheme}), two claims yield a proof for the take-off theorem:

\begin{itemize}
\item[Claim 1.] 
  $g_{\lambda_n}^{m} (\omega_{\lambda_n}) = \ray_{\lambda_n}  (r_0/d^{n})$.
\item[Claim 2.]  $\omega_{\lambda_n} = \ray_{\lambda_n}
  (r_0/d^{n+m})$.
\end{itemize}

\noindent
\emph{Proof of Claim 1.}
For $n$ sufficiently large, by analytic continuation, we may find an arc
$\gamma: [r_0/d^{n},r_0/d^{n-1}] \to \Omega_{\lambda_n} $ connecting
$g^{m}_{\lambda_n} (\omega_{\lambda_n})$ with
$g^{m+1}_{\lambda_n} (\omega_{\lambda_n})$ such that
$g_{\lambda_n}^{n-1}(\gamma(r)) = \ray_{\lambda_n} (r d^{n-1})$.
By induction, for
$j= 1, \dots, n$, it follows that
$g_{\lambda_n}^{n-j}(\gamma(r)) = \ray_{\lambda_n}(r d^{(n-j)})$.
In particular, when $j=n$ and $r = r_0$ we obtain Claim 1.

\smallskip
\noindent
\emph{Proof of Claim 2.}
To prove the second claim we strongly employ the fact that the parabolic basin
contains a unique critical point.
Indeed, observe that the 
attracting petal  $g_0(\Omega_{0,-})$ is free of critical values
of $g^{m}_0$. Hence,
arcs $\beta$  connecting
$g^{m+1}_0(\omega_0)$ to $g^{2m+1}_0(\omega_0)$ contained in $g_0(\Omega_{0,-})$
lift, under $g^m_0$, to  arcs $\tilde{\beta}$ connecting some point $w$
to $g^{m+1}_0(\omega_0)$ where $w$ is independent of the path $\beta$
and the choice of the petal $g_0(\Omega_{0,-})$, modulo homotopy relative to the  critical values of $g^{m}_0$. 
The action of $g_0$ in $V(\omega_0)$ is conformally conjugate to the map $B: z \mapsto (3z^2+1)/(z^2+3)$ acting on $\D$.
Since the critical point $z=0$ of $B$ in $\D$ has forward orbit contained in $]0,1[$,
it is not difficult to conclude that $\tilde{\beta}$ connects the critical
value $g_0(\omega_0)$ with $g^{m+1}_0(\omega_0)$. 
By continuity, for $n$ large enough, 
the portion $\beta_n$ of  $\ray_{\lambda_n}$ connecting
$g^{m+1}_{\lambda_n}(\omega_{\lambda_n})$ to $g^{2m+1}_{\lambda_n}(\omega_{\lambda_n})$, lifts under $g_{\lambda_n}^m$ to an arc
$\tilde{\beta}_n$, contained in $\ray_{\lambda_n}$, connecting
$g_{\lambda_n} (\omega_{\lambda_n})$ with $g^{m+1}_{\lambda_n}(\omega_{\lambda_n})$. It follows that $g_{\lambda_n} (\omega_{\lambda_n})$ lies in
$\ray_{\lambda_n}$. Finally, the portion of $\ray_{\lambda_n}$ between
$g_{\lambda_n}(\omega_{\lambda_n})$ and $g^2_{\lambda_n}(\omega_{\lambda_n})$
lifts, under $g_{\lambda_n}$ to a path compactly contained in $V(\omega_0)$ which
necessarily connects $\omega_{\lambda_n}$ to $g_{\lambda_n}(\omega_{\lambda_n})$
since $\omega_{\lambda_n}$ is the unique $g_{\lambda_n}$ preimage of $g_{\lambda_n}(\omega_{\lambda_n})$ in $V(\omega_0)$ for $n$ sufficiently large.
It follows that $\omega_{\lambda_n}$ is contained in $\ray_{\lambda_n}$. That is, we have proven Claim 2.

\smallskip
From the second claim we conclude that arbitrarily close to $f_0$ there exist
$f_{\lambda_n}$ which lies in a parameter ray $R_{\cU_n}(\theta)$ for some escape region $\cU_n$. Since there are finitely many escape regions and each one has finitely many parameter rays with argument $\theta$, after passing to a subsequence
we may assume that both the escape region and the ray,
say  $R_{\cU}(\theta)$, are independent of $n$. Therefore $R_{\cU}(\theta)$
lands at $f_0$.
\end{proof}

\subsection{Landing at pcf maps: kneading and  preperiodic ray connections}
\label{s:pcf-landing}

We deduce from Theorem~\ref{thr:bms} the following:

\begin{theorem}
  \label{thr:arrival-preperiodic-connection}
  Consider a parameter ray $R_\cU(\theta)$ of an escape region $\cU$ such that
  $\theta + 1/3$ and $\theta-1/3$ are strictly preperiodic arguments.
  Then $R_\cU(\theta)$ lands at a pcf map $f_0$
  with take-off argument~$\theta$ and $\kappa(f_0,\theta) = \kappa (\cU)$.
  Moreover, if maps $f \in R_\cU(\theta)$ have
  a preperiodic ray connection between
  $-a(f)$ and $a(f)$, then  $-a(f_0) \in \partial V(a(f_0))$  and
  $\theta$ is a supporting argument for $V(a(f_0))$.
 \end{theorem}

 \begin{proof}
   From Theorem~\ref{thr:bms} we know
   that $\cR_\cU(\theta)$ lands at a pcf map $f_0$.
   In particular all Julia cycles of $f_0$ are repelling
   and $R_f(\theta)$ converges
   uniformly to $R_{f_0} (\theta)$, 
   as $f \in \cR_\cU(\theta)$ converges to $f_0$. Therefore,  $R_{f_0}(\theta)$ lands at the cocritical point
   $2a(f_0)$ and $R_{f_0}(\theta\pm 1/3)$ land at $-a(f)$, so $\theta$ is a take-off argument for $f_0$.
   Moreover, $R^*_f(\theta\pm 1/3)$ converge to $R_{f_0}(\theta\pm 1/3)$, thus $\kappa(\cU) = \kappa(f_0,\theta)$. See~Figure~\ref{trekingA}.

   In the presence of a preperiodic ray connection between $-a(f)$ and $a(f)$ the periodic ray $R_f(\theta)$
   lands at $\zeta(f) \in \partial V(a(f))$ and it is relatively supporting for $V(a(f))$. Since $\zeta(f)$ converges to a repelling periodic point
   $\zeta(f_0) \in \partial V(a(f_0))$, we have that $R_{f_0} (\theta)$ lands at $\zeta(f_0)$ and it is relatively supporting for $V(a(f_0))$.
   From~\cite{KiwiWOP}, observing that the critical values of $f_0$ are not
   separated by rays landing at the orbit of $\zeta(f_0)$, 
   the periodic orbit portrait of $\zeta(f_0)$
   consists of exactly one cycle of rays if the period $\zeta(f_0)$ is a proper divisor of $p$, and of exactly two cycles of rays
   if the period of $\zeta(f_0)$ is $p$.
   In both cases, $R_{f_0} (\theta)$ is in fact a left or right supporting ray for $V(a(f_0))$.
 \end{proof}

\section{Immigration and trekking}
\label{s:trekking}

In \S~\ref{s:bounded-hyperbolic} we introduce and discuss parameter internal rays of  type $A$ and $B$ hyperbolic components.
In \S~\ref{s:sectorAB} we discuss basic properties of maps in $0$-parameter internal rays as well as in the sectors bounded by these rays.
The landing of  $0$ and $1/2$-parameter internal rays is the subject of
\S~\ref{s:landing-internal-parameter}. We will say that a ``root'' of a type $A$ or $B$ component is the landing point
of a $0$-parameter internal ray and a ``co-root'' is the landing point of a $1/2$-parameter internal rays of a type $A$ component.
The effect that crossing from one root to another root of a type $B$ hyperbolic component has on kneading words is the content of \S~\ref{s:trekkingB}. The crossing of type $A$ components is discussed in \S~\ref{s:trekkingA}.
In \S~\ref{s:parabolic-immigration} we establish that
in the presence of a periodic ray connection the
corresponding parameter ray lands at a ``root'' of a type $A$ or $B$ component.  In \S~\ref{s:pcf-immigration} we show that a parameter ray with a preperiodic ray connection lands at a  ``co-root'' of a type $A$ component.

\subsection{Parameter internal rays}
\label{s:bounded-hyperbolic}
Given a type $A$ or $B$ hyperbolic component $\cH$ let $0 \le k < p$ be such that $-a(f) \in V_f(a_k(f))$ for all $f \in \cH$. There is exactly one map $f_\star \in \cH$ such that $-a(f_\star) = a_k(f_\star)$ (see~~\cite{MilnorHyperbolicComponents}). We say that the pcf map $f_\star$ is the \emph{center} of $\cH$. 
For each $f \in \cH \setminus \{ f_\star\}$ there
exists exactly one argument $s$ such that
$I^*_{f,a_k} (s)$ terminates at $-a(f)$. The B\"ottcher coordinate $\varphi_{f,a_k}$ has a well defined limit
as $z\in I^*_{f,a_k}(t)$ converges to $-a(f)$ which, by abuse of notation,
we write as $\varphi_{f,a_k} (-a(f))$.
Thus we may introduce:
$$\begin{array}[h]{rccll}
    \Phi_\cH : & \cH  & \to & \D \\
               & f & \mapsto & \varphi_{f,a_k} (-a(f)), & \text{ if } f \neq f_\star,\\
     & f_\star & \mapsto & 0.
  \end{array}
  $$

  \begin{proposition}
    \label{parametrizationH}
    The map $\Phi_\cH :  \cH  \to  \D$ is a holomorphic branched covering with a unique branch point $\fstar$. The degree of $\Phi_\cH$ is $2$ if $\cH$ is of type $A$ and $3$ if $\cH$ is of type $B$. 
  \end{proposition}

   Given $t \in \RZ$, we say that an arc $\cI_\cH(t)$ is a \emph{parameter internal ray with argument $t$} if
   $$ \Phi_\cH : \cI_\cH(t) \to [0,1[ \exp(2 \pi i t)$$ is a homeomorphism.

   \smallskip
Auxiliary families of  cubic and quartic polynomials with a fixed critical point $\mu$ will be useful to apply the results contained in~\cite{MilnorHyperbolicComponents}  in order to prove Proposition~\ref{parametrizationH}.  Indeed, let 
$$A_\mu (z) = (z-\mu)^2(z+2\mu) + \mu$$ and
$$B_\mu(z) = (z^2 - \mu^2)^2 + \mu$$
where $\mu \in \C$.
Observe that $\cS_1$ corresponds exactly to the polynomials $A_\mu$ with marked critical point $\mu$ and free critical point
$\omega=-\mu$.
The quartic polynomial $B_\mu$ has critical points at $\pm \mu$ and $0$. Note that $B_\mu(-\mu) = B_\mu (\mu)=\mu$ while $\omega=0$ is the free critical point of the family.

The main hyperbolic component 
$\cH_A$ of the $A_\mu$ family is the one containing the monomial $A_0(z) =z^3$
while the main hyperbolic component $\cH_B$ of the $B_\mu$ family is the
one containing $B_0 (z) =z^4$. We will rely on a result by Milnor~\cite{MilnorHyperbolicComponents} which implies that type $A$ and type $B$ hyperbolic components $\cH \subset \cS_p$ posses a natural conformal isomorphism with the main hyperbolic components of the $A_\mu$ and $B_\mu$ families, respectively.

    We let $T =A$ or $B$ and freely use the notation $T_\mu$   for
    the family $A_\mu$ or $B_\mu$. Similarly, $\cH_T$ denotes the main hyperbolic component of the corresponding family. To describe Milnor's conformal isomorphism we choose a continuously varying fixed point $z_0 (\mu)$ in the Julia set of $T_\mu \in \cH_T$ and a  continuously  varying fixed point
 $z (f)$ of $f^p : \partial V(a_k(f)) \to \partial V(a_k(f))$  for $f \in \cH$.
Then according to~\cite[Theorem~6.1]{MilnorHyperbolicComponents}, given $f \in \cH$
there exists a unique map $T_\mu :=\Psi_\cH(f)  \in \cH_T$  such that $f^p : \overline{V(a_k(f))} \to
  \overline{V(a_k(f))}$ is  conjugate to
$T_\mu : K(T_\mu) \to K(T_\mu)$  via a
homeomorphism $h_f : \overline{V(a_k(f))} \to K(T_\mu)$, conformal in the interior, such that
$h_f(z(f)) = z_0 (\mu)$. Moreover, $\Psi_\cH : \cH \to \cH_T$  is a conformal isomorphism. 

\begin{proof}[Proofs of Proposition~\ref{parametrizationH}] 
It suffices to establish the corresponding assertions for $\cH_A$ and $\cH_B$.
Although we may directly apply  the content of Proposition~2.4 (b) in~\cite{RoeschSone} for $\cH_A$ we briefly sketch here
how to proceed in both cases: $A$ and $B$.

For $0 \neq \mu \in \cH_T$ let  $\varphi_\mu$ be the B\"ottcher coordinate that conjugates the action of $T_\mu$ near $z=\mu$ with $z \mapsto z^2$.
Denote the basin of $\mu$ under $T_\mu$ by $V(\mu)$.
As before, $|\varphi_\mu|$ extends continuously
to   $V(\mu)$ via the functional relation $|\varphi_\mu \circ A_\mu| =
|\varphi_\mu|^2$ and $\varphi_\mu$ extends
along the flow lines of $-\nabla |\varphi_\mu|$. For each $s \in \RZ$,
denote by  $I^*_\mu(s)$ the arc starting at $\mu$
such that $\varphi_\mu (I^*_\mu(s)) = [0,\rho(s)[ \exp(2 \pi i s)$ for some maximal $\rho(s) >0$. There exists a unique $t \in \RZ$ such that $I^*_\mu(t)$
terminates at the free critical point $\omega$.
Let $\Phi_T : \cH_T \to \D$ be defined by $\Phi_T (\mu) = \rho(t) \exp(2 \pi i t)$ if $\mu \neq 0$, and $\Phi_T(0) =0$.
It is not difficult to show that
$\Phi_T$ is a proper holomorphic map. The standard quasiconformal surgery argument shows that $\Phi_T$
is branch point free in $\cH_T \setminus \{0\}$ (e.g.~\cite[Section~4.2.2]{BrannerFagellaBook}).
Parameter internal rays for $\cH_T$ are now defined via $\Phi_T$ as above for $\cH$. 
Moreover, if $\cH \subset \cS_p$ is of type $T$, then $\Phi_\cH = \Phi_T \circ \Psi_\cH$. 

We claim that the degree of $\Phi_A$ is $d_A:=2$ and of $\Phi_B$ is $d_B:=3$.
In fact, it is sufficient to show that there are exactly $d_T$
parameter internal rays $\cI \subset \cH_T$  with argument $0$.
Given such an internal ray $\cI$, any accumulation point $\mu_0$ of $\cI$ in $\partial \cH_T$ is such that
the internal ray $I_{\mu_0}(0)$ lands at a parabolic fixed point, for otherwise $I_{\mu}(0)$ is smooth and lands at a repelling
fixed point for all $\mu$ sufficiently close to $\mu_0$ which would contradict the definition of $\cI$.
There are exactly $d_T$ parameters $\mu$  for which
$T_\mu$ has a parabolic fixed point with multiplier $1$, namely the solutions of $\mu^2 = -4/9$ for $T=A$ and
the solutions of $\mu^3 =-2/3$ for $T=B$.
Thus $\cI$ must limit one of these three parameters as it approaches $\partial \cH_T$.
By the Maximum Principle, distinct parameter internal rays with argument $0$ must land at distinct parabolic parameters.
Hence the degree of $\Phi_T$ is at most $d_T$. Now let $\zeta = \exp(2 \pi i /d_T)$ and observe that
$\zeta T_\mu (\zeta z) = T_{\zeta \mu}$. Hence, $\Phi_T (\zeta \mu ) = \Phi_T(\mu)$ and we conclude that $\deg \Phi_T = d_T$.
Proposition~\ref{parametrizationH} now follows from the identity $\Phi_\cH = \Phi_T \circ \Psi_\cH$.
 \end{proof}

 \subsection{Sectors of type $A$ and $B$ components}
 \label{s:sectorAB}
 
 We say that a \emph{sector $S$ of  $\cH$} is a connected component of the preimage under $\Phi_\cH$ of $\{ z \in \D \setminus \{0\} : s < \arg z < t \}$ for some $s,t \in \RZ$ not necessarily distinct. 

 We are particularly interested on sectors bounded by $0$-rays.
 Consider a sector $S$ of $\cH$ such that
 $\Phi_\cH : S \to \D \setminus [0,1[$ is a bijection and  $s$-internal rays
 $\cI(s)$ contained in $\overline{S}$
  converge to the $0$-rays $\cI^\pm$ as $s \to 0^\mp$. We say that $S$ is \emph{the sector of $\cH$ from $\cI^-$ to $\cI^+$}.  


 

 
   \begin{proposition}
     \label{sectorH}
     Let $\cH$ be a type $A$ or $B$ hyperbolic component.
     Consider a continuously varying fixed point $z(f)$  of $f^p : \partial V(a_k(f)) \to \partial V(a_k(f))$, for $f \in \cH$.
    Then  the following hold:
    \begin{enumerate}   
      \item There exist unique $0$-parameter internal rays $\cI^-$ and $\cI^+$ such that for all $f \in \cI^-$ (resp. $\cI^+$), the ray $I^-_{f,a_k}(0)$ (resp. $I^+_{f,a_k}(0)$) lands at $z(f)$.
    \item The sector $S$ in $\cH$ from $\cI^-$ to $\cI^+$ coincides with the set formed by all $f \in \cH$ such that the $0$-internal ray
      $I_{f,a_k}(0)$ is smooth and lands at $z(f)$.
    \end{enumerate}
  \end{proposition}

  \begin{definition}
    \label{d:sector-theta}
    In the above proposition, when $z(f)$ is the landing point of an external ray $R_f(\theta')$ for all $f \in \cH$, the corresponding
    sector is denoted by $S(\cH,\theta')$ and called
    the \emph{sector of $\cH$ associated to $\theta'$.}
  \end{definition}

\smallskip
  The proof of the proposition will simultaneously yield additional information:
  
  \begin{lemma}
    \label{consecutive}
    Let $\cI$ be a parameter internal ray of $\cH$ with argument $0$. Denote by
    $z^\pm(f)$ the landing point of $I^\pm_{f,a_k}(0)$ for $\fstar \neq f \in \cI$. Then $z^-(f)$ is the successor of $z^+(f)$ among the fixed points of
    $f^p: \partial V(a_k(f)) \to \partial V(a_k(f))$ in the standard orientation
    of $\partial V(a_k(f))$. 
  \end{lemma}


\begin{proof}[Proof of Proposition~\ref{sectorH} and Lemma~\ref{consecutive}]
The parameter internal rays with argument $0$ for $\cH_A$ and $\cH_B$ have an explicit description.
Indeed, it is not difficult to check that for $\mu \in ]0,2/3[ \cdot i \subset \cH_A$
the arc  $]0,|\mu|[ \cdot (-i)$ is mapped by
$\varphi_\mu$ into $]0,1[$. It follows that $\cI_A =  ]0,2/3[ \cdot i$ is one parameter internal ray with argument $0$
and $(-1) \cdot \cI_A = ]0,2/3[ \cdot (-i)$ is the other one. 
Similarly, we let $\mu_0 = -(2/3)^{1/3} < 0$ and check that $]\mu_0, 0] \subset \cH_B$. Moreover,
for $\mu_0 < \mu <0$ we have that $[\mu,0[$ is mapped under
$\varphi_\mu$ into $]0,1[$. Hence $\cI_B = ]\mu_0,0]$ is one parameter internal ray with argument $0$ and
$\exp(\pm 2 \pi i/3) \cdot \cI_B$ are the other ones.

We first prove the analogue of Proposition~\ref{sectorH} for $\cH_A$ and $\cH_B$.
Assume that $\mu \neq 0$ lies in a  ray $\cI$ with argument $0$.
Denote by $z^\pm (\mu)$ the landing points of $I^\pm_\mu (0)$.
It follows that $z^\pm (\mu)$ depend continuously on $\mu$.
Denote by $z_j(\mu)$ the fixed point where external ray $R_\mu(j/d_T)$
lands where $j=0,\dots, d_T-1$.
For $0 \neq \mu \in \cI_A$ (resp. $\cI_B$),
exploiting the symmetry under reflection around the imaginary axis (resp. real axis)
$z^-(\mu) = z_1(\mu)$  and $z^+(\mu) = z_0(\mu)$.
(resp. $z^+(\mu) = z_1(\mu)$ and $z^-(\mu) = z_2(\mu)$). The  multiplication by $\zeta = \exp(2 \pi i/d_T)$ conjugacy
maps the landing point $z^\pm(\mu)$ of $I^\pm_\mu(0)$ for $0 \neq \mu \in \cI_T$ onto
the landing point $z^\pm(\zeta \mu)$ of  $I^\pm_{\zeta \mu} (0)$  for $\zeta \mu \in \zeta \cI_T$.
The same conjugacy maps the fixed point $z_j(\mu)$ onto the fixed point
$z_{j+1}(\mu)$ subscripts $\mod d_T$. Hence for any $z_j(\mu)$ there exist a unique parameter internal ray $\cI^\pm$ such that $z_j(\mu)$ is landing
point of $I^\pm_\mu(0)$ 
for all $0 \neq \mu \in \cI^\pm$. Moreover along any parameter internal ray with argument $0$ the
landing point of $I^-_\mu(0)$ is a successor of the landing point of $I^+_\mu(0)$ among the fixed point in the Julia set $J(T_\mu)=\partial V(\mu)$.
Now we may transfer these results to $\cH$, via $\Psi_\cH$, to conclude that Lemma~\ref{consecutive} as well as (1) of the proposition hold.

For $0 < \mu \in  \cH_A$ it is not difficult to see that
$I_\mu(0)$ is smooth and lands at $z_0(\mu)$. It
follows that the sector  $S$ of $\cH_A$ from $(-1) \cdot \cI_A$ to $\cI_A$   (i.e. contained in the right half plane) is such that $I_\mu(0)$ lands at $z_0(\mu)$.
 Along $\cI_A$, we have that $I^+_\mu(0)$ lands at $z_0(\mu)$
and along $(-1) \cdot \cI_A$ we have that  $I^-_\mu(0)$ lands at $z_0(\mu)$.
Conjugacy via multiplication by $-1$ establishes a similar assertion for the sector in the left half plane. Each one of these
sectors map bijectively under $\Phi_A$ onto $\D \setminus ]0,1[$. Transferring these results to any hyperbolic component $\cH$ of type $A$
part (2)  of the proposition follows for  type $A$ components.
A similar analysis of maps $B_\mu$ with $\mu >0$ in $\cH_B$ proves assertion (2) of the proposition for type $B$ components.
\end{proof}



\subsection{Landing of $0$ and $1/2$-parameter internal rays}
\label{s:landing-internal-parameter}

Here we  describe where $0$-parameter internal rays of type $A$ and $B$ hyperbolic components land. Landing points of $1/2$-rays of
type $A$ components are also considered.
Before stating and proving the results, we summarize basic dynamical plane facts:

\begin{proposition}
  \label{dynamics-typeAB}
  Let $f \in  \cC(\cS_p)$ be a hyperbolic map such that $-a(f) \in V(a_k(f))$ for some $0 \le k < p$. Then the following statements hold:
  \begin{enumerate}
  \item  $f^p : \partial V(a_k(f)) \to \partial V(a_{k}(f))$ is conjugate to $m_3: \RZ \to \RZ$ if $k=0$ and to
    $m_4: \RZ \to \RZ$ otherwise.
  \item Suppose that $z_2$ is the successor of $z_1$ among the fixed points of  $f^p : \partial V(a_k(f)) \to \partial V(a_{k}(f))$. Then the open arc from $z_2$ to $z_1$ in $\partial V(a_{k}(f))$ contains a unique point  $w_1$
    (resp. $ w_2$) 
    such that $f(w_1) = f(z_1)$ (resp. $f(w_2) = f(z_2)$).
    Moreover, the list $z_2,w_1, w_2, z_1$ respects cyclic order in $\partial V(a_k(f))$.
    \item If $R_f(t)$ lands at a fixed point of $f^p : \partial V(a_k(f)) \to \partial V(a_{k}(f))$ then the period of $t$ is exactly $p$. 
  \item There exists at most one periodic point of $f$ in $ \partial V(a_k(f))$  with more than one external ray landing at it.
  \item There exists at most one periodic point of $f$ in $ \partial V(a_k(f))$ of period
     strictly less then $p$. 
  \end{enumerate}
\end{proposition}

\begin{proof}
  Since $f^p : \partial V(a_k(f)) \to \partial V(a_{k}(f))$  is an uniformly expanding degree $3$ or $4$ map of a Jordan curve,
  the first assertion follows.  The second assertion holds
  since it is true for $m_3$ and $m_4$.
  Let $z$ be a fixed point $z$ of $f^p : \partial V(a_k(f)) \to \partial V(a_{k}(f))$. For a small arc $\gamma$ starting at $z$, contained in
  $\partial V (a_k(f))$, we have that $f^p(\gamma) \supset \gamma$.
  Since $f^p$ preserves the cyclic order of arcs emanating from $z$, every periodic ray landing at $z$ must be fixed under $f^p$ and we have proven (3).

  Consider the set $X$ formed by the periodic points of $f$ which belong to  $\partial V(a_j(f))$ for some $0\le j <p$ and have more than one ray landing at it. To prove (4) we assume that $X \neq \emptyset$. It is sufficient to
  show that $X$ has exactly one point in $ \partial V(a_\ell(f))$ for some $0 \le \ell <p$. In fact, 
  let  $Z_f(w,s,t)$  be a sector that has minimal angular length among all
  the sectors based at elements of $X$. Note that $Z_f(w,s,t)$ is disjoint from $X$. 
  Every sector based at a periodic point contains a postcritical point (e.g.~\cite{KiwiWOP}). Hence,
  $Z_f(w,s,t)$ contains $V(a_\ell(f))$ for some $0 \le \ell <p$ and $w$ is the unique element of $X$ in $\partial V(a_\ell(f))$.
  Statement (5) is a direct consequence of (4). 
\end{proof}

\subsubsection{Landing of $0$-parameter rays}
\label{s:landing0-internal}
Consider a $0$-ray $\cI_\cH(0)$ of a type $A$ or $B$ hyperbolic component $\cH \subset \cC (\cS_p)$ centered at $\fstar$. For $f \in \cI_\cH(0)$ such that $f \neq \fstar$,  we will freely use the following notation
(see~Figure~\ref{f:0-internal-notation}):

(1) The landing points of
$I^\pm_{f,a_k} (0)$ are denoted by $z^\pm (f)$. By Lemma~\ref{consecutive}, $z^-(f)$ is a successor of $z^+(f)$ among the fixed points of $f^p$ in $\partial V(a_k(f))$. 

(2) Denote by $w^\pm(f)$ the preimages of $f(z^\pm(f))$ in $\partial V(a_k(f))$ so that $$z^-(f), w^+(f), w^-(f),z^+(f)$$ are listed respecting cyclic order.
These points exist and are well defined due to Proposition~\ref{dynamics-typeAB}.

(3) For each one of the points $z = z^\pm(f)$ or $ w^\pm(f)$ there exists a unique
sector based at $z$ that contains $V(a_k(f))$.
This sector is bounded at the left by a ray with argument $t(z)$ and
at the right by one with argument $s(z)$.
That is, we let $s(z), t(z) \in \QS$ be such that
$V(a_k(f))$ is contained in the sector $\sector_f (z,t(z),s(z))$ based
at $z=z(f)$.
Notice that $s(z)$ and $t(z)$ are not necessarily distinct.
In fact, by Proposition~\ref{dynamics-typeAB} (4), $s(z^-)=t(z^-)$ or $s(z^+)=t(\label{f:0=internal-notation}z^+)$.
Moreover $s(z),t(z)$ are independent of $f \in \cI_H(0) \setminus \{\fstar\}$.
The arguments $t(z)$ and $s(z)$ are left and right supporting for $V(a_k(f))$.

\begin{figure}[h]
  \centering
   \fbox{ \includegraphics[width=8.8cm]{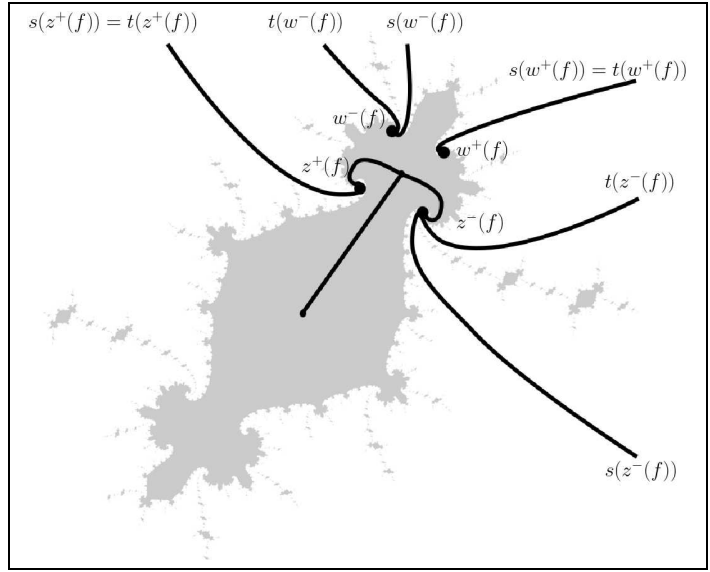} }
  \caption{Illustrates notations (1), (2) and (3) of \S~\ref{s:landing0-internal}\label{f:0-internal-notation}}
\end{figure}

\begin{proposition}
  \label{internal-landing-0}
  Consider a type $A$ or $B$ hyperbolic component and an internal ray $\cI_\cH(0)$ as above.
  Then $\cI_\cH(0)$ lands at a parabolic map $f_0$ such that
  the landing point of $I_{f_0,a_k}(0)$ is a parabolic periodic point $z_0$. Moreover, all of the following statements hold:
\begin{enumerate}
\item The period of $z_0$ is the minimum $\ell$ of the periods of 
  $z^+(f)$ and $z^-(f)$ for any $\fstar \neq f \in \cI_\cH(0)$.
  The multiplier of $z_0$ is a primitive $p/\ell$-root of unity. 
\item $z^\pm(f) \to z_0$ as $f \in \cI_\cH(0)$ approaches $f_0$ and 
  $z_0 \in \partial V_{f_0} (-a(f_0))$.
\item A ray $R_{f_0}(t)$ lands at $z_0$ if and only if $R_f(t)$ lands at
  the $f^\ell$-orbit of $z^+(f)$ or $z^-(f)$
  for all $\fstar \neq f \in \cI_\cH(0)$.
 \item  $\sector_{f_0} (z_0, t(z^+),s(z^-))$ is the sector at $z_0$ that
  contains $a_k(f_0)$.
\item $w^\pm (f) \to w_0$ as $f \in \cI_\cH(0)$ approaches $f_0$ where $w_0$ is the unique non-periodic preimage of $f_0(z_0)$ in $\partial V_{f_0} (-a(f_0))$.
\item The external rays with arguments $s(w^\pm)$ and $t(w^\pm)$ of $f_0$ 
    land at $w_0$.
\item  $\sector_{f_0}(z_0, t(z^-),s(z^+))$ is a sector at $z_0$ and 
  $\sector_{f_0}(w_0, t(w^-),s(w^+))$ is a sector at $w_0$. Both of these sectors contain $-a(f_0)$.
\item If $\theta$ is such that $\{\theta + 1/3,\theta-1/3\}$ is one of
  the pairs $\{t(z^-), t(w^-)\}$, $\{s(z^-), s(w^-)\}$, $\{t(z^+), t(w^+)\}$, $\{s(z^+), s(w^+)\}$, then $\theta$ is a take-off argument for $f_0$.
\item For any pair of arguments
$s, t \in \{ s (z^\pm), t(z^\pm), s (w^\pm), t(w^\pm) \}$, 
  if $$a_j(f_\star) \in \sector_{f_\star} (\overline{V(a_k(f_\star))}, s,t),$$ then
$$a_j(f_0) \in \sector_{f_0} (\overline{V(a_k(f_0))\cup V(-a(f_0))}, s,t).$$
\end{enumerate}
\end{proposition}

For an illustration of a $0$-internal ray of a type $B$ component landing at a parabolic map see Figure~\ref{NoTriv}.

\begin{proof}
  If $g$ is an accumulation point of $\cI_\cH (0)$ in $\partial \cH$ and
  $z$ is the landing point of $I_{g,a_k}(0)$, then $z$ is not repelling. Otherwise, for all $f$ close to $g$ the ray $I_{g,a_k}(0)$ would be smooth and land at the analytic continuation of $z$. By the Snail Lemma $df^p/dz (z)=1$. There are only finitely many maps $g \in \cS_p$ possessing such a parabolic periodic point.
  Taking into account that the accumulation set of $\cI_\cH (0)$ in $\partial \cH$ is connected we conclude that $I_\cH(0)$ lands
  at a map $f_0$ with a parabolic periodic point $z_0$, say of period $\ell$. Moreover, $z_0$ is the landing point of $I_{f_0,a_k}(0)$. Hence, $\ell$ divides $p$ and $z_0$ has a repelling direction (in the direction of $I_{f_0,a_k}(0)$) of exact period $p/\ell$ under $f^\ell_0$. Thus, the multiplier of $z_0$ is a primitive $p/\ell$-root of unity.

  Since there is only one cycle of repelling petals at $z_0$ and periodic points along $I_\cH(0)$ move continuously, there exist one period $\ell$ periodic point $z_0(f)$ converging to $z_0$ and one period $p/\ell$-orbit of $f^{\ell}$ converging to $z_0$  (note that possibly $\ell=p$). No other
  periodic points converge to $z_0$.

  For the rest of the proof we need to consider the hyperbolic metric $d_f$ in $V(a_k(f))$.
  Given $\fstar \neq f \in \cI_\cH(0)$ the image $U^*_f$ of the B\"ottcher coordinate $\varphi_{f,a_k}: V^*(a_k(f)) \to \D$ is the unit disk with the ``needle'' $[|\varphi_{f,a_k}(-a(f))|, 1[$ removed, as well as all its iterated preimages under $z \mapsto z^2$.
If $0 < r <  |\varphi_{f,a_k}(-a(f))|$, then the hyperbolic distance in $U^*_f$ between $r$ and $r^2$ is bounded by a constant independent of $f \in \cI_\cH(0)$. Observe that $d_f$ is bounded above by the hyperbolic distance in $U^*_f$.
 Let
 $I^*_f(0)$ be the $\nabla |\varphi_{f,a_k}|$-flow line
 from $a_k(f)$ to $-a(f)$. There exists a constant $C$ independent of $f$
 such that 
  $d_f(z, f^p(z)) < C$ for all $z \in I^*_f(0)$. 
  Consider $X = K(f_0) \cap \limsup I^*_f (0)$ and note that $-a(f_0) \in X$.
  We claim that points in $X \cap J(f_0)$ are periodic of period dividing $p$.
  In fact, if $w \in X \cap J(f_0)$,
  then there exist  $f_n \to f$ and $w_n \to w$ such that $w_n \in I^*_{f_n}$. Since $\limsup J(f_n) \supset J(f_0)$  we have that
  the Euclidean distance  from $w_n$ to $J(f_n)$ converges to
  $0$, which together
  with $d_{f_n}(w_n,f^p_n(w_n)) <C$, implies that $|w_n - f^p_n(w_n)| \to 0$. Thus, $f_0^p(w) =w$ as claimed.

  As in the proof of Lemma~\ref{l:limsup}, it follows that $X \cap J(f_0)$ is a singleton for otherwise there would 
  exist a bounded
  Fatou component with two distinct periodic points of period dividing $p$ in its boundary.  
  Since $z_0 \in X$ we have that $X \cap J(f_0) = \{ z_0 \}$.
  Moreover $-a(f_0) \in X$ so $z_0 \in \partial V(-a(f_0))$. 

  Now (1) and (2) would follow after proving that $z^\pm(f) \to z_0$.
  We argue in a similar fashion but instead of considering $I^*_f(0)$ we consider arcs $I^\pm_f$ landing at $z^\pm(f)$ (cf. proof of Lemma~\ref{l:limsupperiodic}). To introduce these arcs
let us consider $h: \H \to \D$ given by $z \mapsto \exp ( i z)$ and let $\widetilde{U}_f = h^{-1} (U^*_f)$.
For $f \in \cI_\cH(0)$ such that $-\log|\varphi_{f,a_k}(-a(f))| < 2 \pi$, we have that $\widetilde{U}_f$ contains the sectors $S^+= \{ z \in \H : \pi/4<\arg z<\pi/2 \}$ and
$S^-= \{ z \in \H : \pi/2<\arg z<3\pi/4 \}$.
Given $\alpha \in ]\pi/4, \pi/2[$ (resp. $]\pi/4, \pi/2[$),
the arc $\gamma_f(\alpha) = \varphi_{f,a_k}^{-1}(h(]0,\infty[ \cdot \exp(i\alpha))$
is $f^p$ invariant and
$d_f (z,f^p(z))$ is uniformly bounded for all $z \in \gamma_f(\alpha)$ and all $f \in \cI_\cH(0)$ such that $-\log|\varphi_{f,a_k}(-a(f))| < 2 \pi$. It follows that
$\gamma_f(\alpha)$ limits towards a periodic point $w(\alpha)$ at one end
and towards $a_k(f)$ at the other. The periodic point $w(\alpha)$
is independent of $\alpha \in ]\pi/4, \pi/2[$ (resp. $]\pi/4, \pi/2[$)
(e.g. see~\cite[Corollary~17.10]{DynInOne}).
Hence, $w(\alpha)= z^-(f)$ (resp. $z^+(f)$) since $\gamma_f(\alpha)$ converges
to $I^-_f(0)$ (resp. $I^+_f(0)$) as $\alpha \to \pi/2^-$ (resp. $\pi/2^+$).
The corresponding arc $\gamma_{f_0}(\alpha)$ in the dynamical plane of $f_0$
is homotopic rel $J(f_0)$ to the internal ray $I_{f_0,a_k}(0)$, hence it limits toward $z_0$ at one end. 
A similar argument as the one used above for $X$ yields that $J(f_0) \cap \limsup \gamma_f(\alpha) = \{z_0\}$. Therefore, $z^\pm(f) \to z_0$.

Let us now show that  if a periodic ray $R_f(t)$ lands at $z^\pm(f)$, then
$R_{f_0} (t)$ lands at $z_0$. Note that the connected set
$K(f_0) \cap \limsup R_f(t)$ contains $z_0$. We may again use the elementary hyperbolic distance estimate to show that $J(f_0) \cap \limsup R_f(t)$ consists of points of period dividing $p$. Again taking into account that  periodic Fatou components of $f_0$ have degree $2$ return maps and therefore only have one point of period dividing $p$ on their boundaries, we conclude that $J(f_0) \cap \limsup R_f(t)$ is just the singleton $\{ z_0 \}$. Thus, $R_f(t)$ lands at $z_0$.

From the previous paragraph, if $R_f(t)$ lands at  $f^{\ell m}(z^\pm (f))$ then  $R_{f_0}(t)$ lands at $z_0$.
For $(3)$, we claim that if a ray $R_{f_0} (t)$ lands at $z_0$ then
$R_f(t)$ lands at  $f^{\ell m}(z^\pm (f))$ for some $m$.
Observe that no ray landing at $f^{\ell m + j}(z^\pm(f))$ for $j=1, \dots, \ell-1$ may land at $z_0$ since the period of $z_0$ is exactly $\ell$. Now if a ray $R_f(t)$ does not land in the orbit of $z^\pm(f)$ then its landing point, say $z(f)$ converges to a repelling periodic point $z(f_0)$. Thus $R_{f_0}(t)$ lands at $z(f_0) \neq z_0$, which proves our claim and  (3) follows.

Now we assert that the arguments of the rays of $f_0$ landing at $z_0$ are contained in
$[s(z^+),t(z^+)] \cup [s(z^-),t(z^-)]$. 
In fact, from Proposition~\ref{dynamics-typeAB} (4),  $s(z^+) = t(z^+)$ or $s(z^-)=t(z^-)$. For simplicity assume that the latter holds. If $\ell < p$, 
then the period of $z^+(f)$ is $\ell$ and
the boundaries of $V(a_k(f)), V(a_{k+\ell}(f)), \dots, V(a_{k+p-\ell}(f))$ meet at $z^+(f)$.
Hence for $0 < m < p/\ell$ the sector $\sector_f (z^+(f), s(z^+), t(z^+))$
contains $f^{\ell m}(z^{-}(f))$. Thus the rays landing at
$z_0$ in the dynamical plane of $f_0$ are contained in  $[s(z^+),t(z^+)] \cup [s(z^-),t(z^-)]$.

From the previous paragraph, $\sector_{f_0}(z_0,t(z^+),s(z^-))$ is a sector at $z_0$.
Observe that in the dynamical plane of $f \in \cI_\cH(0)$, the internal ray
$I_{f,a_k} (1/3)$ lands at a periodic point $u(f)$ in the arc from $z^+(f)$ to
$z^-(f)$ of $\partial V(a_k(f))$. The repelling periodic point $u(f)$ is the landing point of an external ray $R_f(t')$ with $t' \in ]t(z^+),s(z^-)[$ and
$u(f)$ must converge to a repelling periodic point $u(f_0)$ of $f_0$. Hence,
$I_{f_0,a_k} (1/3)$ and $R_{f_0}(t')$ land at  $u(f_0)$. Therefore,
$a_k(f_0) \in \sector_{f_0}(z_0,t(z^+),s(z^-))$. That is, we have proven (4).

Consider $\fstar \neq f \in \cI_\cH(0)$. 
In order to study the convergence of $w^\pm(f)$, let $J_f^\pm$ be the portion of
$I^\pm_{f,a_k}(0)$ connecting $-a(f)$ with $z^\pm(f)$. Then there exist
arcs $L^\pm_f$ connecting $-a(f)$ with $w^\pm(f)$ such that
$f: L^\pm_f \to f(J_f^\pm)$ is a homeomorphism. As $f \to f_0$ along $I_\cH(0)$,
consider $w^\pm \in J(f_0) \cap \limsup L^\pm_f$. Then $f(w^\pm) =f(z_0)$.
We claim that $w^+ \neq z_0$ and $ w^- \neq z_0$. Otherwise for $\sigma =+$ or $-$,
we could choose
$f_n \in \cI_\cH(0)$ converging to $f_0$,
$w_n \in L_{f_n}^\sigma$ converging to $w^\sigma$
and $ z_n \in L_{f_n}^\sigma$ converging to $z^\sigma$ such that $f(z_n) = f(w_n)$. This would imply that $f_0$ is not locally injective at $z_0$ which is impossible. We conclude that $J(f_0) \cap \limsup L^\pm_f$ consists
of strictly preperiodic points that map onto $f(z_0)$. Since Fatou components of
$f_0$ contain at most one of these points in their boundary and $-a(f_0) \in 
\limsup L^+_f \cap \limsup L^-_f$ it follows that $w^+=w^-$ is the unique point $w_0$
in $\partial V(-a(f_0))$ with this property. Moreover, any limit point of $w^\pm(f)$ lies in $J(f_0) \cap \limsup L^\pm_f$. Hence $w^\pm(f) \to w_0$ and we have proven (5).

Let $\sigma = +$ or $-$. 
For $f \in \cI_\cH(0)$ consider $t$ such that $R_f(t)$ lands at $w^\sigma(f)$.
As $f \in \cI_\cH(0)$ converges to $f_0$ we claim that that $J(f_0) \cap \limsup R_f(t)$ consists of $f_0$-strictly preperiodic points $w$ such that $f_0(w) = f(z_0)$. In fact, we may apply a similar reasoning than in the previous paragraph to show that if $z_0 \in  J(f_0) \cap \limsup R_f(t)$ then $z_0$ is critical which is impossible. Thus, every point in $J(f_0) \cap \limsup R_f(t)$ is a strictly preperiodic preimage of $f_0(z_0)$. Since there is at most one such preimage in the boundary of any bounded Fatou component of $f_0$, it follows that
$\{ w_0 \} = J(f_0) \cap \limsup R_f(t)$ which implies statement (6).

To establish (7) just observe that the rays of $f_0$ with arguments $s(w^\pm)$ and $t(w^\pm)$ all lie in the sector $\sector_{f_0} (z_0, t(z^-),s(z^+))$ since $w^\pm(f)$ lie in the arc from $z^-(f)$ to $z^+(f)$ of $\partial V(a_k(f))$. Hence, $w_0 \in \sector_{f_0} (z_0, t(z^-),s(z^+))$ and therefore $-a(f_0) \in \sector_{f_0} (z_0, t(z^-),s(z^+))$. It follows that the sector at $w_0$ containing $-a(f_0)$ must be bounded by rays with arguments that map onto $3t(z^-)$ and  $3s(z^+)$. Thus (7) holds.

Since $s(z^+) = t(z^+)$ or $s(z^-)=t(z^-)$, the sectors $\sector_{f_0} (z_0, t(z^-),s(z^+))$ and  $\sector_{f_0}(z_0,t(z^+),s(z^-))$ share a boundary ray. Hence the rays with arguments involved in the definitions of these sectors land at $z_0$ through a repelling petal adjacent to the attracting petal corresponding to
$V(-a(f_0))$ and (8) follows.

Finally to prove (9) just observe that if $0 < j < p$ and $j \neq k$ then
$I_{f,a_j}(1/3)$ lands at a repelling periodic point which is the landing point of some external ray $R_f(t')$. In the dynamical plane of the limit $f_0$ we also have that  $I_{f_0,a_j}(1/3)$ and $R_{f_0}(t')$ land at a common repelling periodic point. From here (9) follows.
\end{proof}

\subsubsection{Landing of $1/2$-parameter internal rays}
\label{s:landing-half}
Consider a ray $\cI_\cH(1/2)$ of a type $A$ hyperbolic component centered at $\fstar$.
For any $f \in \cI_\cH(1/2)$ distinct from $\fstar$,
we will use the following notation:

(1) Let $z_0(f)$ be the landing point of $I_{f,a}(0)$ and
$w^\pm(f)$ be the landing points of $I^\pm_{f,a}(1/2)$.

(2) For $z=z_0(f),w^\pm(f)$ let $s(z),t(z)$ be such that $\sector_f(z,t(z),s(z))$ is the sector at $z$ containing  $a(f)$.
Note that $s(w^\pm) = s(z_0) \mp 1/3$ and $t(w^\pm) = t(z_0) \mp 1/3$.

\begin{proposition}
  \label{internal-landing-half}
  Consider a type $A$ hyperbolic component and an internal ray $\cI_\cH(1/2)$ as above.
  Then $\cI_\cH(1/2)$ lands at a pcf map $f_0$ such that $-a(f_0)$ is the landing point of $I_{f_0,a}(1/2)$. Moreover, all of the following statements hold:
  \begin{enumerate}
    \item As $f \in \cI_\cH(1/2)$ converges to $f_0$, the periodic point $z_0(f)$ converges to a repelling periodic point $z_0(f_0)$ and $w^\pm(f) \to -a(f_0)$.
    \item The external rays $s(w^\pm)$ and $t(w^\pm)$ land at $-a(f_0)$.
    \item The sector  at $-a(f_0)$  containing $a(f_0)$ is $\sector_{f_0}(-a(f_0),t(w^+),s(w^-))$.
    \item If $\theta = t(z_0)$ or $s(z_0)$, then $\theta$ is a take-off argument for $f_0$
      and $R_{f_0}(\theta)$ lands at $z_0(f_0) \in \partial V(a(f_0))$.
    \item For all $\fstar \neq f \in
      \cI_\cH(1/2)$, if $a_j(f) \in \sector(\overline{V(a(f))}, s, t)$ for some
      $s,t \in \{s(z_0),t(z_0), s(w^\pm), t(w^\pm)\}$, then $a_j(f_0) \in
      \sector(\overline{V(a(f_0))}, s, t)$.
  \end{enumerate}
\end{proposition}

\begin{proof}
  Assume that $f_n \in \cI_\cH (1/2)$ converges to $g \in \partial \cH$.
  We claim that $I_{g,a}(1/2)$ lands at $-a(g)$. To prove this consider
  $L=\limsup I^*_{f_n,a}(1/2)$ and $L_0 =\limsup I^*_{f_n,a}(0)$.
  Let $z_0$ be the landing point of $I_{g,a}(0)$. Since $g(L) = L_0$
  and as we have seen before $L_0 \cap J(g) = \{z_0\}$, if $w \in L \cap J(g)$,
  then $g(w) = g(z_0)$. Moreover, $w \neq z_0$ for otherwise $z_0$ would be a critical point. 
  Note that $-a(g) \in L$.

  Let us first show  that $-a(g)$ lies in the Julia set.
  We suppose that $-a(g)$ lies in the Fatou set and after some work arrive to a contradiction. Observe that
  $V(-a(g))$ has to be a periodic component of a parabolic basin,
  for otherwise $g \in \partial \cH$ would be hyperbolic since $J(g)$ is critical point free.
  In particular, the continuum $L$ that connects $-a(g)$ and $a(g)$ must contain
  the unique point $w_0 \neq z_0$ in $\partial V(a(g))$ such that
  $f(w_0) = f(z_0)$. It follows that $w_0 \in \partial V(-a(g))$ for otherwise
  there would exist a Fatou component distinct from $V(\pm a(g))$ containing $2$ preimages of $f(z_0)$ in its boundary. Let $m$ be such that the Fatou components $V(\pm a(g))$ are fixed under $g^m$. Then $g^m (w_0) = w_0$ which contradicts the fact that $w_0$ is strictly preperiodic. This contradiction shows that $-a(g)$ lies in the Julia set.
  
  It follows that  $g(-a(g)) = g(z_0)$ and therefore $-a(g)$ is the unique preimage of $g(z_0)$ distinct from $z_0$.
  That is, $L \cap J(g)$ is the singleton $\{-a(g)\}$. Since $I_{g,a}(1/2) \subset L$ we conclude that the landing point of $I_{g,a}(1/2)$ is $-a(g)$. 

  Recall that $g$ is an arbitrary accumulation point of $\cI_{\cH}(1/2)$.
  Since there are only finitely many $g \in \cS_p$ such that $g(-a(g))$ is
  periodic of period dividing $p$, we conclude that $\cI_{\cH}(1/2)$ lands at a map
  $f_0$ such that $-a(f_0)$ is the landing point of $I_{f_0,a}(1/2)$. It follows
  that the landing point $z_0 (f_0)$ of $I_{f_0,a}(0)$ is a repelling periodic point which is the limit of $z_0(f)$. Moreover, the non-periodic preimages $w^\pm (f)$ of $f(z_0(f))$ converge to $-a(f_0)$, the unique non-periodic preimage  of $f_0(z_0(f_0))$.  That is, we have proven (1). 

  Since $z_0(f_0)$ is repelling, the rays of $f_0$ with arguments $s(z_0)$ and
  $t(z_0)$ land at $z_0(f_0)$. Taking into
  account that $-a(f_0)$ and  $z_0(f_0)$ have the same image,
  all the rays with
  non-periodic arguments with image landing
  at $f_0(z_0(f_0))$ must land at $-a(f_0)$. That is (2) holds as well as (3) and (4).

  To establish (5), for $f \in \cI_{\cH}(1/2)$ note that the landing point $z(f)$ of the internal ray $I_{f,a_j}(1/3)$ is a repelling periodic point which is also the landing point of some  a periodic ray $R_f(t')$. It follows that $z(f)$ converges to a repelling periodic point of $f_0$ which is the common landing point of $I_{f_0,a_j}(1/3)$ and $R_{f_0}(t')$. Thus, if $a_j(f)$ lies in the region $\sector_f(\overline{V(a(f))},s,t)$ as in the statement of the proposition, then $t' \in ]s,t[$ and
  $R_{f_0}(t')$ lands at $\partial V(a_j(f_0))$. We conclude that (5) holds.
\end{proof}

\subsection{Type $B$ trekking Theorem}
\label{s:trekkingB}

Recall that $S(\cH,\theta')$ denotes a sector of a type $A$ or $B$ hyperbolic component between $0$-internal parameter rays as introduced
in Definition~\ref{d:sector-theta}. 

\begin{theorem}
  \label{thr:trekkingB}
  Assume that $\cH$ is a type $B$ component centered at $f_ \star$
  such that  $-a(f_\star) =a_k(f_\star)$. Suppose that
  $R_{f_\star}(\theta \pm 1/3)$ support $\partial V(a_k(f))$ and
  $\theta'=\theta +1/3$ or $\theta-1/3$ is a period $p$ argument.
  Then:
  \begin{enumerate}
  \item the $0$-parameter internal rays that bound the sector $S(\cH,\theta')$ land  at parabolic maps
    $f_0, f_1$ with take-off argument $\theta$ and,
  \item the associated kneadings $\kappa (f_0, \theta)$ and $\kappa (f_1,\theta)$  differ at the $k$-th symbol and coincide in the rest.
  \end{enumerate}
\end{theorem}

\begin{proof}
For $\theta, \theta'$ as in the statement of the theorem, denote by $\theta''$ 
the argument distinct from $\theta,\theta'$ such that $3 \theta'' = 3 \theta$.
That is, $\theta'$ and $\theta''$ support $V(a_k(\fstar))$.
Then $R_f(\theta')$ and $R_f(\theta'')$ land at points $z(f), w(f) \in \partial V(a_k(f))$, respectively, for all $f \in \cH$.

Consider the sector $S(\cH,\theta')$ which according to  Proposition~\ref{sectorH} is bounded by
$0$-parameter rays $\cI^\pm$. It will be convenient to rename the landing point of $\cI^\pm$ 
by $f_\pm$, respectively.
Let $\sigma \in \{+,-\}$, also from Proposition~\ref{sectorH},
for all $\fstar \neq f \in \cI^\sigma$, the internal ray $I^\sigma_{f,a_k}(0)$ lands at
$z(f)$. In the notation of \S~\ref{s:landing0-internal}, $z(f) =z^\sigma(f)$
and $w(f) =w^\sigma(f)$.  Moreover, either $\theta'=s(z^\sigma)$ and
$\theta''=s(w^\sigma)$ or, $\theta'=t(z^\sigma)$ and
$\theta''=t(w^\sigma)$. In view of Proposition~\ref{internal-landing-0} (8), $\theta$ is a take-off argument for the
landing point $f_\sigma$ of $\cI^\sigma$.

From Proposition~\ref{internal-landing-0} (9), if $j \neq k$, then 
$a_j(f_\star) \in \sector_{f_\star} (\overline{V(a_k(f_\star))},\theta',\theta'')$ if and only
$a_j(f_\sigma) \in \sector_{f_\sigma} (\overline{V(-a(f_\sigma))},
\theta',\theta'')$. In particular, for $j \neq k$, the kneading symbol corresponding to
$a_j(f_\sigma)$ is independent on whether $\sigma=+$ or $\sigma=-$.

To finish the proof we have to show that the $k$-th symbol of the kneading words of $\kappa(f_+,\theta)$ and $\kappa (f_-,\theta)$ are different.
In fact, for $f$ along the ray $\cI^-$, we have that the landing point
of $I^-_{f,a_k}(0)$ is $z^-=z(f)$ and the landing point of
$I^+_{f,a_k}(0)$ is a point $z^+$ which is in the  arc
in $\partial V(a_k(f))$ from $w^-=w(f)$ to $z^-$. Thus, 
$]t(z^+),s(z^-)[ \subset ]\theta'',\theta'[$.
Hence, $a_k(f_-) \in \sector_{f_-} (z_0,t(z^+),s(z^-)) \subset \sector_{f_-} (\overline{V(-a(f_-)},\theta'',\theta')$. 
Along the ray $\cI^+$  we have that the landing point
of $I^+_{f,a_k}(0)$ is $z^+=z(f)$ and the landing point of
$I^-_{f,a_k}(0)$ is a point $z^-$ which is in the arc
in $\partial V(a_k(f))$ from $z^+$ to $w^+=w(f)$.
Since $]t(z^+),s(z^-)[ \subset ]\theta',\theta''[$
we have  $a_k(f_+) \in \sector_{f_+} (z_0,t(z^+),s(z^-)) \subset \sector_{f_+} (\overline{V(-a(f_+))},
\theta',\theta'')$.
That is, we have proven that $a_k(f_+) \in \sector_{f_+} (\overline{V(-a(f_+))},\theta',\theta'')$ and  
$a_k(f_-) \in \sector_{f_-} (\overline{V(-a(f_-))},\theta'',\theta')$. 
Thus the $k$-th symbol of the kneading words of $\kappa(f_+,\theta)$ and $\kappa(f_-,\theta)$ are indeed different.
\end{proof}

\subsection{Type $A$ trekking Theorem}
  \label{s:trekkingA}
\begin{theorem}
  \label{thr:trekkingA}
  Assume that $\cH$ is a type $A$ component
  centered at $f_\star$ and  $\theta'$ is a period $p$ supporting argument for $V(a(f_\star))$.
  The $0$-parameter rays of $\cH$ land at parabolic maps $f_1,f_2$ and the
  $1/2$-parameter ray contained in $\overline{S(\cH,\theta')}$ lands at a pcf
  map $f_0$ such that the following holds,
  modulo relabeling of $f_1$ and $f_2$:
  \begin{enumerate}
  \item For $j=0,1,2$ the map $f_j$ has take-off argument $\theta'+j/3$.
    \item $$\kappa(f_j,\theta'+j/3) =\iota_0 \dots \iota_{p-1} 0$$
  where
  $$\iota_n =
  \begin{cases}
    1, & \text{if }  3^n \theta' \in j/3+]\theta'+1/3,\theta'+2/3[,\\
    0, & \text{otherwise}.
  \end{cases}
$$
  \end{enumerate}
\end{theorem}

\begin{figure}[h]
\begin{center}
\fbox{\includegraphics[width=7cm]{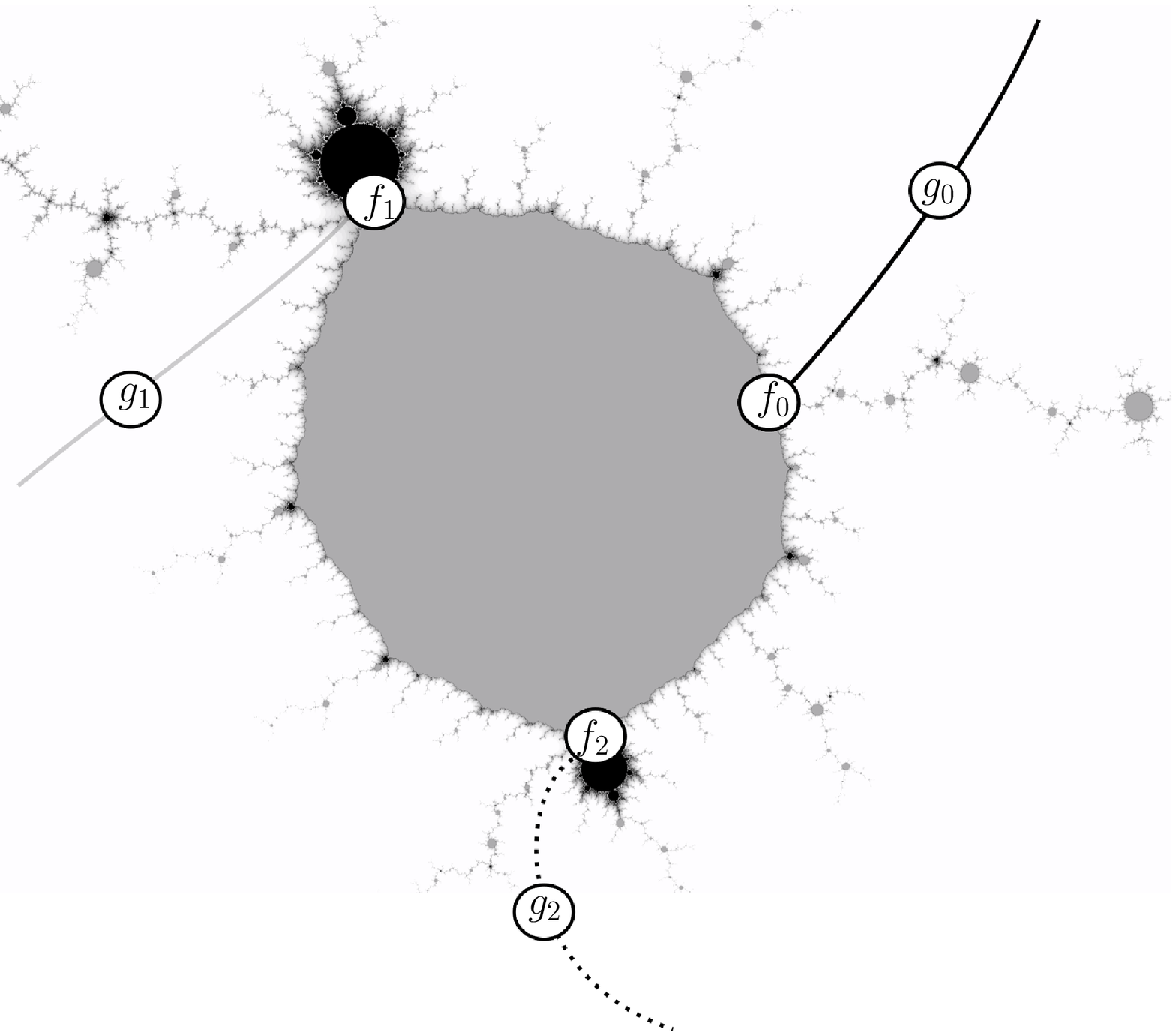} }
\end{center}

\begin{minipage}[c]{.30\linewidth}
\begin{center}
\fbox{\includegraphics[width=4.2cm]{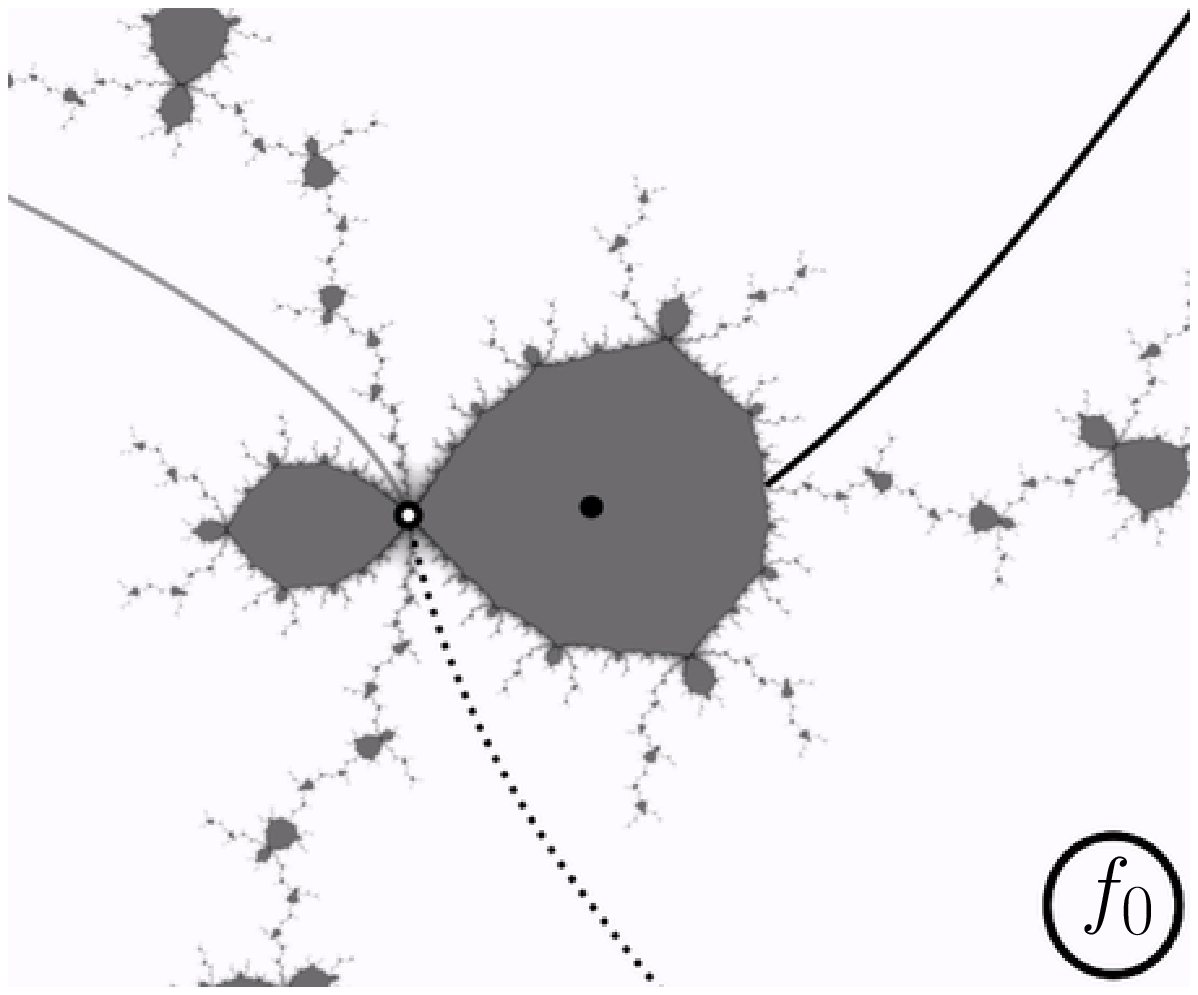} } 
\end{center}
\end{minipage}
\begin{minipage}[c]{.30\linewidth}
\begin{center}
\fbox{\includegraphics[width=4.2cm]{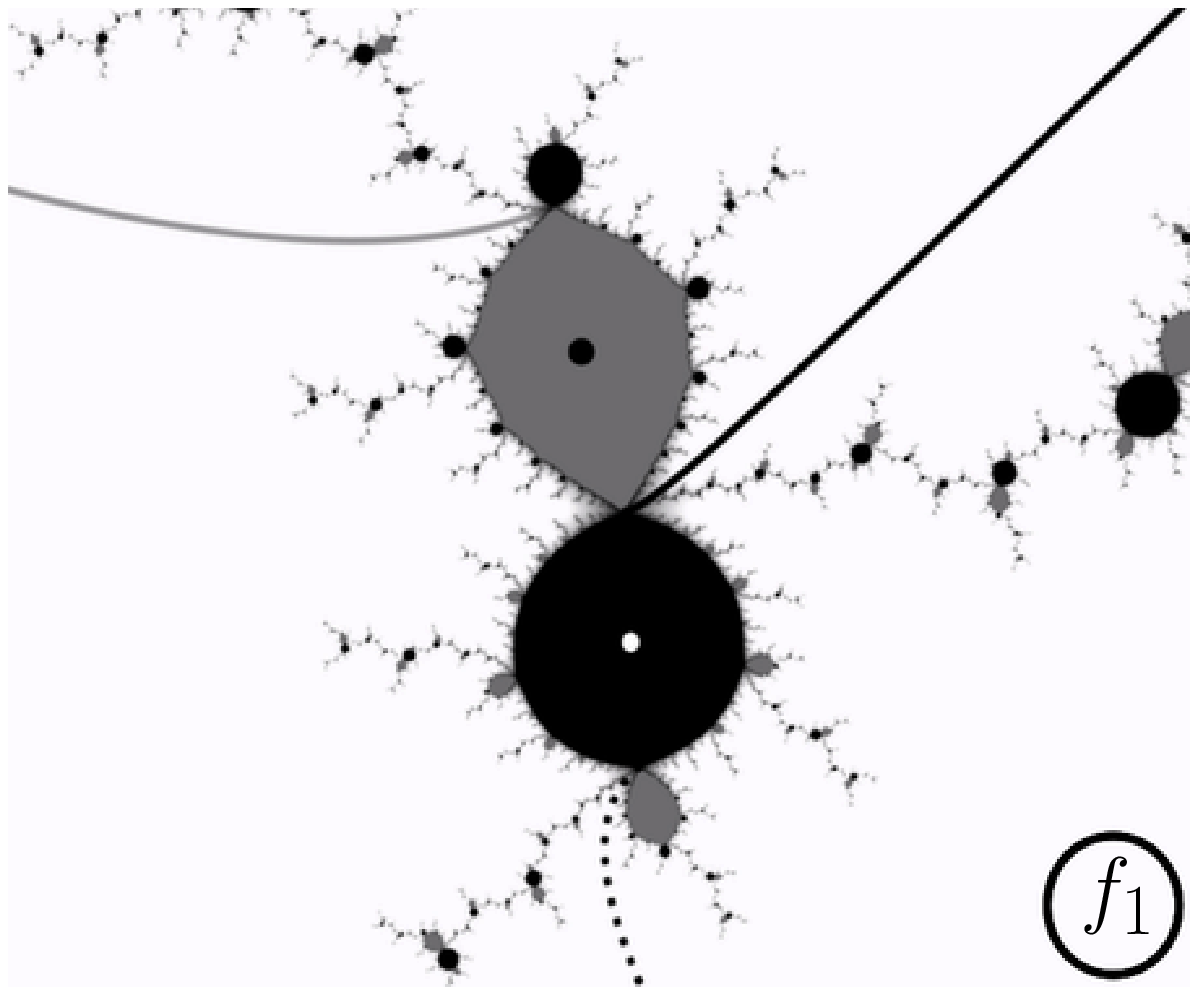} } 
\end{center}
\end{minipage}
\begin{minipage}[c]{.3\linewidth}
\begin{center}
\fbox{\includegraphics[width=4.2cm]{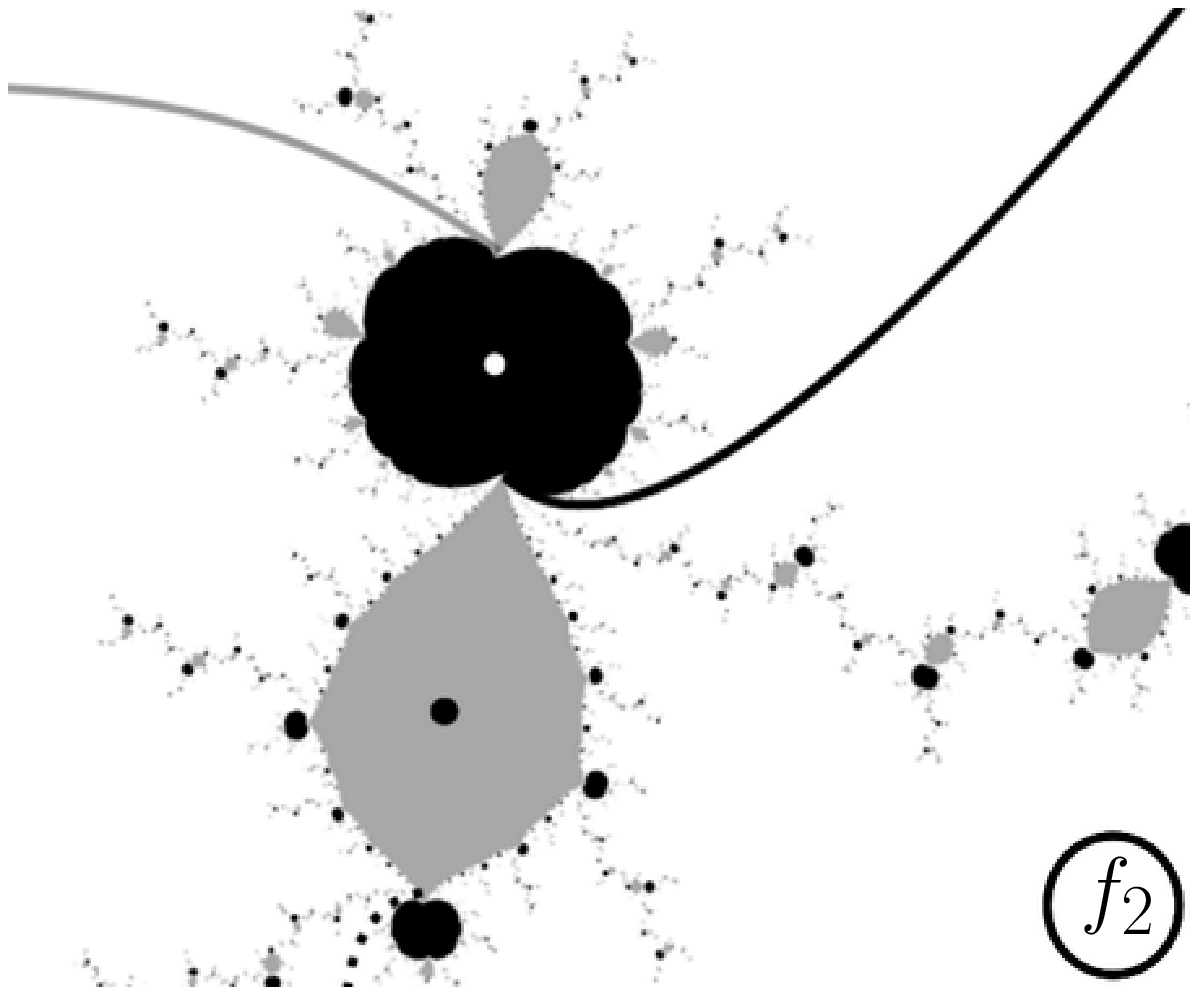} } 
\end{center}
\end{minipage}
\begin{minipage}[c]{.3\linewidth}
\begin{center}
\fbox{\includegraphics[width=4.2cm]{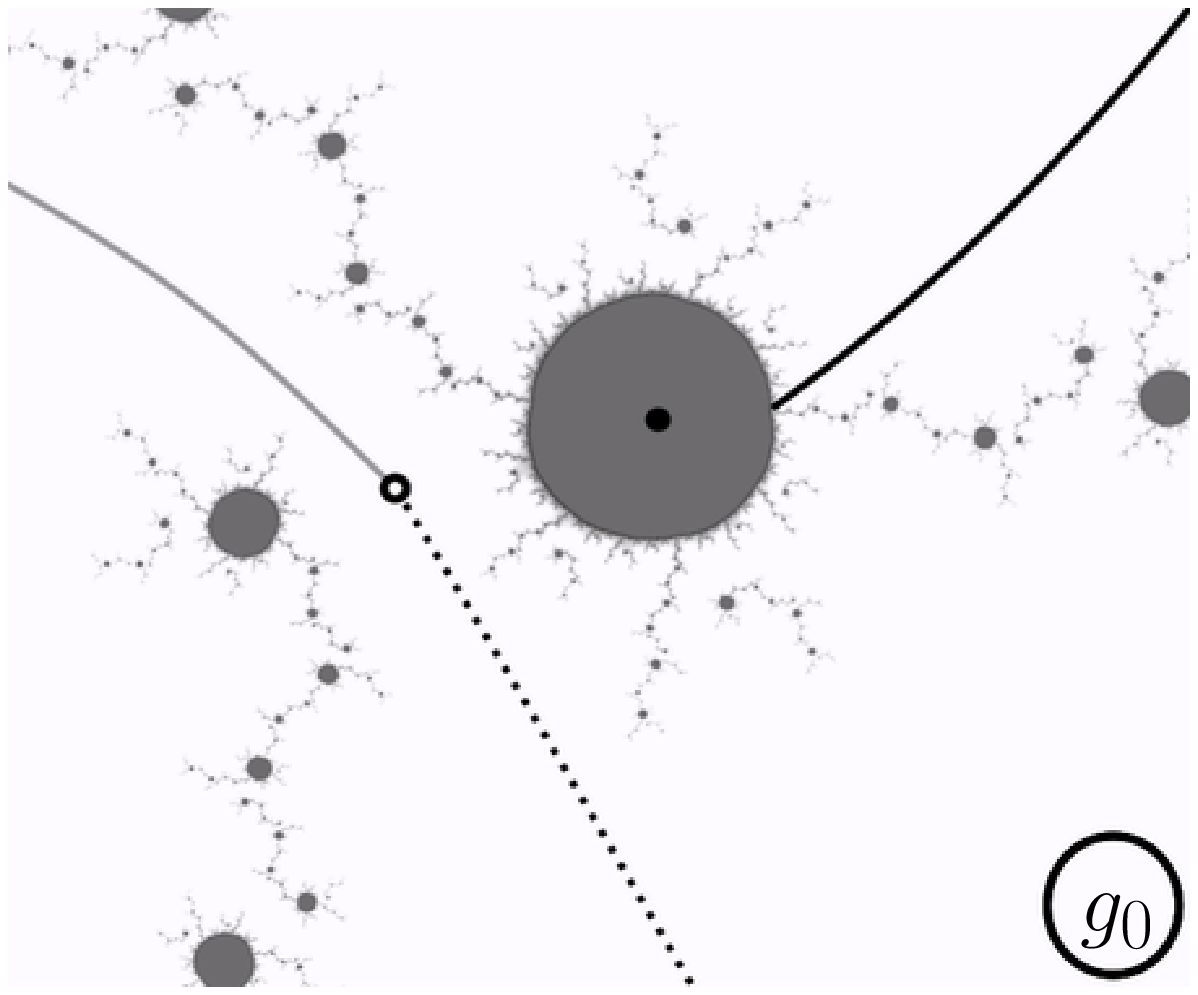} } 
\end{center}
\end{minipage}
\begin{minipage}[c]{.3\linewidth}
\begin{center}
\fbox{\includegraphics[width=4.2cm]{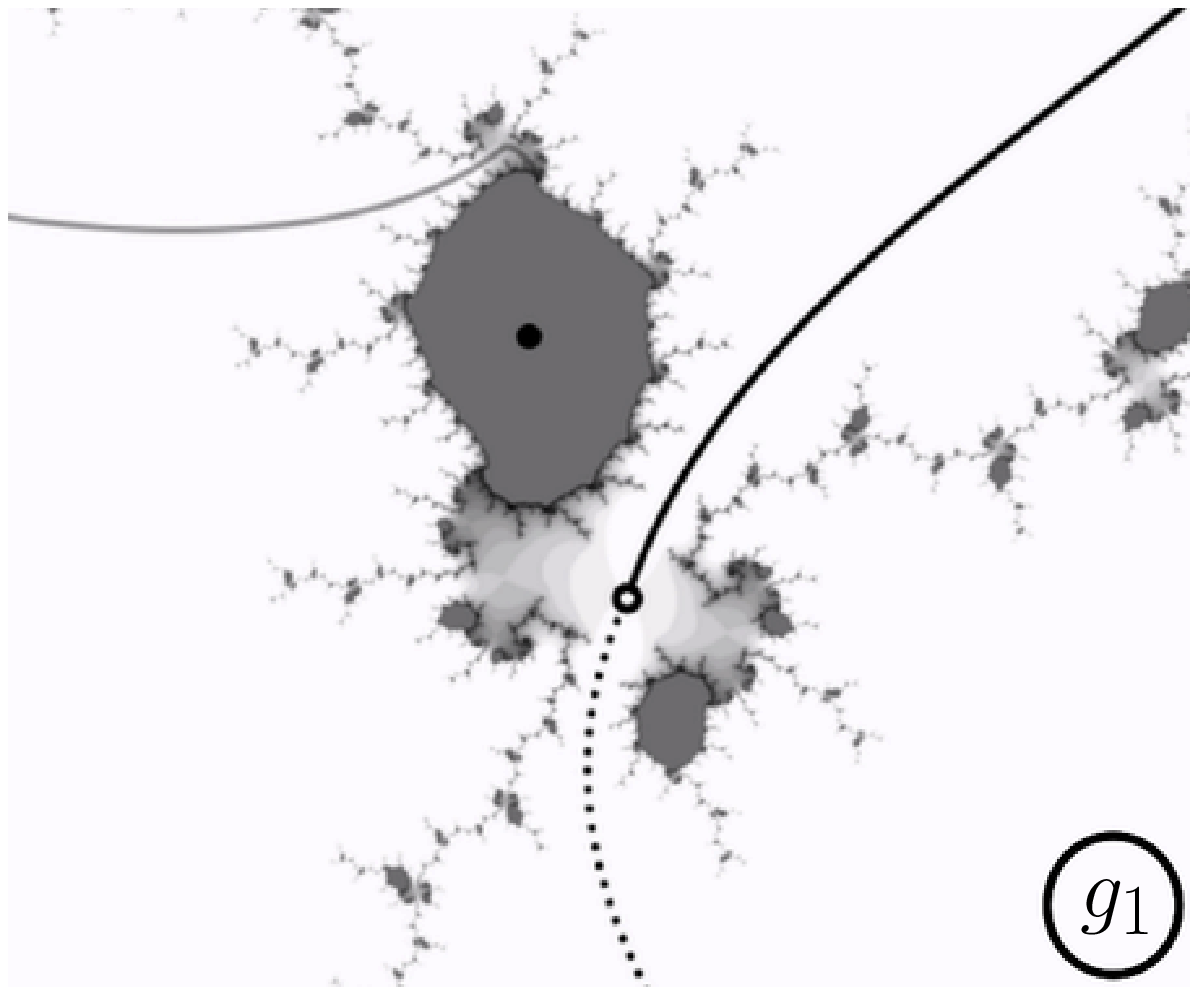} } 
\end{center}
\end{minipage}
\begin{minipage}[c]{.3\linewidth}
\begin{center}
\fbox{\includegraphics[width=4.2cm]{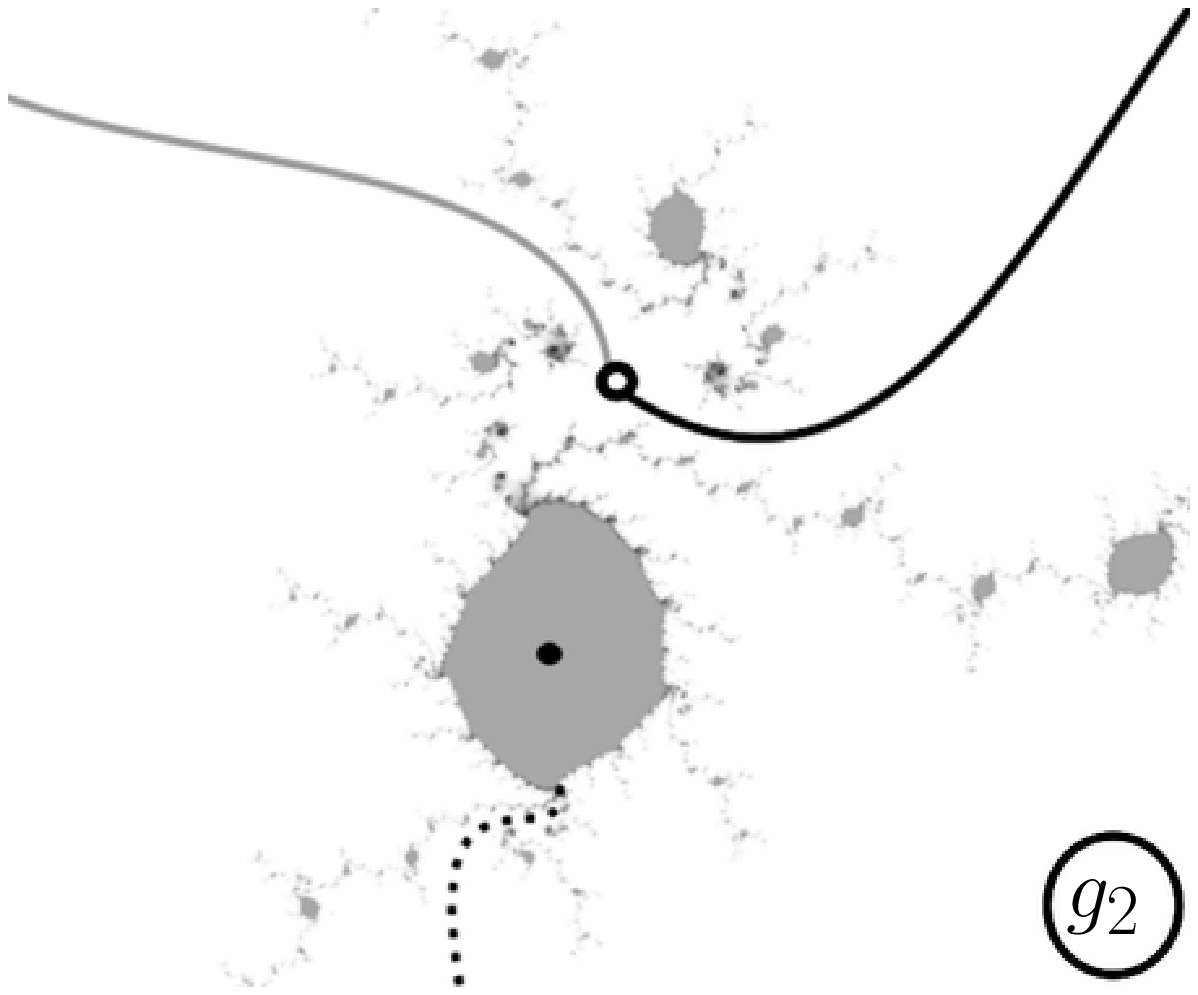} } 
\end{center}
\end{minipage}
\caption{The upper  illustration is centered at a type $A$ hyperbolic component. The black parameter ray has a
  $p$-periodic argument $\theta$ landing at a co-root. The other ones have arguments $\theta\pm1/3$ and land at the roots. Details of the corresponding dynamical planes are illustrated below where the white dot corresponds to $-a$ and the black one to $+a$.
 In the first row, the position of $+a$ with respect to these rays illustrates Theorem \ref{thr:trekkingA}. In the second, it illustrates theorems~\ref{thr:take-off-parabolic} and~\ref{thr:take-off-parabolic}.
}
\label{trekingA}
\end{figure}

After proving the previous theorem, in Corollary~\ref{cor:trekkingA}, we deduce that we can choose to cross from one root or co-root 
to another one in order to switch the symbols  $0, 1$
of an initial segment of the kneading
words involved. See Figure~\ref{trekingA}.

\begin{proof}
    Consider the sector $S(\cH,\theta')$ which according to  Proposition~\ref{sectorH} is bounded by
$0$-parameter rays $\cI^\pm$. Let $f_1$ be the landing point of $\cI^+$ and $f_2$ the landing point of $\cI^-$.

Consider $\fstar \neq f \in \cI^+$ and denote by $z^\pm(f)$ the landing point of
$I^\pm_{f,a}(0)$.
According to Proposition~\ref{sectorH},  
$R_f(\theta')$ lands at $z^+(f)$.
By \S~\ref{s:landing0-internal}, the open arc in $\partial V(a(f))$ from $z^-(f)$ to $z^+(f)$ contains a unique point $w^+(f)$ with image $f(z^+(f))$. Necessarily, $R_f(\theta'+2/3)$ lands at $w^+(f)$.
Since $\theta'$ is supporting, in the notation of \S~\ref{s:landing0-internal} (3),
we have that $\theta' = s(z^+(f))$ or $t(z^+(f))$ and $\theta'+2/3 = s(w^+(f))$ or $t(w^+(f))$, respectively. From Proposition~\ref{internal-landing-0} (8),
it follows that $\theta'+1/3$ is a take-off argument for $f_1$, which proves
assertion (1) for $j=1$. 
Moreover, points $a_n(f_1)$ in
the region $\sector_{f_1}(V(-a(f_1)),\theta'+2/3,\theta')$ correspond to symbols ``$1$'' in
the kneading word 
$\kappa(f_1, \theta'+1/3)$. Equivalently, since $R_{f_1} (3^n \theta')$ lands in $\partial V(a_n(f_1))$, 
symbols ``$1$'' correspond to rays $R_{f_1} (3^n \theta')$ contained in this region, that is,
to $3^n \theta' \in 1/3 + ]\theta'+1/3,\theta'+2/3[$. Thus we have proven the theorem for $j=1$.

A similar analysis for $\fstar \neq f \in \cI^-$
with the aid of propositions~\ref{sectorH} and~\ref{internal-landing-0}
proves that the landing point $f_2$ of $\cI^-$ has the claimed properties.

Now consider the landing point $f_0$ of the $1/2$ internal ray contained in the sector $S(\cH,\theta')$.
By Proposition~\ref{internal-landing-half} (4), we have that  $\theta'$ is a take off argument for $f_0$ with corresponding
ray landing at $\partial V(a(f_0))$. Thus $\kappa (f_0, \theta')$ is as described in the statement of the theorem.
\end{proof}

\begin{corollary}
  \label{cor:trekkingA}
   Assume that $\cH$ is a type $A$ component
  centered at $f_\star$ and $\theta'$ is a period $p$ supporting argument for $V(a(f_\star))$. Let $g_0$ be the landing point of a $0$ or $1/2$ ray in 
  $\overline{S(\cH,\theta')}$ with take-off argument $\vartheta \in m_3^{-1}(3\theta')$ and kneading word $\kappa(g_0,\vartheta) =1^{\mu_0-1} 0 \, \iota_{\mu_0+1} \cdots 0$.
  Then there exists a  landing point $g_1$ of a $0$ or $1/2$ ray in 
  $\overline{S(\cH,\theta')}$ with take-off argument $\vartheta' \in m_3^{-1}(3\theta')$ and kneading word $\kappa(g_1,\vartheta') =0^{\mu_0-1} 1 \,\iota'_{\mu_0+1} \cdots 0$.
\end{corollary}

\begin{proof}
  Given $g_0$ we have that $3^{\mu_0} \vartheta \in ]\vartheta-1/3,\vartheta[$ or $]\vartheta,\vartheta+1/3[$. In the first case, choose $g_1$ with take-off argument $\vartheta'=\vartheta+1/3$ and in the latter with take-off argument $\vartheta'=\vartheta-1/3$.
  From the previous theorem, it follows that $\kappa(g_1,\vartheta')$ is
  as desired.
\end{proof}

\subsection{Parabolic immigration}
\label{s:parabolic-immigration}

\begin{theorem}
  \label{thr:parabolic-immigration}
  Let $f_0$ be a parabolic map with  take-off argument $\theta$.
  Assume that the periodic argument $\theta' =\theta+1/3$ or $\theta-1/3$ is a
  period $p$ relatively supporting argument for $V(a_k(f_0))$. 
  Then there exists a type $A$ or $B$ hyperbolic component $\cH$ centered at a  map $\fstar$
  such that the following hold:

  (1) $a_k(\fstar) = -a(\fstar)$,

  (2)  the rays $R_\fstar (\theta \pm 1/3)$ support  $V(-a(f_\star))$,

  (3) $f_0$ is the landing point of a $0$-internal ray $\cI$ of $\cH$ contained in  $\partial S(\cH,\theta')$.
\end{theorem}

Note that by Theorem~\ref{thr:arrival-periodic-connection}
the hypothesis of the previous theorem are fulfilled by maps $f_0$
which are the landing points
of parameter rays $R_\cU(\theta)$ with a ray connection between $-a(f)$ and
$a_k(f)$ for  $f \in R_\cU(\theta)$.

In order to prove the theorem above we first identify  parabolic maps in
$\cS_p$ with the aid of critical portraits.
Given a map $g \in \cS_p$ with a parabolic cycle we say that 
the ordered pair of pairs of arguments in $\QS$:
$(\{\beta',\beta''\}, \{\vartheta',\vartheta''\})$ is a \emph{left (resp. right) supporting
  critical marking for $g$} if the following hold:

(1) $\beta'\neq \beta''$ and $3\beta'= 3\beta''$, $\vartheta'\neq \vartheta''$ and $3\vartheta'= 3\vartheta''$.

(2) $R_g(\beta')$ lands at the parabolic cycle. $R_g(\beta')$ and $R_g(\beta'')$ left (resp. right) support
$V(-a(g))$.

(3) The internal ray $I_{g,a}(0)$ and $R_g(\vartheta')$ land at the same point.
$R_g(\vartheta')$ and $R_g(\vartheta'')$ left (resp. right) support $V(a(g))$.

\begin{lemma}
  Let $g_0,g_1 \in \cS_p$ be maps with a parabolic cycle such that $g_0$ and $g_1$ have the same left (resp. right) supporting critical marking. Then $g_0=g_1$.
\end{lemma}

The proof of the lemma presented here relies on Ha\"issinsky's deformation of geometrically finite polynomials~\cite{Haissinsky2000} and in Poirier's uniqueness result for
critically marked pcf polynomials~\cite{PoirierCP}.
However, using the pull-back argument, one could give a longer but more direct proof that bypasses the use of Ha\"issinsky's deformation and show, in the spirit of Poirier's work, that $g_0$ and $g_1$ are conjugate by a
quasiconformal map $h$ which is conformal in the Fatou set and conclude that $g_0=g_1$.

\begin{proof}
  For simplicity of the exposition we assume that the arguments
  involved are right supporting.

According to
Ha\"issinsky~\cite{Haissinsky2000}, for $i=0,1$,
there exist a continuous family
  $G_{i,\eta}\in \cS_p$ parametrized by
  $\eta \in [0,1]$ such that
  $G_{i,\eta}$ is hyperbolic for all $\eta \in [0,1[$, $G_{i,1}=g_i$,
  and the Caratheodory loop of $G_{i,\eta}$ is independent of
  $\eta$. The maps $G_{i,\eta}$ for $\eta <1$ have one attracting cycle
  converging to the parabolic cycle of $g_i$ as $\eta \to 1$. The multiplier of
  this attracting cycle converges to $1$. 

  In particular, $ \beta', \beta''$ are right
  supporting arguments for $V(-a(G_{i,\eta}))$ and $\vartheta', \vartheta''$
  are right supporting arguments for $V(a (G_{i,\eta}))$.  

  Note that $G_{i,\eta}$ belongs to a hyperbolic component $\cH_i$ of
  disjoint type, for $i=0,1$ and $\eta<1$. 
  According to Milnor $\cH_i$ contains a unique pcf polynomial $g_{i,\star}$~\cite{MilnorHyperbolicComponents}.
  Since the Caratheodory loops of $g_{i,\star}$ and $g_i$ agree,
  an admissible right supporting critical portrait for $g_{i,\star}$ is given by
  the pairs  $\{ \vartheta', \vartheta'' \}$ and $ \{ \beta', \beta''\} $.
  According to Poirier it follows that $g_{0,\star}= g_{1,\star}$. Therefore $\cH_0 =\cH_1$. Since the hyperbolic component $\cH_0= \cH_1$ is parametrized by
  the multiplier of the attracting cycle (not equal to the orbit of $+a$) it follows that $g_0=g_1$.   
\end{proof}

\begin{proof}[Proof of Theorem~\ref{thr:parabolic-immigration}]
  Suppose that $f_0$ is a parabolic map  as in the statement of
  the theorem. That is, $R_{f_0}(\theta')$ is a period $p$ ray landing at a parabolic periodic point $z_0$ relatively supporting $V(-a(f_0))$ and
$V(a_k(f_0))$. Here $\theta'=\theta +1/3$ or $\theta-1/3$, and $R_{f_0}(\theta\pm1/3)$ lands in $\partial V(-a(f_0))$.

It will be convenient to consider the  period $p$ arguments
$\alpha', \beta', \gamma' \in \QS$ such that the corresponding rays land at
$z_0$ and  the sectors at $z_0$ containing $a_k(f_0)$ and $-a(f_0)$ are
$\sector_{f_0}(z_0,\alpha',\beta')$ and $\sector_{f_0}(z_0,\beta',\gamma')$, maybe not respectively.
Obviously, $\theta'= \alpha'$ or $\beta'$ or $\gamma'$. 
If $-a(f_0) \in \sector_{f_0}(z_0,\alpha',\beta')$, then $\beta'$ is right supporting
for $-a(f_0)$. Otherwise, $-a(f_0) \in \sector_{f_0}(z_0,\beta',\gamma')$ and $\beta'$ is left supporting for $-a(f_0)$.

Denote by $w_0\neq z_0$ the point in $\partial V(-a(f_0))$ such that
$f_0(w_0) =f_0(z_0)$. Denote by $\alpha'',\beta'',\gamma''$ the arguments of the rays landing at $w_0$ such that $3 \alpha'' = 3 \alpha'$, $3 \beta''
= 3 \beta'$, $3 \gamma'' = 3 \gamma'$.
For $\theta, \theta'$ as in the statement of the theorem
it follows that $\theta'=\alpha',\beta'$ or $\gamma'$ and
$\{\theta+1/3,\theta-1/3\}$ is 
$\{ \alpha',\alpha''\}, \{ \beta',\beta''\},$ or $\{ \gamma',\gamma''\}$, respectively.

Observe that if $\beta'$ is left supporting for $-a(f_0)$, then
$$\gamma', \alpha', \beta', \gamma'', \alpha'', \beta'' $$
are in cyclic order.
Otherwise, if $\beta'$ is right supporting for $-a(f_0)$, then
$$\beta', \gamma', \alpha', \beta'', \gamma'', \alpha'' $$
are in cyclic order.

\smallskip
To find a hyperbolic component with $f_0$ in its boundary, loosely speaking, we
``unpinch'' $\beta'$. To formalize this idea we use the language of laminations.
According to~\cite{kiwi04}, a closed equivalence relation $\lambda$ in $\RZ$ such that all equivalence classes are finite is called a \emph{real lamination of degree~$3$} if the following holds:
\begin{itemize}
\item If $A$ is an equivalence class then $m_3(A)$ is an equivalence class.
\item If $A$ is an equivalence class and $]s,t[$ is a connected component of $\RZ \setminus A$, then $]3s,3t[$ is a connected component of $\RZ \setminus m_3(A)$ (i.e. $m_3$ is consecutive preserving on classes).
  \item If $A$ and $B$ are distinct equivalence classes, then $A$ is contained in a connected component of $\RZ \setminus B$ (i.e. classes are \emph{unlinked}). 
\end{itemize}
Given a map $f \in \cC(\cS_p)$ with locally connected Julia set,
the equivalence relation $\lambda(f)$ in $\RZ$ that identifies $s$ and $t$ if and only if $R_{f}(s)$ and $R_{f}(t)$ land at the same point is a real lamination.
We say that $\lambda(f)$ is the \emph{real lamination of $f$}.

Since $f_0$ is geometrically finite its Julia set is locally connected
and therefore  $\lambda(f_0)$ is a real lamination.
Let $\lambda$ be the equivalence relation in $\RZ$
that identifies two distinct arguments $s,t$ if and only if $s,t$ are
$\lambda(f_0)$-equivalent and both $s$ and $t$ are not in the grand orbit of $\beta'$.
It is not difficult to check that $\lambda$ is a real lamination.
For this real lamination $\lambda$ we have the following:

\begin{lemma} There exists a hyperbolic pcf polynomial
  $f_\star \in \cS_p$ such that
  $\lambda = \lambda(\fstar)$,  $a_k(\fstar)=-a(\fstar)$ and the rays with arguments $\alpha',\beta',\gamma', \alpha'', \beta'',\gamma''$ land in $\partial V(a_k(\fstar))$.
  Moreover, $\alpha'$ is left supporting, $\gamma'$ is right supporting and,
  $\beta'$ is both left and right supporting for $V(a_k(\fstar))$. 
\end{lemma}

\begin{proof}
  Since every $\lambda(f_0)$-class maps bijectively onto its image the same occurs with $\lambda$-classes. In particular, the action of $m_3$ on $(\RZ)/\lambda$
  has no periodic topological circle with bijective return map.
  Thus, we may apply~\cite[Theorem~1, Lemma 6.34]{kiwi04},
  to obtain a polynomial $f$ such that
  $\lambda(f)= \lambda$,  every cycle of $f$ is repelling or superattracting
  and every Fatou critical point is eventually periodic.
  Since the Julia set of $f$ is critical point free,
  $f$ is a pcf hyperbolic map with locally connected Julia set naturally homeomorphic to $(\RZ)/\lambda$. In a locally connected
Julia set, the landing point of $R_f(s)$ and $R_f(t)$ lie in the boundary of the same bounded Fatou component if and only there is no $\lambda(f)$-class $A$
such that $s,t$ are in different components of $\RZ \setminus A$.
Hence if $R_{f_0}(s)$ and $R_{f_0}(t)$ land in the boundary of a bounded Fatou component of $f_0$, then the corresponding rays $R_{f}(s)$ and $R_{f}(t)$ land
in the boundary of some bounded Fatou component of $f$. 
It follows that $\alpha',\beta',\gamma', \alpha'', \beta'',\gamma''$ land in the boundary of a Fatou component $V_f(\omega)$
which has to be periodic and contain a critical point $\omega$.

We claim that the period of $V_f(\omega)$ is $p$. It will be convenient to introduce the Caratheodory loops $\gamma: \RZ \to J(f)$ and $\gamma_0 : \RZ \to J(f_0)$ that map an argument $t$ to the landing point of the ray $R_f(t)$ and $R_{f_0}(t)$, respectively. 
Let $C \subset \RZ$ be the set formed by all arguments $t$ such that
$\gamma_0 (t)$ belongs to
$\partial V_{f_0}(-a(f_0))$ or to $ \partial V_{f_0}(a_k(f_0))$. From the previous paragraph, $\gamma(C) \subset \partial V_f(\omega)$. We conclude that
$V_f(\omega)$ has period at most $p$.
To rule out the possibility of a period
strictly smaller than $p$, let $0 \le j <p$ be such that $]3^j \alpha',3^j \gamma'[$ has minimal length. It is sufficient to show that the region between the rays with arguments $3^j \alpha'$ and $3^j \gamma'$ contains the
component $V_{f_0}(a_j(f_0))$ but no other component in its cycle.
By contradiction suppose that
$Z_{f_0} (f_0^j(z_0), 3^j \alpha', 3^j \gamma')$ contains
$V_{f_0}(a_n(f_0))$ for some $n \neq j$, then choose $n$ such that
the $]3^n \gamma', 3^n \alpha'[$ has minimal length and
$]3^n \gamma', 3^n \alpha'[ \subset ]3^j \alpha', 3^j \gamma'[$.
Since every sector contains a postcritical point, the region 
$Z_{f_0} (f_0^n(z_0), 3^n \gamma', 3^n \alpha')$
must contain a postcritical point say $a_m(f_0)$ with $m \neq n,j$.
Therefore, either  $]3^n \alpha', 3^n \gamma'[$
or $]3^m \gamma', 3^m \alpha'[$  is contained in $]3^j \alpha', 3^j \gamma'[$,
which contradicts the choice of $j$ and $n$ and implies that the period of $V_f(\omega)$ is $p$ as claimed.

\smallskip
Recall that $\gamma(C) \subset \partial V_f(\omega)$. If $k=0$ we conclude that $\omega$ is a double critical point and $f$ is conjugate to a hyperbolic pcf map $f_\star \in \cS_p$ with $-a(f_\star) = a_0(f_\star)$. If $k \neq 0$, then $V_f(\omega)$ is the
image under the $k$-th iterate of a critical Fatou component $V_f(\omega')$
and $f$ is conjugate to a map
$f_\star \in \cS_p$ such that the labeling of the critical points $\omega$ and $\omega'$ correspond to $-a(f_\star)$ and
$a(f_\star)$.

The fact that the arguments $\alpha', \beta'$ and $\gamma'$ have the supporting properties of the statement is a direct consequence of the definition of $\lambda$.
\end{proof}

To finish the proof of Theorem~\ref{thr:parabolic-immigration}
let us assume that $\beta'$ is left supporting for $-a(f_0)$.
The case in which is left supporting follows along similar lines.
Under this assumption, $\vartheta'=3^{p-k} \alpha'$ is left supporting for $V(a(f_0))$.
Let $\vartheta'' = \vartheta'+1/3$ or $\vartheta'-1/3$ be such that
$\vartheta'$ and $\vartheta''$ are in the same component of $\RZ \setminus \{\beta',\beta''\}$.
It follows that
$(\{\beta',\beta''\}, \{\vartheta',\vartheta''\})$ is an admissible
critical marking for $f_0$. 

Let $\cH$ be the hyperbolic component centered at the map $\fstar$ given by the lemma.  Given $f \in \cH$,   denote by $z_\alpha(f)$
(resp. $z_\beta(f)$, $w_\alpha(f)$, $w_\beta(f)$) the landing point of the ray of $f$ with argument $\alpha'$ (resp. $\beta'$, $\alpha''$, $\beta''$).
Then $z_\beta(f)$ is a successor of 
$z_\alpha(f)$ among the $f^p$-fixed points in $\partial V(a_k(f))$.
By Proposition~\ref{sectorH}, there exists a unique $0$-parameter internal ray $\cI$ of $\cH$ such that for all $f \neq \fstar$ along $\cI$
the dynamical  ray  $I^-_{f,a_k}(0)$
lands at $z_\alpha(f)$. The same proposition guarantees that
$I^+_{f,a_k}(0)$ lands at $z_\beta(f)$.
From Proposition~\ref{internal-landing-0}, it follows that the landing point $g$ of this ray
$\cI$ is a parabolic map  such that the ray  $R_g(\beta')$
lands at a parabolic point $z_0(g)$, it is right supporting for
$V(-a(g))$, and the ray $R_g(\beta'')$ lands at
$w_0 (g) \in \partial V(-a(g))$.  Moreover, $R_g(\alpha')$ lands at
$z_0(g)$ and it is right supporting for $a_k(g)$.
It follows that $(\{\beta',\beta''\}, \{\vartheta',\vartheta''\})$ is also
an admissible
critical marking for $g$. Thus $g=f_0$.

To finish the proof we check that $\cI$ is one of the
rays $\cI^\pm$ that bound the sector $S(\cH,\theta')$.
Indeed, applying Proposition~\ref{sectorH}, when $\theta' =\alpha'$ or $\gamma'$,  we conclude that $\cI = \cI^-$, otherwise, $\theta'=\beta'$ and it follows that $\cI = \cI^+$. 
\end{proof}

\subsection{Prerepelling immigration}
\label{s:pcf-immigration}

\begin{theorem}
  \label{thr:prerepelling-immigration}
  Let $f_0 \in \partial \cC(\cS_p)$ be a pcf map with take-off argument $\theta$.
  Suppose that $\theta$ is also a period $p$ argument which supports $V(a(f_0))$. 
  Then there exists a type $A$ hyperbolic component $\cH$ centered at a map $\fstar$
  such that the following hold:

  (1)  $R_\fstar (\theta)$ supports  $V(-a(f_\star))$,

  (2) the $1/2$-internal ray $\cI$ of $\cH$ contained in 
$\overline{S(\cH,\theta)}$ lands at $f_0$.
\end{theorem}

Observe that due to Theorem~\ref{thr:arrival-preperiodic-connection}
a landing point $f_0$ of a parameter ray with a preperiodic ray connection between $-a(f)$ and $a(f)$ fulfills the hypothesis of the previous theorem. 

\begin{proof}
Consider a map $f_0$ as in the statement of Theorem~\ref{thr:prerepelling-immigration}. For simplicity of the exposition we assume that the take-off argument 
$\theta'$ is left supporting for $V(a(f_0))$. Note that the rays $R_{f_0}(\theta' \pm 1/3)$ land at $-a(f_0) \in \partial V(a(f_0))$. Also $\theta'-1/3$ is left supporting for $V(a(f_0))$. The map $f_0$ is pcf so is completely determined by a left supporting critical marking~\cite{PoirierCP}. That is, if a pcf map $g \in \cS_p$ 
is such that $\{\theta',\theta'-1/3\}$ are left supporting arguments for  $V(a(g))$ and $\{\theta'+1/3,\theta'-1/3\}$ land at $-a(g)$, then $g=f_0$.

Either from Poirier's work or from the combinatorics of the multibrot set $\cM_3$, which is the connectedness locus of the family $z \mapsto z^3 +c$, it follows that there exists $f_\star (z) = z^3 + c_\star$ such that the critical point 
$\pm a(f_\star)=0$ has period $p$ and $\theta'$ is supporting for $V(-a(f_\star))$. In fact, one can use the combinatorics of $f_0$ to show that if an argument $t$ has the same itinerary than $\theta'$ according to the partition $[\theta',\theta'+1/3[, [\theta'+1/3,\theta'+2/3[, [\theta'+2/3,\theta'[$ of $\RZ$, then
$t=\theta'$. That is, $\{\theta', \theta'+1/3,\theta'+2/3\}$ is a left supporting critical marking of a pcf map $f_\star$ according to Poirier. Alternatively, using the results in~\cite{RationalMultibrot} it follows that the multibrot
external ray with argument $3\theta'$ lands at the root of a hyperbolic component (of the family $z \mapsto z^3 +c$) centered at a pcf map $f_\star$ with the required properties. 

Let $\cH \subset \cS_p$ be the
hyperbolic component centered at $f_\star$ and 
$\cI$ its $1/2$-internal ray contained in  $\overline{S(\cH,\theta')}$.
Let $g$ be the landing point of $\cI$. 
In the notation of \S~\ref{s:landing-half}, we have $\theta'=s(z)$. From
Proposition~\ref{internal-landing-half} we conclude that $\{\theta',\theta'-1/3\}$ are supporting arguments for $+a(g)$ and $\{\theta'+1/3,\theta'-1/3\}$ land at $-a(g)$. Therefore $g=f_0$.
\end{proof}

\section{Proof of Theorem~\ref{thr:main}}
\label{s:final}

Consider an escape component $\cU$ with undistinguished kneading word $\kappa$ (i.e. $\kappa \neq 1^{p-1}0$).
Let $\mu$ denote the maximal return time to $D_0$.
According to Theorem \ref{thr:ray-connection}, there exists $f\in\cU$ with a ray connection $R_f^\sigma(\theta')$ between $-a(f)$ and $a_k(f)\in D_0$ such that the return time of $a_k(f)$ to $D_0$ is maximal. Denote by $\cR_\cU(\theta)$ the parameter ray containing $f$. Observe that
$\theta'=\theta+1/3$ or $\theta-1/3$.


According to Theorem \ref{thr:arrival-periodic-connection} and Theorem \ref{thr:arrival-preperiodic-connection}, the ray $\cR_\cU(\theta)$ lands at a parabolic or a pcf map $f_0$ such that $\kappa(f_0,\theta)=\kappa$. The map $f_0$ is pcf only when $k=0$ and $\theta'$ is preperiodic.

From Theorem \ref{thr:parabolic-immigration} when $\theta'$ is periodic or from Theorem \ref{thr:prerepelling-immigration} when $\theta'$ is preperiodic,
$f_0$ is in the boundary of a type $A$ or $B$ hyperbolic component $\cH$.
The component is of type $A$ if $k=0$ and of type $B$ if $k \neq 0$.

If $k \neq 0$, we may apply Theorem \ref{thr:trekkingB} to find in  $\partial \cH$ a parabolic map $f_1$
with take-off argument $\theta$ such that $\kappa':=\kappa(f_1,\theta)$ is obtained from $\kappa(f_0,\theta)$ by a type $B$ move.
If $k=0$, we apply Corollary \ref{cor:trekkingA} to find in  $\partial \cH$ a parabolic or pcf map $f_1$ with take-off argument $\vartheta = \theta+1/3$
or $\theta-1/3$ such that $\kappa':=\kappa(f_1,\vartheta)$ is obtained from $\kappa(f_0,\theta)$ by a type $A$ move.


Applying Theorem \ref{thr:take-off-parabolic} when $f_1$ is a parabolic map and, theorems~\ref{thr:bms} and~\ref{thr:arrival-preperiodic-connection} when $f_1$ is a pcf map
we obtain a new escape region  $\cU'$ such that $\kappa(\cU')=\kappa'$ and $f_1\in\partial \cU'$.
That is, $\cU'$ and $\cU$ are in the same connected component of $\cS_p$.

Recall that there is only one escape region with the distinguished kneading word $1^{p-1}0$.
The theorem then follows from the  lemma below
which shows that after successively applying at most $p^2+p$ times type $A$ or $B$ moves  
we obtain the distinguished kneading word $1^{p-1}0$.

\begin{lemma}
  \label{lem:typeAB}
  Consider any sequence $\kappa_1, \dots, \kappa_{\ell}$ of kneading words such that
  $\kappa_{i+1}$ is obtained from $\kappa_i$ by a type $A$ or a type $B$ move.
  Then $\ell \le p^2+p$. 
\end{lemma}

\begin{proof}
  Let us introduce two numbers $b(\kappa)$ and $w(\kappa)$ associated to a word $\kappa = \iota_1 \dots \iota_{p-1} 0$.
  The weight $w(\kappa)$ is simply defined as the number of symbols $1$ in $\kappa$.
  For $\kappa \neq 0^p$ let $n \ge 0$, $m \ge 1$ be such that
  $\kappa = 0^n 1^m 0 \cdots,$
  and define $b(\kappa) = n+m +1$. Let  $b(0^p)=0$. It follows that
  $0 \le b(\kappa) \le p$ and $0 \le w(\kappa) \le p-1$. That
  is,
  $(b(\kappa),w(\kappa)) \in \{0, \dots, p\} \times \{0, \dots,
  p-1\}.$ It is not difficult to check that if $\kappa'$ is obtained
  via a type $A$ or $B$ move from $\kappa \neq 1^{p-1}0$ then, in the
  lexicographic order, $(b(\kappa'), w(\kappa'))$ is greater than
  $(b(\kappa), w(\kappa))$. Moreover, $1^{p-1}0$ is the unique maximal
  kneading word with respect to this order. The lemma follows.
\end{proof}

\bibliographystyle{alpha}
\bibliography{/Users/jkiwi/Dropbox/Papers/index2020}

\begin{thebibliography}{BKM10}

\bibitem[AK20]{ArfeuxKiwiDCDS}
Matthieu Arfeux and Jan Kiwi.
\newblock Topological cubic polynomials with one periodic ramification point.
\newblock {\em Discrete Contin. Dyn. Syst.}, 40(3):1799--1811, 2020.

\bibitem[AY09]{AspenbergYampolskyS2}
Magnus Aspenberg and Michael Yampolsky.
\newblock Mating non-renormalizable quadratic polynomials.
\newblock {\em Comm. Math. Phys.}, 287(1):1--40, 2009.

\bibitem[BEK18]{buff2018rational}
Xavier Buff, Adam~L. Epstein, and Sarah Koch.
\newblock Rational maps with a preperiodic critical point, 2018.

\bibitem[BF14]{BrannerFagellaBook}
Bodil Branner and N\'{u}ria Fagella.
\newblock {\em Quasiconformal surgery in holomorphic dynamics}, volume 141 of
  {\em Cambridge Studies in Advanced Mathematics}.
\newblock Cambridge University Press, Cambridge, 2014.
\newblock With contributions by Xavier Buff, Shaun Bullett, Adam L. Epstein,
  Peter Ha\"{\i}ssinsky, Christian Henriksen, Carsten L. Petersen, Kevin M.
  Pilgrim, Tan Lei and Michael Yampolsky.

\bibitem[BH88]{BrannerHubbardCubicI}
Bodil Branner and John~H. Hubbard.
\newblock The iteration of cubic polynomials. {P}art {I}: {T}he global topology
  of parameter space, the.
\newblock {\em Acta Math.}, 160(3-4):143--206, 1988.

\bibitem[BH92]{BrannerHubbardCubicII}
Bodil Branner and John~H. Hubbard.
\newblock The iteration of cubic polynomials. {P}art {II}: {P}atterns and
  parapatterns.
\newblock {\em Acta Math.}, 169(3-4):229--325, 1992.

\bibitem[BKM10]{AKMCubic}
Araceli Bonifant, Jan Kiwi, and John Milnor.
\newblock Cubic polynomial maps with periodic critical orbit. {II}. {E}scape
  regions.
\newblock {\em Conform. Geom. Dyn.}, 14:68--112, 2010.

\bibitem[BL14]{BuffTanDynatomic}
Xavier Buff and Tan Lei.
\newblock The quadratic dynatomic curves are smooth and irreducible.
\newblock In {\em Frontiers in complex dynamics}, volume~51 of {\em Princeton
  Math. Ser.}, pages 49--72. Princeton Univ. Press, Princeton, NJ, 2014.

\bibitem[BM]{cm3}
Araceli Bonifant and John Milnor.
\newblock Cubic polynomial maps with periodic critical orbit, part {III}:
  External rays.
\newblock In preparation.

\bibitem[BMS]{ParabolicGreen}
Araceli Bonifant, John Milnor, and Scott Sutherland.
\newblock Parabolic implosion and the relative green's function.
\newblock In preparation.

\bibitem[Bou92]{BouschThesis}
Thierry Bousch.
\newblock {\em Sur quelques probl\`emes de dynamique holomorphe}.
\newblock PhD thesis, Universit\'e de Paris-Sud, Centre d'Orsay, 1992.

\bibitem[Bra93]{BrannerMilnorFest60}
Bodil Branner.
\newblock Cubic polynomials: turning around the connectedness locus.
\newblock In {\em Topological methods in modern mathematics ({S}tony {B}rook,
  {NY}, 1991)}, pages 391--427. Publish or Perish, Houston, TX, 1993.

\bibitem[DH85]{OrsayNotes}
A.~Douady and J.~H. Hubbard.
\newblock {\em \'{E}tude dynamique des polyn\^omes complexes. {I,II}},
  volume~85 of {\em Publications Math\'ematiques d'Orsay [Mathematical
  Publications of Orsay]}.
\newblock Universit\'e de Paris-Sud, D\'epartement de Math\'ematiques, Orsay,
  1985.
\newblock With the collaboration of P. Lavaurs, Tan Lei and P. Sentenac.

\bibitem[DP17]{DeMarcoPilgrimClassification}
Laura DeMarco and Kevin Pilgrim.
\newblock The classification of polynomial basins of infinity.
\newblock {\em Ann. Sci. \'{E}c. Norm. Sup\'{e}r. (4)}, 50(4):799--877, 2017.

\bibitem[DS10]{DeMarcoSchiff}
Laura DeMarco and Aaron Schiff.
\newblock Enumerating the basins of infinity of cubic polynomials.
\newblock {\em J. Difference Equ. Appl.}, 16(5-6):451--461, 2010.

\bibitem[EMS16]{RationalMultibrot}
Dominik Eberlein, Sabyasachi Mukherjee, and Dierk Schleicher.
\newblock Rational parameter rays of the multibrot sets.
\newblock In {\em Dynamical systems, number theory and applications}, pages
  49--84. World Sci. Publ., Hackensack, NJ, 2016.

\bibitem[Fau92]{FaughtThesis}
D.~Faught.
\newblock {\em Local connectivity in a family of cubic polynomials}.
\newblock PhD thesis, Cornell University, 1992.

\bibitem[FG18]{FavreGauthierSpecialCubic}
Charles Favre and Thomas Gauthier.
\newblock Classification of special curves in the space of cubic polynomials.
\newblock {\em Int. Math. Res. Not. IMRN}, (2):362--411, 2018.

\bibitem[FKS16]{FirsovaKahnSelinger}
T.~{Firsova}, J.~{Kahn}, and N.~{Selinger}.
\newblock On deformation spaces of quadratic rational maps.
\newblock {\em ArXiv e-prints}, June 2016.

\bibitem[GM93]{GoldbergMilnor}
Lisa~R. Goldberg and John Milnor.
\newblock Fixed points of polynomial maps. ii. fixed point portraits.
\newblock {\em Ann. Sci. \'Ecole Norm. Sup. (4)}, 26(1):51--98, 1993.

\bibitem[GY18]{GhiocaYeSpecialCubic}
Dragos Ghioca and Hexi Ye.
\newblock A dynamical variant of the {A}ndr\'{e}-{O}ort conjecture.
\newblock {\em Int. Math. Res. Not. IMRN}, (8):2447--2480, 2018.

\bibitem[Ha{\"{\i}}00]{Haissinsky2000}
Peter Ha{\"{\i}}ssinsky.
\newblock D\'{e}formation {$J$}-\'{e}quivalente de polyn\^{o}mes
  g\'{e}ometriquement finis.
\newblock {\em Fund. Math.}, 163(2):131--141, 2000.

\bibitem[HK17]{HironakaKoch}
Eriko Hironaka and Sarah Koch.
\newblock A disconnected deformation space of rational maps.
\newblock {\em J. Mod. Dyn.}, 11:409--423, 2017.

\bibitem[Kiw97]{kiwi-1997}
Jan~Beno Kiwi.
\newblock {\em Rational rays and critical portraits of complex polynomials}.
\newblock ProQuest LLC, Ann Arbor, MI, 1997.
\newblock Thesis (Ph.D.)--State University of New York at Stony Brook.

\bibitem[Kiw02]{KiwiWOP}
Jan Kiwi.
\newblock Wandering orbit portraits.
\newblock {\em Trans. Amer. Math. Soc.}, 354(4):1473--1485 (electronic), 2002.

\bibitem[Kiw04]{kiwi04}
Jan Kiwi.
\newblock $\mathbb{R}$eal laminations and the topological dynamics of complex
  polynomials.
\newblock {\em Adv. Math.}, 184(2):207--267, 2004.

\bibitem[Lei00]{TanLeiTheme}
Tan Lei.
\newblock Local properties of the mandelbrot set at parabolic points.
\newblock In {\em Mandelbrot Set, Theme and Variations, The}, volume 274 of
  {\em London Math. Soc. Lecture Note Ser.}, pages 133--160. Cambridge Univ.
  Press, Cambridge, 2000.

\bibitem[LP96]{LevinPrzytyckiExternal}
G.~Levin and F.~Przytycki.
\newblock External rays to periodic points.
\newblock {\em Israel J. Math.}, 94:29--57, 1996.

\bibitem[LS]{LevinSodin}
Genadi Levin and M~Sodin.
\newblock Polynomials with disconnected julia sets and green maps.

\bibitem[Mil06]{DynInOne}
John Milnor.
\newblock {\em Dynamics in One Complex Variable}, volume 160 of {\em Annals of
  Mathematics Studies}.
\newblock Princeton University Press, Princeton, NJ, third edition, 2006.

\bibitem[Mil09]{MilnorPeriodicCubic}
John Milnor.
\newblock Cubic polynomial maps with periodic critical orbit. {I}.
\newblock In {\em Complex dynamics}, pages 333--411. A K Peters, Wellesley, MA,
  2009.

\bibitem[Mil12]{MilnorHyperbolicComponents}
John Milnor.
\newblock Hyperbolic components.
\newblock In {\em Conformal dynamics and hyperbolic geometry}, volume 573 of
  {\em Contemp. Math.}, pages 183--232. Amer. Math. Soc., Providence, RI, 2012.
\newblock With an appendix by A. Poirier.

\bibitem[Mor96]{MortonDynatomic}
Patrick Morton.
\newblock On certain algebraic curves related to polynomial maps.
\newblock {\em Compositio Math.}, 103(3):319--350, 1996.

\bibitem[Oud02]{Oudkerk}
Richard Oudkerk.
\newblock The parabolic implosion: {L}avaurs maps and strong convergence for
  rational maps.
\newblock In {\em Value distribution theory and complex dynamics ({H}ong
  {K}ong, 2000)}, volume 303 of {\em Contemp. Math.}, pages 79--105. Amer.
  Math. Soc., Providence, RI, 2002.

\bibitem[Poi09]{PoirierCP}
Alfredo Poirier.
\newblock Critical portraits for postcritically finite polynomials.
\newblock {\em Fund. Math.}, (203):107--163, 2009.

\bibitem[PR00]{PetersenRyd}
Carsten~Lunde Petersen and Gustav Ryd.
\newblock Convergence of rational rays in parameter spaces.
\newblock In {\em The {M}andelbrot set, theme and variations}, volume 274 of
  {\em London Math. Soc. Lecture Note Ser.}, pages 161--172. Cambridge Univ.
  Press, Cambridge, 2000.

\bibitem[Ree92]{ReesI}
M.~Rees.
\newblock A partial description of parameter space of rational maps of degree
  two. i.
\newblock {\em Acta Math.}, 168(1-2):11--87, 1992.

\bibitem[Ree95]{ReesII}
M.~Rees.
\newblock A partial description of the parameter space of rational maps of
  degree two. ii.
\newblock {\em Proc. London Math. Soc. (3)}, 70(3):644--690, 1995.

\bibitem[Ree03]{ReesAsterisque}
Mary Rees.
\newblock Views of parameter space: {T}opographer and {R}esident.
\newblock {\em Ast\'erisque}, (288):vi+418, 2003.

\bibitem[Roe07]{RoeschSone}
Pascale Roesch.
\newblock Hyperbolic components of polynomials with a fixed critical point of
  maximal order.
\newblock {\em Ann. Sci. \'Ecole Norm. Sup. (4)}, 40(6):901--949, 2007.

\bibitem[RS20]{ramadas2020quadratic}
Rohini Ramadas and Rob Silversmith.
\newblock Quadratic rational maps with a five-periodic critical point, 2020.

\bibitem[RY08]{RoeschYin}
Pascale Roesch and Yongcheng Yin.
\newblock The boundary of bounded polynomial {F}atou components.
\newblock {\em C. R. Math. Acad. Sci. Paris}, 346(15-16):877--880, 2008.

\bibitem[Sch94]{SchleicherPhD}
Dierk Schleicher.
\newblock {\em Internal addresses in the {M}andelbrot set and irreducibility of
  polynomials}.
\newblock PhD thesis, Cornell University, 1994.

\bibitem[Shi98]{ShishikuraHausdorff}
Mitsuhiro Shishikura.
\newblock Hausdorff dimension of the boundary of the mandelbrot set and julia
  sets, the.
\newblock {\em Ann. of Math. (2)}, 147(2):225--267, 1998.

\bibitem[Shi00]{ShishikuraTheme}
Mitsuhiro Shishikura.
\newblock Bifurcation of parabolic fixed points.
\newblock In {\em The {M}andelbrot set, theme and variations}, volume 274 of
  {\em London Math. Soc. Lecture Note Ser.}, pages 325--363. Cambridge Univ.
  Press, Cambridge, 2000.

\bibitem[Tim10]{TimorinRegluing}
V.~Timorin.
\newblock Topological regluing of rational functions.
\newblock {\em Invent. Math.}, 179(3):461--506, 2010.

\bibitem[Zak18]{ZakeriRotationSets}
Saeed Zakeri.
\newblock {\em Rotation sets and complex dynamics}, volume 2214 of {\em Lecture
  Notes in Mathematics}.
\newblock Springer, Cham, 2018.

\end{thebibliography}

\end{document}